\documentclass[a4paper,10pt,reqno]{amsart}

\usepackage[T1]{fontenc}
\usepackage[utf8]{inputenc} 
\usepackage{lmodern}
\usepackage[english]{babel}
\usepackage{amsmath,amsfonts,amssymb,amsthm,enumerate}

\usepackage{amssymb}

\numberwithin{equation}{section}

\usepackage{tikz}
\usetikzlibrary{arrows}
\pgfarrowsdeclare{<<<}{>>>}
{
\arrowsize=0.2pt
\advance\arrowsize by .5\pgflinewidth
\pgfarrowsleftextend{-4\arrowsize-.5\pgflinewidth}
\pgfarrowsrightextend{.5\pgflinewidth}
}
{
\arrowsize=0.2pt
\advance\arrowsize by .5\pgflinewidth
\pgfpathmoveto{\pgfpointorigin}
\pgfpathlineto{\pgfpoint{-1.2mm}{-.8mm}}
\pgfusepathqstroke
\pgfpathmoveto{\pgfpointorigin}
\pgfpathlineto{\pgfpoint{-1.2mm}{.8mm}}
\pgfusepathqstroke
}%%%% DEFINIZIONE FRECCIA CARINA:FINE (per usarla, iniziare una figura con \begin{tikzpicture}[>=>>>] )

\pgfarrowsdeclare{<<<<}{>>>>}
{
\arrowsize=0.2pt
\advance\arrowsize by .5\pgflinewidth
\pgfarrowsleftextend{-4\arrowsize-.5\pgflinewidth}
\pgfarrowsrightextend{.5\pgflinewidth}
}
{
\arrowsize=0.2pt
\advance\arrowsize by .5\pgflinewidth
\pgfpathmoveto{\pgfpointorigin}
\pgfpathlineto{\pgfpoint{-1.8mm}{-1.2mm}}
\pgfusepathqstroke
\pgfpathmoveto{\pgfpointorigin}
\pgfpathlineto{\pgfpoint{-1.8mm}{1.2mm}}
\pgfusepathqstroke
}%%%% DEFINIZIONE GROSSA:FINE (per usarla, iniziare una figura con \begin{tikzpicture}[>=>>>>] )

\newcounter{mt}
\def\maintheorem#1#2#3{\par \medskip \noindent {\bf Theorem~\mref{#1}}~(#2).~{\it #3}\par}
\def\mref#1{\Alph{#1}}
\def\maintheoremdeclaration#1{\stepcounter{mt}\newcounter{#1}\setcounter{#1}{\arabic{mt}}}
\maintheoremdeclaration{main}
\usepackage{caption}
\usepackage{url}

\def\N{\mathbb{N}}
\def\R{\mathbb{R}}
\def\Q{\mathcal{Q}}
\def\S{\mathbb{S}}
\def\P{\mathbb{P}}
\def\PP{\mathcal{P}}
\def\C{\mathcal{C}}
\def\Ceeb{C_{\rm per}}
\def\h{\mathfrak h}
\def\len{{\rm len}}
\def\Cgr{C_{\rm vol}}
\def\proofof#1{\begin{proof}[Proof of #1]}
\def\angle#1#2#3{#1\widehat{#2}#3}
\def\comp{\subset\subset}
\def\eps{\varepsilon}

\def\doppio#1#2{\stackrel{\hbox{\tiny $#1$}}{\hbox{\tiny $#2$}}}
\def\step#1#2{\par\vspace{5pt}\noindent{\underline{\it Step~#1.}}\emph{ #2}\\}

\def\E{\mathcal{E}}
\def\F{\mathcal F}
\def\Pe{P_{\rm Eucl}}
\def\Ve#1{|#1|_{\rm Eucl}}
\def\haus{\mathcal{H}}
\def\I{\mathcal{I}}

\theoremstyle{plain}
\newtheorem{thm}{Theorem}[section]
\newtheorem{lem}[thm]{Lemma}
\newtheorem{example}[thm]{Example}
\newtheorem{prop}[thm]{Proposition}
\newtheorem{cor}[thm]{Corollary}
\newtheorem{defn}[thm]{Definition}
\newtheorem{rmk}[thm]{Remark}

\allowdisplaybreaks

\usepackage{hyperref}
\usepackage{bookmark}
\title[Steiner property for planar clusters. Anisotropic case.]{On the Steiner property for planar minimizing clusters. The anisotropic case.}

\author[V. Franceschi]{Valentina Franceschi$^*$}
\address{$^*$Department of Mathematics, University of Padova}
\email{valentina.franceschi@unipd.it}

\author[A. Pratelli]{Aldo Pratelli$^\flat$}
\address{$^\flat$Department of Mathematics, University of Pisa}
\email{aldo.pratelli@unipi.it}

\author[G. Stefani]{Giorgio Stefani$^\sharp$}
\address{$^\sharp$Scuola Internazionale Superiore di Studi Avanzati (SISSA)}
\email{gstefani@sissa.it}

\date{\today}

\begin{document}

\begin{abstract}
In this paper we discuss the Steiner property for minimal clusters in the plane with an anisotropic double density. This means that we consider the classical isoperimetric problem for clusters, but volume and perimeter are defined by using two densities. In particular, the perimeter density may also depend on the direction of the normal vector. The classical ``Steiner property'' for the Euclidean case (which corresponds to both densities being equal to $1$) says that minimal clusters are made by finitely many ${\rm C}^{1,\gamma}$ arcs, meeting in finitely many ``triple points''. We can show that this property holds under very weak assumptions on the densities. In the parallel paper~\cite{FPS1} we consider the isotropic case, i.e., when the perimeter density does not depend on the direction, which makes most of the construction much simpler. In particular, in the present case the three arcs at triple points do not necessarily meet with three angles of $120^\circ$, which is instead what happens in the isotropic case.
\end{abstract}

\maketitle

\section{Introduction}

In recent years, a lot of attention is being paid to isoperimetric problems in $\R^N$ depending on two densities. This means that we are given two l.s.c. functions $g:\R^N\to (0,+\infty)$ and $h:\R^N\times\S^{N-1}\to (0,+\infty)$, usually called \emph{densities}, and the volume and perimeter of any set $E\subseteq\R^N$ of locally finite perimeter are defined by
\begin{align}\label{weightedvolper}
|E| = \int_E g(x)\,dx\,, && P(E) = \int_{\partial^* E} h(x,\nu_E(x))\, d\haus^1(x)\,,
\end{align}
where, as usual, $\partial^* E$ is the reduced boundary of $E$ and, for every $x\in\partial^* E$, $\nu_E(x)$ is the outer normal vector to $E$ at $x$ (see~\cite{AFP} for definitions and properties of sets of finite perimeter). There are several reasons why this problem is attracting a big interest, that we are not going to describe here, we limit ourselves to point out some basic bibliography, more information can be found there and in the references therein~\cite{RCBM08,CMV10,CRS12,MP13,BBCLT16,ABCMP17,DFP,G18,ABCMP19,CGPRS20,FP20}.\par

In this paper we consider the isoperimetric problem for clusters. In other words, we do not want to minimize the perimeter of a single set of given volume, but of a ``cluster'', that is, a group of sets with given volumes. This is not simply the ``sum'' of isoperimetric problems for single sets, because the common boundary is only counted once. A practical example of such a problem is given by soap bubbles, which behave more or less as minimal clusters with the Euclidean density. Of course a single bubble must be a ball; however, when there are two bubbles, the best situation is not given by two distinct balls, but by a cluster with the usual shape of two soap bubbles, which minimize the total perimeter by having a large common portion of the boundary. Also the problem of studying minimal clusters, in the Euclidean case, has been deeply investigated in the last decades, and completely solved for ``double bubbles'', i.e., when the cluster is made by two sets (see~\cite{FABHZ,HMRR,Rei}, and see also~\cite{W,PT0,PT1,PT2} for the case of three or four sets with equal volumes). In the planar case $N=2$ it is known that minimal clusters enjoy a strong regularity property. More precisely, the boundary of any minimal cluster is made by finitely many ${\rm C}^{1,\gamma}$ arcs (and then, by standard regularity, they are actually ${\rm C}^\infty$), which meet in finitely many junction points. Each of these junction points is actually a triple point --that is, exactly three arcs meet-- and the three arcs form three angles of $120^\circ$. This property, usually called \emph{Steiner property}, is now widely known. Two very good references are the classical paper~\cite{Tay} and the recent book~\cite{Mag}. We refer the reader also to~\cite{M94}, where the author studies existence and regularity results for minimal clusters in compact Riemannian surfaces, and to the recent papers~\cite{MN,MN2}, where the multiple bubble problem in the Gaussian space, in the Euclidean space, and on the sphere, are considered.\par

Our goal is to extend the regularity of minimal clusters in the plane to the case when perimeter and volume are given by two densities, and we are able to do this in a wide generality. In the parallel paper~\cite{FPS1} we consider the isotropic case, that is, when the density $h$ only depends on the point but not on the direction of the normal vector, and in that case the whole construction is rather simple. In the more general anisotropic case, that we consider here, the underlying idea is still simple, but several technical points become much more complicated. Moreover, the ``$120^\circ$ property'', which is still true in the isotropic case, becomes false. At this regard, see the discussion in Section~\ref{whichdir}. Considering the anisotropic case is important for several reasons, for instance to treat Riemannian surfaces. In this case, the local expression of the perimeter density is anisotropic, and in particular the ``$120^\circ$ property'', considering the angles in the Euclidean sense on local charts, does not always hold. More in general, our approach can be used to work with Riemannian or Finsler manifolds with density (see for instance~\cite{Morganbook}).\par

To consider the isoperimetric problem for clusters, the first thing to do is to extend the definition~(\ref{weightedvolper}) of volume and perimeter of a single set. For a given $m\geq 2$, a \emph{$m$-cluster} is a collection $\E=\{E_1,\,E_2,\, \dots\, ,\, E_m\}$ of $m$ essentially disjoint sets of locally finite perimeter in $\R^2$, and its \emph{volume} is the vector $|\E|=(|E_1|,\, |E_2|,\, \dots\,,\, |E_m|)\in (\R^+)^m$. We set, for brevity, $E_0=\R^2\setminus (\cup_{i=1}^m E_i)$ and $\partial^* \E=\cup_{i=1}^m \partial^* E_i$. The \emph{perimeter} of a cluster $\E$ is then defined as
\begin{equation}\label{PbE}
P(\E) = \frac{P(\cup_{i=1}^m E_i) + \sum_{i=1}^m P(E_i) }2 \,.
\end{equation}
It is very important to understand the meaning of this definition, which is discussed in detail in Section~\ref{sec:defper} below. Hence, a reader who sees this definition for the first time might want to read that section before going on with this introduction.\par

In order to present our main result, a few definitions are in order. We start with the \emph{strict convexity} and the \emph{uniform roundedness} of $h$ in the second variable.
\begin{defn}[Strict convexity and uniform roundedness in the second variable]\label{strictconv}
Let $h:\R^2\times\S^1\to (0,+\infty)$ be given, and extend it to the whole $\R^2\times\R^2$ by positive $1$-homogeneity, i.e., set $h(x,\lambda\nu)=\lambda h(x,\nu)$ for every $x\in\R^2,\,\lambda\geq 0$ and $\nu\in\S^1$. We say that \emph{$h$ is strictly convex in the second variable} if for every $x\in\R^2$ the unit ball
\begin{equation}\label{unitball}
\C(x) = \big\{ \nu\in\R^2:\, h(x,\nu)\leq 1\big\}
\end{equation}
is strictly convex. This is equivalent to ask that $h\big(x,t\nu+(1-t)\mu\big) < th(x,\nu) + (1-t) h(x,\mu)$ for every $x\in\R^2$ and $0<t<1$, and for every two non-zero vectors $\nu,\,\mu\in\R^2$ which are not positively parallel (i.e., $\nu=\lambda\mu$ for some $\lambda>0$). We say that \emph{$h$ is locally uniformly round in the second variable} if the infimum of the generalized curvatures of the balls $\C(x)$ with $x$ in any bounded set $\textsf{D}\subseteq\R^2$ is strictly positive. Formally speaking, there is a constant $c>0$ so that for every $x\in \textsf{D},\, \nu\in \S^1,\, w\in\R^2$ with $|w|\leq 1$ and $w\perp \nu$ one has
\begin{equation}\label{numero}
\frac{h(x,\nu+w) + h(x,\nu-w)}2 \geq h(x,\nu) + c |w|^2\,.
\end{equation}
\end{defn}

The uniform roundedness implies the strict convexity. Indeed, the strict convexity requires the unit balls $\C(x)$ to be strictly convex, so with strictly positive generalised curvature, while the uniform roundedness requires a (strictly positive) uniform bound from below on the curvature.\par

Let us briefly explain the role of these properties in our construction. First of all, the shortest path between two points close to each other can be far from the straight line if $h$ is not strictly convex, and the uniform roundedness is necessary to quantify this closeness (this is related with the so-called ``excess''). As a consequence, it is easy to guess that the regularity may fail without the uniform roundedness.\par

The fact that junction points are necessarily triple points, instead, also requires the regularity of the unit ball, or in other words the fact that $h\in {\rm C}^1$. More precisely, the content of Section~\ref{triplepoints} is to show that multiple points are necessarily triple points and the boundary of $\E$ is done by locally finitely many curves as soon as $h$ is ${\rm C}^1$ and strictly convex in the second variable (the uniform roundedness is not needed there). On the contrary, as discussed in detail in Section~\ref{caseL1}, quadruple points may occur for a density which is not ${\rm C}^1$, but uniformly round, hence also strictly convex. More precisely, we will first observe that quadruple points may occur for the $L^\infty$ density $h(x,\nu)=\max\{|\nu_1|,\,|\nu_2|\}$, which is not ${\rm C}^1$, but also not uniformly round nor strictly convex, and then we will show that the presence of quadruple points is still true with a simple modification of the $L^\infty$ density, which becomes uniformly round but remains not ${\rm C}^1$.\par

Summarizing, to obtain a Steiner property for minimal clusters (that is, $\partial\E$ is done by regular arcs meeting in triple points, see Definition~\ref{defstp}) one has to assume that $h$ is ${\rm C}^1$ and uniformly round in the second variable. Our main result, Theorem~\mref{main}, says that under these assumptions, and together with the same $\eps-\eps^\beta$ property and volume growth condition as in the isotropic case, it is still true that the Steiner property holds.

\begin{defn}[$\eta$-growth condition and $\eps-\eps^\beta$ property for clusters]\label{defge}
Given a power $\eta\ge 1$, an \emph{$\eta$-growth condition} is said to hold if there exist $\Cgr>0,\,R_\eta>0$ such that, for every $x\in\R^2$ and every $r<R_\eta$, the ball $B(x,r)$ has volume $|B(x,r)|\leq \Cgr r^\eta$. We say that the \emph{local $\eta$-growth condition} holds if for any bounded domain $\textsf{D}\comp \R^2$ there exist $\Cgr>0,\,R_\eta>0$ such that the above property holds for balls $B(x,r)\subseteq \textsf{D}$.\par
We say that a cluster $\E$ satisfies the $\eps-\eps^\beta$ property for some $0<\beta\leq 1$ if there exist $R_\beta>0,\,\Ceeb>0,\,\bar\eps>0$ such that, for every vector $\eps\in\R^m$ with Euclidean norm $|\eps|\leq \bar\eps$ and every $x\in\R^2$, there exists another cluster $\F$ such that
\begin{align}\label{propeeb}
\F\Delta \E \subseteq \R^2\setminus B(x,R_\beta)\,, &&
|\F|=|\E|+\eps\,, &&
P(\F)\leq P(\E) + \Ceeb |\eps|^\beta\,.
\end{align}
If this holds then, for each $t\leq \bar\eps$, we call $\Ceeb[t]$ the smallest constant such that the above property is true for every $|\eps|\leq t$. Clearly $t\mapsto \Ceeb[t]$ is an increasing function, and $\Ceeb[\bar\eps]\leq \Ceeb$.
\end{defn}

We underline that both the above assumptions are satisfied for a wide class of densities. In particular, the growth (or local growth) condition clearly holds with $\eta=2$ whenever the density $g$ is bounded (or locally bounded). Concerning the $\eps-\eps^\beta$ property, this is a crucial tool when dealing with isoperimetric problems. It is simple to observe that it is valid with $\beta=1$ for every cluster of locally finite perimeter whenever the density $h$ is regular enough (at least Lipschitz) in the first variable. It is also known that, if $h$ is $\alpha$-H\"older in the first variable, then every cluster of locally finite perimeter satisfies the $\eps-\eps^\beta$ property with
\[
\beta=\frac 1{2-\alpha}\,,
\]
the proof can be found in~\cite{CP} for the special case $g=h$ and in~\cite{PS19} for the general case. The case $\alpha=0$ is particular, also because there is not a unique possible meaning of ``$0$-H\"older function''. More precisely, the $\eps-\eps^{1/2}$ property holds as soon as $h$ is locally bounded. If $h$ is continuous, instead, not only the $\eps-\eps^{1/2}$ property holds, but in addition $\Ceeb[t]\searrow 0$ if $t\searrow 0$. We can be even more precise: $\Ceeb[t]\lesssim \sqrt{\omega_h(\sqrt t)}$, being $\omega_h$ the modulus of continuity of $h$ in the first variable (see~\cite{PS2}). Notice that the required regularity for $h$ here is in the first variable. In particular, $\omega_h$ is defined as $\omega_h(t)=\sup \{|h(x,\nu)-h(y,\nu)|:\, \nu\in\S^1,\, |y-x|\leq t\}$. \par

We can now give the formal definition of the Steiner property, already described above, and of the Dini property.

\begin{defn}[Steiner property]\label{defstp}
A cluster $\E$ is said to satisfy the \emph{Steiner property} if $\partial\E$ is a locally finite union of ${\rm C}^1$ arcs, and each junction point is endpoint of exactly three different arcs, arriving with three different tangent vectors.
\end{defn}

\begin{defn}[Dini property]
We say that an increasing function $\varphi:\R^+\to\R^+$ \emph{satisfies the Dini property} if for every $C>1$ one has
\[
\sum_{n\in\N} \varphi(C^{-n}) < +\infty\,,
\]
which in particular implies $\lim_{t\searrow 0} \varphi(t)=0$. We say that \emph{$\varphi$ satisfies the $1/2$-Dini property} if $\sqrt{\varphi}$ satisfies the Dini property. A uniformly continuous function $f$ is said \emph{Dini continuous} whenever
\[
\int_0^1\frac{\omega_f(t)}t\, dt<+\infty\,,
\]
where $\omega_f$ is the modulus of continuity of $f$. It is known that $f$ is Dini continuous if and only if $\omega_f$ satisfies the Dini property. We say that \emph{$f$ is $1/2$-Dini continuous} if $\omega_f$ satisfies the $1/2$-Dini property.
\end{defn}

Notice that if $\varphi:\R^+\to\R^+$ is non-decreasing, such that $\varphi(0)=0$ and $\alpha$-H\"older, then it satisfies the Dini property. We are now in position to state the main result of the present paper.

\maintheorem{main}{Steiner regularity for minimal clusters}{Let $g:\R^2\to (0,+\infty)$ be a l.s.c. function, and let $h:\R^2\times\S^1\to (0,+\infty)$ be a continuous function, which is ${\rm C}^1$ and uniformly round in the second variable in the sense of Definition~\ref{strictconv}. Let $\E$ be a minimal cluster, and assume that for some $\eta,\,\beta$ the local $\eta$-growth condition holds, as well as the $\eps-\eps^\beta$ property for $\E$. Assume also that $h$ is locally $1/2$-Dini continuous in the first variable, and that either
\begin{enumerate}[(i)]
\item $\eta \beta>1$, or
\item $\eta\beta=1$ and the function $t \mapsto \Ceeb[t]$ satisfies the $1/2$-Dini property.
\end{enumerate}
Then $\E$ satisfies the Steiner property. Moreover, if $\eta\beta>1$ and $h$ is locally $\alpha$-H\"older in the first variable then the arcs of $\partial\E$ are actually ${\rm C}^{1,\gamma}$ with $\gamma=\frac 12\, \min\{\eta\beta-1,\, \alpha\}$.}

It is to be observed that this result strongly generalizes the classical Euclidean case. In fact, we require that $\eta\beta\geq 1$, while in the Euclidean case one has $\eta=2$ and $\beta=1$. This result also extends the isotropic case considered in~\cite{FPS1}, but the proof there is considerably simpler. \par

Concerning the $1/2$-Dini property, it is a standard assumption to get the ${\rm C}^1$ regularity of the boundary, see for instance~\cite{Tam,Mag}. Notice that one can always apply Theorem~\mref{main} if $g$ is locally bounded and $h$ is locally $1/4$-Dini continuous in the first variable (i.e., $\sqrt[4]{\omega_h}$ satisfies the Dini property), since in this case $\eta=2$ and $\beta=1/2$, and the required continuity of $h$ and $\Ceeb$ follows by the fact that $\Ceeb\lesssim \sqrt{\omega_h}$, already observed above.\par

Finally, we remark that under quite mild assumptions (which broadly cover the Euclidean case) the boundedness of the minimal clusters is known, see for instance~\cite{CP,DFP,PS18,PS19}. Of course, whenever optimal clusters are bounded, the arcs given by Theorem~\mref{main} become finite and not just locally finite.

Section~\ref{sec:proof} is dedicated to the proof of Theorem~\mref{main}, whose scheme is summarized in Section~\ref{sec:exproof} below. In Section~\ref{s:final}, we present two important final comments. The first one is about the role of the assumptions of Theorem~\mref{main} on the Steiner regularity of minimal clusters. The uniform roundedness in the second variable is necessary to obtain the regularity of the free boundary. Counterexamples without the strict convexity are trivially found as Wulff shapes for unit balls with flat sides, and the role of the uniform roundedness is clear by the argument in Section~\ref{s:interface}. It is less obvious to understand the role of the ${\rm C}^1$-regularity of $h$, but in Section~\ref{caseL1} we show that it is indeed crucial to prove that multiple junctions can only be triple points, by providing an explicit example of a uniformly round density $h$ which is not ${\rm C}^1$ allowing for minimal clusters with quadruple junctions.

In Section~\ref{whichdir} we discuss which are the admissible directions for the tangents of $\partial^*\mathcal E$ at a triple point and we establish some necessary minimality conditions for optimal triples. In particular, given one of the directions at a triple point, in the symmetric case it is always possible to uniquely determine the pair of the other two directions, while in the asymmetric case there could be no such pair, or more than one, even infinitely many.

We refer the reader to the parallel work \cite[Section~3]{FPS1} for some relevant applications of Theorem~\mref{main}, besides Riemannian manifolds.

\subsection{Meaning of the definition of perimeter for clusters\label{sec:defper}}

This short section is devoted to discuss the definition~(\ref{PbE}) of the perimeter for clusters, which might be a bit obscure at first sight. We do this with the aid of the example in Figure~\ref{Figcluster}, where a $3$-cluster is shown.
\begin{figure}[htbp]
\begin{tikzpicture}[>=>>>,smooth cycle]
\filldraw[fill=red!15!white, draw=black, line width=.8pt] plot [tension=1] coordinates {(0.5,-2) (-0.5,-1) (3,0) (4,-2)};
\filldraw[fill=yellow!15!white, draw=black, line width=.8pt] plot [tension=1] coordinates {(0,0) (1,1) (3,0) (2,-1.5)};
\filldraw[fill=green!15!white, draw=black, line width=.8pt] plot [tension=1] coordinates {(2,1) (3,1.5) (5,0) (4,-1)};
\draw (.85,-1.5) node[anchor=east] {$E_1$};
\draw (1.9,-0.2) node[anchor=east] {$E_2$};
\draw (4,0.3) node[anchor=east] {$E_3$};
\draw[->] (4,1.05) -- (4.46,1.7);
\draw[->] (0,.65) -- (-0.62,1.16);
\draw[->] (-0.83,-1.44) -- (-1.63,-1.44);
\draw[<->] (2.15,-0.03) -- (2.81,.43);
\draw[<->] (2.02,-1.9) -- (2.02,-1.1);
\draw[<->] (3.44,-1.18) -- (3.86,-0.5);
\end{tikzpicture}
\caption{Example of a $3$-cluster.}\label{Figcluster}
\end{figure}
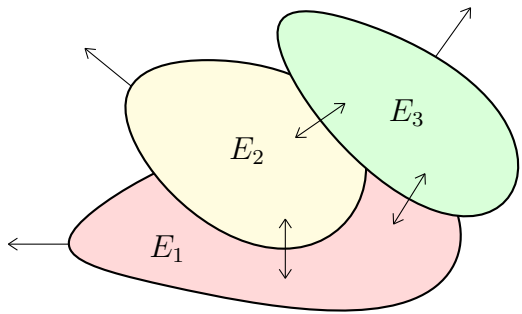
Observe that $\haus^1$-almost every $x\in \partial^*\E$ belongs to either a single one of the boundaries $\partial^* E_i$, $1\leq i\leq m$, or to two ones, and the reduced boundary of $\cup_{i=1}^m E_i$ is done exactly by the points of $\partial^*\E$ which belong to a single one. So, each point of $\partial^* \E$ is counted twice in the expression~(\ref{PbE}), and this explains the reason of the factor $1/2$. Concerning a direction of the normal vectors, for a point $x$ which belongs to a single boundary $\partial^* E_i$ the normal vector at $x$ to $\partial^* E_i$ and to $\cup_{i=1}^m E_i$ is the same, say $\nu(x)$, and then in the expression~(\ref{PbE}) the infinitesimal perimeter of the cluster produced by the point $x$ is simply $h(x,\nu(x))$. In a sense, $\nu(x)$ is ``the normal vector to $\E$ at $x$''. For a point $x$ which belongs to $\partial^* E_i$ and to $\partial^* E_j$ for two different indices $1\leq i,\,j\leq m$, instead, the two normal vectors at $x$ to $E_i$ and to $E_j$ are opposite, so there is a vector $\nu(x)$ such that the contribution of $x$ to the perimeter of the cluster is $\frac 12 h(x,\nu(x))+ \frac 12 h(x,-\nu(x))$. This also explains why in Figure~\ref{Figcluster} at some points of $\partial^*\E$ an arrow of length $1$ is attached, and at other points two opposite arrows of length $1/2$.\par

A couple of final comments are now in order. First of all, in the isotropic case, i.e., when $h$ only depends on $x$, then the contribution of every point $x\in\partial^*\E$ is simply $h(x)$, hence~(\ref{PbE}) can be rewritten in the much simpler form
\[
P(\E) = \int_{\partial^*\E} h(x)\, d\haus^1(x)\,.
\]
There is also an intermediate situation, namely, if the perimeter density is anisotropic but symmetric, that is, $h(x,\nu)=h(x,-\nu)$ for every $x\in\R^2,\,\nu\in\S^1$. Also in this case it is possible to express the perimeter of the cluster $\E$ in a much simpler way than~(\ref{PbE}), that is,
\[
P(\E) = \int_{\partial^*\E} h(x,\nu(x))\, d\haus^1(x)\,,
\]
where the vector $\nu(x)$ is defined as above for every $x\in\partial^*\E$.\par

We remark that the anisotropic but symmetric case is only slightly more complicate to treat than the isotropic case, and most of the difficulties of the case considered in this paper are due to the asymmetry. Roughly speaking, the big issue in the non symmetric case is that the set $E_0$ behaves in a different way than the sets $E_i$ with $1\leq i \leq m$, while in the symmetric case there is locally no difference between the different sets.\par

We also underline that some authors suggest, for the perimeter of a cluster in the case with densities, to use the definition $P(\E) = \frac 12 \sum_{i=0}^m P(E_i)$ in place of~(\ref{PbE}). This is of course a possible choice. However, in this case the contribution to the perimeter of the cluster given by any point $x\in\partial^*\E$ is always given by $\frac 12 h(x,\nu(x))+ \frac 12 h(x,-\nu(x))$, regardless whether or not $x$ belongs to $\partial^* E_0 = \partial^* \big(\cup_{i=1}^m E_i\big)$. Hence, the problem does not change at all by replacing $h$ with the density $\tilde h(x,\nu)=(h(x,\nu)+h(x,-\nu))/2$. In other words, with this choice one only has to consider the symmetric case, and as said above this would require a considerably simpler proof for our main result.

\subsection{Scheme of the proof\label{sec:exproof}}

The core of the proof is to show that, under the assumptions of Theorem~\mref{main}, minimal clusters have (locally) finitely many triple junctions.
Once this is proved, the argument to obtain either the ${\rm C}^1$ or the ${\rm C}^{1,\gamma}$ regularity of the free boundary is standard and relies on the uniform roundedness in the second variable, and on the $1/2$-Dini continuity of $h$ in the first variable and of $t\mapsto \Ceeb[t]$ 
(if $\eta\beta=1$), see Section~\ref{s:interface}. 

In order to show that multiple junctions of minimal clusters are locally finite and triple, we start by considering the special situation where $\partial^*\E$ consists of more than three radii inside a small ball. We prove that, in this case, the cluster can be modified so to decrease the perimeter of a quantity which is proportional to the radius of the ball. This is the content of Section~\ref{sect90}, where we use the strict convexity and the ${\rm C}^1$-regularity of $h$ in the second variable. We stress that, while in the isotropic case this boils down to a trivial trigonometric estimate, the anisotropic case is quite delicate, especially if $h$ is not symmetric. In fact, Section~\ref{sect90} deals with a density $h(x,\nu)$ only depending on $\nu$. However, several local arguments can be proved first assuming that $h(x,\nu)$ only depends on $\nu$, and then reaching the general case via the continuity of $h$ in the first variable.

The second observation is contained in Lemma~\ref{13/2}, where we show that the (Euclidean) length of the boundary of a minimal cluster inside a small ball is controlled by a constant multiple of its radius. As in the isotropic case, this is a simple consequence of the validity of the $\varepsilon-\varepsilon^\beta$ property and of the $\eta$-growth condition, summarized together in Lemma~\ref{labello}. 

Given a minimal cluster $\mathcal E=\{E_1,\dots,E_m\}$, we will call ``colored'' the chambers $E_i$ for $i=1,\, \dots\,,\, m$ and ``white'' the external chamber $E_0=\R^2\setminus (\bigcup_{i=1}^mE_i)$. The next step is then to prove a first mild regularity property of each colored region, namely, the intersection of $E_i$ with a small ball is an open set whose boundary is a Jordan curve. As a consequence, the reduced boundary of the cluster inside the ball coincides with the topological one, see Lemma~\ref{lem1reg}. We stress that the same regularity holds also for the white region, but this will be proved much later.

To achieve the above regularity property, we first show that there can be no ``islands'' (i.e., isolated portions of colored regions) in small balls. More precisely, if a colored region $E_i$ intersects a small ball, then it must also intersect the boundary of the ball. While in the isotropic and in the symmetric cases this ``no-island property'' follows as a rather simple consequence of the $\eps-\eps^\beta$ property and of the $\eta$-growth conditions (and the argument applies also to the white chamber), in the asymmetric case the proof is much more delicate. The analogous of the no-island property for the white chamber, that we call ``no-lake property'', is again true, but at this stage we only establish a very simple version of it. At this point, we can finally prove the mild regularity of Lemma~\ref{lem1reg}. This requires playing with the notions of quasi-minimality and porosity and using deep properties of the reduced boundary.

We then consider the Jordan curves obtained in Lemma~\ref{lem1reg} and prove that, thanks to the strict convexity of $h$, if two such curves have two points in common, then they must share a common subcurve. The proof of this fact is quite involved, especially in the asymmetric case, and for future purposes we actually need a quantitative version of it, see Lemma~\ref{noncrazyint}. %We are then able to deduce some mild regularity properties of the boundary of the white chamber $E_0$. 

We are now in a position to show that $\partial^*\mathcal E$ intersects the boundary of a ball $B(x,\rho)$ in at most $3$ points for many small radii $\rho$. This is the most complex step of the proof, since at this point we have to rely on all the above information and make an \emph{ad hoc} construction, see Lemma~\ref{atmost3}. Having this result at disposal, we deduce the full version of the no-lake property, see Lemma~\ref{nolakesgen}.

We are finally ready to conclude. As desired, we deduce that all multiple points are triple points and that they are a positive distance apart from each other. This is the content of Lemma~\ref{dist>R6} and easily follows by using Lemma~\ref{atmost3}, the no-island and no-lake properties, and exploiting the Jordan curves found above.

\begin{rmk}
Throughout the argument used to prove that multiple junctions are locally finite triple points, we use that $h$ is positive, locally bounded, ${\rm C}^1$ and strictly convex in the second variable, as well as the validity of the $\varepsilon-\varepsilon^\beta$ property and of the $\eta$-growth conditions used to construct competitors. 
On the other hand, the $1/2$-Dini continuity of $h$ (and of $t\mapsto \Ceeb[t]$) and the uniform roundedness of $h$ are only needed to prove regularity of the free boundary.
\end{rmk}

\section{Proof of the main result\label{sec:proof}}

The proof of the main result, Theorem~\mref{main}, is presented in this section. In turn, this is subdivided in five subsections. While the first one collects some standard definitions and technical tools, in the second one we present the basic geometric estimate from which the fact that junction points are necessarily triple points follows. This estimate, which is trivial in the isotropic case, follows by convexity via a suitable first order expansion in general. The third subsection is devoted to show that there are (locally) finitely many junction points, each of which where exactly three different sets meet, and in the fourth one we obtain the regularity. The actual proof of the theorem, presented in the last subsection, basically only consists in putting the different parts together.

Since we aim to prove Theorem~\mref{main}, from now on we assume that $h$ is continuous and that the local $\eta$-growth condition holds for some $\eta\geq 1$. Moreover, we assume that $\E$ is a minimal cluster, for which the $\eps-\eps^\beta$ property holds, and such that either assumption~(i) or~(ii) of Theorem~\mref{main} holds.

\subsection{Some definitions and technical tools}

Let us fix some notation, that will be used through the rest of the paper. Since we are interested in a local property, in the proof of Theorem~\mref{main} we will immediately start by fixing a big closed ball $\textsf{D}\subseteq\R^2$, and the whole construction will be performed there. Hence, all the following definitions will depend upon $\textsf{D}$, in particular we assume that $|B(x,r)|\leq \Cgr r^\eta$ for every ball $B(x,r)$ with $x\in \textsf{D}$ and $r\leq {\rm diam}(\textsf{D})$.\par

Since $h$ is continuous, we can call $\omega:\R^+\to\R^+$ its modulus of continuity in the first variable inside $\textsf{D}$, that is,
\[
\omega(t) := \sup \Big\{ \big|h(x,\nu)-h(y,\nu)\big|:\, \nu\in\S^1,\, x,\,y \in \textsf{D},\, |y-x|\leq t\Big\}\,.
\]
In particular, if $h$ is locally $\alpha$-H\"older in the first variable, then $\omega(r) \leq C r^\alpha$ for a suitable constant $C$. Moreover, we will call $0<h_{\rm min}\leq h_{\rm max}$ the maximum and the minimum of $h$ in $\textsf{D}\times\S^1$.\par

Observe that, if assumption~(ii) of Theorem~\mref{main} is satisfied, then in particular one can choose the constant $\Ceeb$ to be as small as desired, up to decrease the value of $\bar\eps$. As a consequence, we assume that the constant $\bar\eps$ of the $\eps-\eps^\beta$ property is so small that
\begin{equation}\label{Cperissmall}
\Ceeb < \min \Big\{ \Ceeb^1,\, \Ceeb^2,\, \Ceeb^3,\, \Ceeb^4,\, \Ceeb^5\Big\} \qquad \hbox{if $\eta\beta=1$}\,,
\end{equation}
where all the constants $\Ceeb^i$ depend solely on $h$ and on $\textsf{D}$, and are defined in formulas~(\ref{ceeb1}), (\ref{ceeb2}), (\ref{ceeb3}), (\ref{ceeb4}) and~(\ref{ceeb5}) respectively.

\begin{lem}[Isoperimetric inequality with exponent]\label{disgusting}
For every measurable set $E\subseteq \textsf{D}$ we have
\[
P(E) \geq \frac{h_{\rm min}}{\Cgr^{1/\eta}}\, |E|^{1/\eta}\,.
\]
\end{lem}
\begin{proof}
By standard approximation, it is enough to prove the inequality for a smooth set $E\subseteq \textsf{D}$, and we can call $E_i$ its connected components. For every $i$, we take an arbitrary point $x_i\in E_i$ and call $r_i$ the diameter of $E_i$. We have then
\begin{align*}
E\subseteq \bigcup_i B(x_i,r_i)\,, && \haus^1(\partial E)=\sum_i \haus^1(\partial E_i) \geq \sum_i r_i\,.
\end{align*}
By construction, for every $i$ we have that $x_i\in \textsf{D}$ and $r_i\leq {\rm diam}(\textsf{D})$. Then, keeping in mind that $\eta\geq 1$, for any $E\subseteq \textsf{D}$ we deduce
\[\begin{split}
P(E)&\geq h_{\rm min} \haus^1(\partial E)
\geq  h_{\rm min} \sum\nolimits_i r_i
\geq \frac{h_{\rm min}}{\Cgr^{1/\eta}}\, \sum\nolimits_i |B(x_i,r_i)|^{1/\eta}\\
&\geq \frac{h_{\rm min}}{\Cgr^{1/\eta}}\, \Big(\sum\nolimits_i |B(x_i,r_i)|\Big)^{1/\eta}
\geq \frac{h_{\rm min}}{\Cgr^{1/\eta}}\, |E|^{1/\eta}\,,
\end{split}\]
so the proof is concluded.
\end{proof}

We introduce now the (standard) notation of relative perimeter. Given a set $E\subseteq\R^2$ of locally finite perimeter, or a cluster $\E$, and given a Borel set $A\subseteq \R^2$, the \emph{relative perimeter of $E$ (or $\E$) inside $A$} is the measure of the boundary of $E$ (or $\E$) within $A$, i.e.,
\begin{align*}
P(E;A) = \int_{A\cap\partial^* E} h(x,\nu_E(x))\, d\haus^1(x)\,, &&
P(\E;A) = \frac{P(\cup_{i=1}^m E_i;A) + \sum_{i=1}^m P(E_i;A) }2 \,,
\end{align*}
compare with~(\ref{PbE}).\par

We conclude this short section by presenting (a very specific case of) a fundamental result due to Vol'pert, see~\cite{Volpert} and also~\cite[Theorem~3.108]{AFP}).
\begin{thm}[Vol'pert]\label{volpert}
Let $E\subseteq\R^2$ be a set of locally finite perimeter, and let $x\in\R^2$ be fixed. Then, for a.e.\ $r>0$, one has that
\[
\partial^* E \cap \partial B(x,r) = \partial^* \big( E \cap \partial B(x,r)\big)\,.
\]
\end{thm}

Notice that, for almost every $r>0$, both sets in the above equality are done by finitely many points. In particular, $E\cap\partial B(x,r)$ is a subset of the circle $\partial B(x,r)$, and its boundary has to be considered in the $1$-dimensional sense. More precisely, for almost every $r>0$ the set $E\cap \partial B(x,r)$ essentially consists of a finite union of arcs of the circle, and the intersection of $\partial^* E$ with the circle is simply the union of the endpoints of all of them. Through the rest of the paper, we will often consider intersections of sets with balls. Even if this will not be repeated every time, we will always consider balls for which Vol'pert Theorem holds true.

\subsection{The $90^\circ$ property\label{sect90}}

This section is devoted to present a geometric estimate, which is the main reason why junction points are triple points. Let us consider for a moment the Euclidean perimeter, let $A,\,O,\,B$ be three points in $\R^2$, and let us assume that $\angle AOB$ is the greatest angle of the triangle $AOB$. A simple trigonometric computation ensures that the shortest connected set containing the three points is the union of the segments $BO$ and $OA$ if $\angle AOB$ is larger than $120^\circ$, while otherwise it is the union of the three segments $AP,\, BP$ and $OP$, being $P$ the unique point of the triangle $AOB$ such that the angles $\angle OPB,\, \angle BPA$ and $\angle APO$ are all $120^\circ$, see Figure~\ref{Fig120}. We can call this the ``$120^\circ$ property''. As a simple consequence, once one knows that the boundary of a minimal cluster is done by ${\rm C}^1$ arcs, it follows that all the junction points must be points where the different arcs meet with angles of $120^\circ$, in particular they must be triple points (i.e., three arcs meet).\par
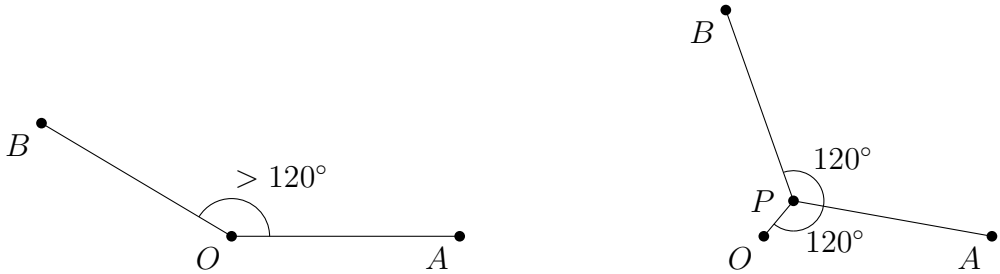
\begin{figure}[htbp]
\begin{tikzpicture}
\fill (1,0) circle (2pt);
\draw (1,0) node[anchor=north east] {$O$};
\fill (4,0) circle (2pt);
\draw (4,0) node[anchor=north east] {$A$};
\fill (-1.5,1.5) circle (2pt);
\draw (-1.5,1.5) node[anchor=north east] {$B$};
\draw (-1.5,1.5) -- (1,0) -- (4,0);
\draw (1.5,0) arc (0:149:.5);
\draw (0.9,0.5) node[anchor=south west] {$>120^\circ$};
\fill (7,0) circle (2pt);
\draw (7,0) node[anchor=north east] {$O$};
\fill (10,0) circle (2pt);
\draw (10,0) node[anchor=north east] {$A$};
\fill (6.5,3) circle (2pt);
\draw (6.5,3) node[anchor=north east] {$B$};
\fill (7.39,.47) circle (2pt);
\draw (7.29,.47) node[anchor=east] {$P$};
\draw (7.39,.47) -- (7,0);
\draw (7.39,.47) -- (10,0);
\draw (7.39,.47) -- (6.5,3);
\draw (7.39,.07) arc (-90:108:.4);
\draw (7.39,.07) arc (-90:-130:.4);
\draw (7.5,0.75) node[anchor=south west] {$120^\circ$};
\draw (7.4,-0.35) node[anchor=south west] {$120^\circ$};
\end{tikzpicture}
\caption{The ``$120^\circ$ property''. The shortest connected set containing three points $A,\, O,\,B$ in two cases.}\label{Fig120}
\end{figure}
Let us now pass to consider a general (strictly convex and ${\rm C}^1$) density for the perimeter, only depending on the direction, so that the unit ball $\C=\C(x)$ defined in~(\ref{unitball}) is the same for every $x\in\R^2$. As soon as $\C$ is not a Euclidean ball, the $120^\circ$ property easily fails. Nevertheless, in order to understand whether or not the shortest connected set containing the three points $A,\, O$ and $B$ is the union of the segments $BO$ and $OA$, there is still an interesting angle. Namely, the angle between the direction of $OA$ and the tangent direction to $\partial\C$ at the point in direction of $OB$ (the direction $\nu$ in Figure~\ref{Figstep3}). It can be shown that this angle must be at least $90^\circ$, see Step~III in the proof of Proposition~\ref{90prop}. Roughly speaking, this is enough to rule out quadruple points, because they should correspond to four angles of exactly $90^\circ$, and in turn this is impossible by the \emph{strict} convexity of the norm. The situation is not really so simple, but this is somehow the underlying idea.\par

Before giving the claim of the property, it is convenient to recall that we are considering perimeter densities which are not necessarily symmetric (the result below in the case of a symmetric density is much simpler to prove). As discussed in Section~\ref{sec:defper}, a consequence of this is that we cannot speak of length of segments, but of length of \emph{oriented} segments, since in general $h(\nu)\neq h(-\nu)$. A simple example can help to further clarify this point. Consider the cluster depicted in Figure~\ref{FigPerAs}, and let us compute its perimeter inside the ball. This perimeter is given by the sum of the lengths of the six radii from $OA$ to $OF$, however these lengths have to be computed carefully if $h$ is not symmetric. The segment $OA$, for instance, belongs to the boundary of $E_0$ and $E_1$, so its contribution to the perimeter is $h(\widehat{OA})$, where $\widehat{OA}$ is the \emph{oriented} segment obtained by rotating the \emph{oriented} segment $OA$ clockwise of $90^\circ$. Just for simplicity of notations, as we will do in the sequel, we write $\h(OA)$ in place of $h(\widehat{OA})$. The segment $OB$, instead, belongs to the boundary of $E_1$ and $E_3$, thus its contribution is $(\h(OB)+\h(BO))/2$. Arguing analogously for the other segments, we readily conclude that the perimeter of the cluster inside the ball is given by
\[
\h(OA) + \frac{\h(OB)+\h(BO)}2 + \frac{\h(OC)+\h(CO)}2 + \h(DO) + \h(OE) + \h(FO)\,.
\]
We are now in a position to state and prove the ``$90^\circ$ property''.
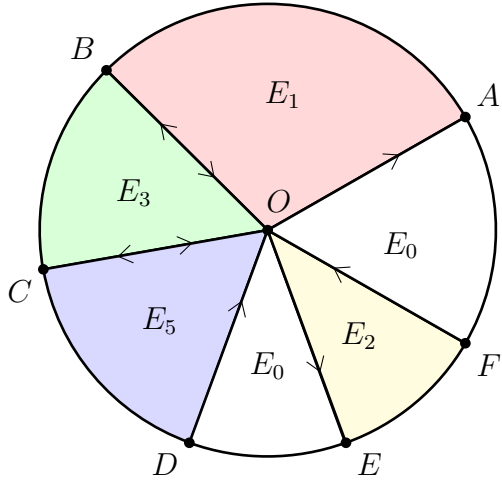
\begin{figure}[htbp]
\begin{tikzpicture}[>=>>>>]
\filldraw[fill=red!15!white, draw=black, line width=1pt] (0,0) -- (2.6,1.5) arc(30:135:3) -- cycle;
\filldraw[fill=green!15!white, draw=black, line width=1pt] (0,0) -- (-2.12,2.12) arc(135:190:3) -- cycle;
\filldraw[fill=blue!15!white, draw=black, line width=1pt] (0,0) -- (-2.95,-.52) arc(190:250:3) -- cycle;
\filldraw[fill=white, draw=black, line width=1pt] (0,0) -- (-1.03,-2.82) arc(250:290:3) -- cycle;
\filldraw[fill=yellow!15!white, draw=black, line width=1pt] (0,0) -- (1.03,-2.82) arc(290:330:3) -- cycle;
\filldraw[fill=white, draw=black, line width=1pt] (0,0) -- (2.6,-1.5) arc(-30:30:3) -- cycle;
\fill (2.6,1.5) circle (2pt);
\draw (2.6,1.5) node[anchor=south west] {$A$};
\fill (-2.12,2.12) circle (2pt);
\draw (-2.12,2.12) node[anchor=south east] {$B$};
\fill (-2.95,-0.52) circle (2pt);
\draw (-2.95,-0.52) node[anchor=north east] {$C$};
\fill (-1.03,-2.82) circle (2pt);
\draw (-1.03,-2.82) node[anchor=north east] {$D$};
\fill (1.03,-2.82) circle (2pt);
\draw (1.03,-2.82) node[anchor=north west] {$E$};
\fill (2.6,-1.5) circle (2pt);
\draw (2.6,-1.5) node[anchor=north west] {$F$};
\fill (0,0) circle (2pt);
\draw (0.15,0.1) node[anchor=south] {$O$};
\draw[->] (0.866,.5) -- (1.733,1);
\draw[<->] (-0.7,0.7) -- (-1.41,1.41);
\draw[<->] (-0.99,-.17) -- (-1.97,-.35);
\draw[<-] (-0.34,-0.94) -- (-.69,-1.88);
\draw[->] (0.34,-0.94) -- (.69,-1.88);
\draw[<-] (0.866,-0.5) -- (1.733,-1);
\draw (0.2,1.8) node {$E_1$};
\draw (-1.76,.5) node {$E_3$};
\draw (-1.42,-1.2) node {$E_5$};
\draw (0,-1.8) node {$E_0$};
\draw (1.2,-1.4) node {$E_2$};
\draw (1.75,-.2) node {$E_0$};
\end{tikzpicture}
\caption{The intersection of a simple cluster with a ball and its perimeter; notice the directions of the arrows.}\label{FigPerAs}
\end{figure}
\begin{prop}[The $90^\circ$ property]\label{90prop}
Let $\overline h:\R^2\to\R^+$ be a ${\rm C}^1$, positively $1$-homogeneous function, strictly positive except at $0$ and with strictly convex unit ball, and let us denote by $\overline P$ the perimeter obtained by substituting $h(x,\nu)$ with $\overline h(\nu)$ in~(\ref{weightedvolper}). There exists $\delta>0$ such that the following is true. Let $\E'\subseteq\R^2$ be a cluster whose boundary, inside the unit ball $B(0,1)$, is done by a finite number of radii of the ball. If these radii are more than three, then there exists another cluster $\F\subseteq\R^2$, coinciding with $\E'$ outside the ball $B(0,1)$, such that
\begin{equation}\label{thesis90}
\overline P(\F) \leq \overline P(\E') - \delta\,.
\end{equation}
\end{prop}
\begin{proof}
We will call for simplicity ``slice'' each of the sectors of the ball $B(0,1)$ having two consecutive radii of $\partial\E'$ in the boundary. We will say that a slice is ``white'' if it is contained in $E_0$, otherwise we will say that it is ``colored''. As already done before, for every $\nu\in\S^1$ we call $\h(\nu)=\overline h(\hat\nu)$, being $\hat\nu$ the angle obtained rotating $\nu$ of $90^\circ$ clockwise, and we call $\tilde\h$ the ``symmetrized version'' of $\h$, that is $\tilde \h(\nu)=(\h(\nu)+\h(-\nu))/2$. We let $K>0$ be a number such that
\begin{equation}\label{largeK}
\frac 1K \leq \overline h(\nu)\leq K \qquad \forall\, \nu\in\S^1\,.
\end{equation}
We assume then that there are at least four slices, and we look for a cluster $\F$ satisfying~(\ref{thesis90}). The proof is divided for clarity in a few steps.
\step{I}{The minimal angle $\theta_{\rm min}$.}
First of all, we shall observe that the thesis is true if one of the angles is too small, that is, there exist $\theta_{\rm min}>0$ and $\delta_1>0$ such that a cluster $\F$ satisfying~(\ref{thesis90}) with $\delta_1$ in place of $\delta$ can be found if one of the angles between the radii is less than $\theta_{\rm min}$. Indeed, let $P,\,Q$ be two consecutive points of $\partial^*\E'\cap\partial B(0,1)$, making with the origin a small angle $\theta$, being $Q$ slightly after $P$ in the counterclockwise sense. Let us consider the three slices around the two radii $OP$ and $OQ$. There are five possibilities: either the three slices are all colored; or only the external slice having $OP$ in the boundary is white; or the internal slice is white; or only the external slice having $OQ$ in the boundary is white; or both the external ones are white. In the first three cases, we let $\F$ be the unique cluster such that 
\begin{align*}
\F=\E' \ \hbox{in } \R^2\setminus B(0,1)\,, &&
\partial\F=\big(\partial\E' \setminus OQ\big) \cup PQ\,.
\end{align*}
Instead, in the fourth case we let $\F$ be the cluster so that
\begin{align*}
\F=\E' \ \hbox{in } \R^2\setminus B(0,1)\,, &&
\partial\F=\big(\partial\E' \setminus OP\big) \cup PQ\,,
\end{align*}
and in the last case $\F$ is the cluster such that
\begin{align*}
\F=\E' \ \hbox{in } \R^2\setminus B(0,1)\,, &&
\partial\F=\big(\partial\E' \setminus (OQ\cup OP)\big) \cup PQ\,.
\end{align*}
Keeping in mind the definition of perimeter and of $\h$ and $\tilde\h$, as well as~(\ref{largeK}), as soon as $\theta$ is small enough, only depending on $K$, in the first two cases we have
\[
\overline P(\F) - \overline P(\E') = \tilde\h(PQ) - \tilde\h(OQ) \leq 2K\sin(\theta/2) -\frac 1K\leq - \frac 1{2K}\,,
\]
and in a similar way also in the fourth and in the fifth case we have $\overline P(\F) - \overline P(\E') \leq - \frac 1{2K}$. Instead, in the third case we have
\[\begin{split}
\overline P(\F)-\overline P(&\E') \leq \h(PQ)-\h(OQ)+\tilde \h (PO)-\h(PO)\\
&=\h(PQ) -\h(OQ) + \frac{\h(OP)}2-\frac{\h(PO)}2 \\
&\leq 2K\sin(\theta/2) -\h(OQ) + \frac{\h(OP)}2-\frac 1{2K}
\leq 2K\sin(\theta/2) -\frac 1{2K} 
\leq -\frac 1{3K}\,,
\end{split}\]
where the first inequality is strict if and only if the two external slices have ``the same colors'', and the second last inequality, namely $\h(OP)\leq 2\h(OQ)$, is true by continuity of $\overline h$ as soon as $\theta$ is small enough. Summarizing, the existence of $\theta_{\rm min}$ and $\delta_1$ as claimed follows, and this step is concluded.\par

In the next steps we will show that, for any cluster $\E'$ as in the claim, with at least four slices, and with all the radii making angles larger than $\theta_{\rm min}$, there exists some cluster $\F$, coinciding with $\E'$ outside of the unit ball, such that $\overline P(\F)<\overline P(\E')$. In particular, since the points of $\partial^*\E'$ in $\partial B(0,1)$ are at most $2\pi/\theta_{\rm min}$, by continuity of $\overline h$ and compactness of $\S^1$ there must be a constant $\delta_2>0$, only depending on $\overline h$, such that $\overline P(\F) \leq \overline P(\E')-\delta_2$. Together with Step~I, this will then clearly conclude the proof, with $\delta=\min\{\delta_1,\,\delta_2\}$.

\step{II}{Proof with more than a white slice.}
We now show the thesis if there are at least two white slices. In fact, in this case, we can take two white slices, corresponding to two arcs with endpoints $P,\,Q$, and $R,\,S$ respectively, in such a way that $\angle SOP<\pi$, and the angle $\angle SOP$ corresponds to a sector of circle which does not intersect $E_0$ (see Figure~\ref{Figstraightcut}).
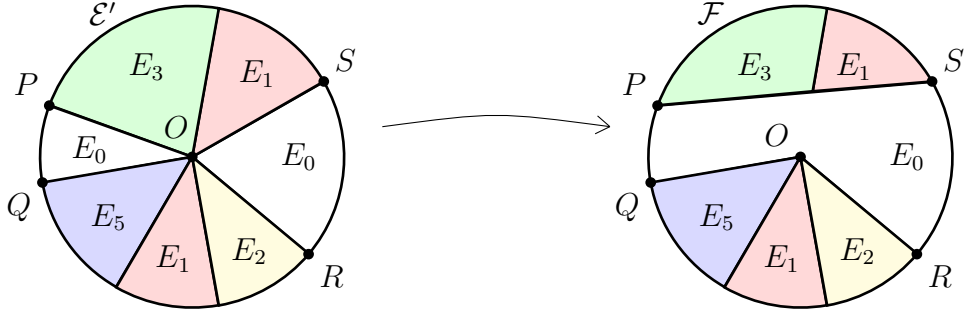
\begin{figure}[htbp]
\begin{tikzpicture}[>=>>>>]
\draw[->] (2.5,.4) .. controls (4,.6) .. (5.5,.4);
\filldraw[fill=white, draw=black, line width=1pt] (0,0) -- (1.53,-1.29) arc(-40:30:2) -- cycle;
\filldraw[fill=red!15!white, draw=black, line width=1pt] (0,0) -- (1.73,1) arc(30:80:2) -- cycle;
\filldraw[fill=green!15!white, draw=black, line width=1pt] (0,0) -- (.35,1.97) arc(80:160:2) -- cycle;
\filldraw[fill=white, draw=black, line width=1pt] (0,0) -- (-1.88,.68) arc(160:190:2) -- cycle;
\filldraw[fill=blue!15!white, draw=black, line width=1pt] (0,0) -- (-1.97,-.34) arc(190:240:2) -- cycle;
\filldraw[fill=red!15!white, draw=black, line width=1pt] (0,0) -- (-1,-1.73) arc(240:280:2) -- cycle;
\filldraw[fill=yellow!15!white, draw=black, line width=1pt] (0,0) -- (0.35,-1.97) arc(280:320:2) -- cycle;
\draw (1.4,0) node {$E_0$};
\draw (0.85,1.1) node {$E_1$};
\draw (-.6,1.2) node {$E_3$};
\draw (-1.35,0.1) node {$E_0$};
\draw (-1.1,-.8) node {$E_5$};
\draw (-0.25,-1.35) node {$E_1$};
\draw (.75,-1.25) node {$E_2$};
\fill (0,0) circle (2pt);
\draw (0.1,0.1) node[anchor=south east] {$O$};
\fill (-1.88,.68) circle (2pt);
\draw (-1.88,.68) node[anchor=south east] {$P$};
\fill (-1.97,-.34) circle (2pt);
\draw (-1.97,-.34) node[anchor=north east] {$Q$};
\fill (1.53,-1.29) circle (2pt);
\draw (1.53,-1.29) node[anchor=north west] {$R$};
\fill (1.73,1) circle (2pt);
\draw (1.73,1) node[anchor=south west] {$S$};
\draw (-1.5,1.9) node[anchor=west] {$\E'$};
\filldraw[fill=white, draw=black, line width=1pt] (8,0) -- (9.53,-1.29) arc(-40:30:2) -- (6.12,.68) arc(160:190:2) -- cycle;
\filldraw[fill=red!15!white, draw=black, line width=1pt] (8.16,0.87) -- (9.73,1) arc(30:80:2) -- cycle;
\filldraw[fill=green!15!white, draw=black, line width=1pt] (8.16,0.87) -- (8.35,1.97) arc(80:160:2) -- cycle;
\filldraw[fill=blue!15!white, draw=black, line width=1pt] (8,0) -- (6.03,-.34) arc(190:240:2) -- cycle;
\filldraw[fill=red!15!white, draw=black, line width=1pt] (8,0) -- (7,-1.73) arc(240:280:2) -- cycle;
\filldraw[fill=yellow!15!white, draw=black, line width=1pt] (8,0) -- (8.35,-1.97) arc(280:320:2) -- cycle;
\draw (9.4,0) node {$E_0$};
\draw (8.7,1.2) node {$E_1$};
\draw (7.4,1.2) node {$E_3$};
\draw (6.9,-.8) node {$E_5$};
\draw (7.75,-1.35) node {$E_1$};
\draw (8.75,-1.25) node {$E_2$};
\fill (8,0) circle (2pt);
\draw (8,0) node[anchor=south east] {$O$};
\fill (6.12,.68) circle (2pt);
\draw (6.12,.68) node[anchor=south east] {$P$};
\fill (6.03,-.34) circle (2pt);
\draw (6.03,-.34) node[anchor=north east] {$Q$};
\fill (9.53,-1.29) circle (2pt);
\draw (9.53,-1.29) node[anchor=north west] {$R$};
\fill (9.73,1) circle (2pt);
\draw (9.73,1) node[anchor=south west] {$S$};
\draw (6.5,1.9) node[anchor=west] {$\F$};
\end{tikzpicture}
\caption{The situation in Step~II.}\label{Figstraightcut}
\end{figure}
We define then $\F$ by ``joining'' the two white slices as in the Figure. Notice that $\partial\F$ is obtained by $\partial\E'$ removing the radii $PO$ and $OS$ and adding the chord $PS$, and by correspondingly shortening the radii contained in the sector. As a consequence,
\begin{equation}\label{step21a}
\overline P(\F)\leq \overline P(\E') - \h(PO) - \h (OS) + \h(PS)\,,
\end{equation}
where the inequality is strict if and only if the sector contains more than a single slice, as in the example depicted in the Figure. Since $\overline h$ is strictly convex in the sense of Definition~\ref{strictconv}, and $PO$ and $OS$ are not parallel because $\angle SOP<\pi$, we have
\[
\h(PS)= 2 \h\bigg(\frac{PO+OS}2\bigg)<\h(PO)+\h(OS)\,,
\]
thus by~(\ref{step21a}) we obtain $\overline P(\F)<\overline P(\E')$ and this step is concluded.

\step{III}{Proof with an angle $\angle AOC<\pi$ containing a single radius, between two colored slices.}
The next step consists in proving the thesis if there are three consecutive radii, $AO,\,BO$ and $CO$, in such a way that both the slices between them are colored, and that $\angle AOC<\pi$. 
\begin{figure}[htbp]
\begin{tikzpicture}[>=>>>>]
\filldraw[fill=red!15!white, draw=black, line width=1pt] (-7,0) -- (-5.47,-1.29) arc(-40:20:2) -- cycle;
\filldraw[fill=yellow!15!white, draw=black, line width=1pt] (-7,0) -- (-5.12,.68) arc(20:70:2) -- cycle;
\filldraw[fill=white, draw=black, line width=1pt] (-7,0) -- (-6.32,1.88) arc(70:320:2) -- cycle;
\draw (-5.65,-0.2) node {$E_1$};
\draw (-6.1,1) node {$E_2$};
\fill (-7,0) circle (2pt);
\draw (-7,0) node[anchor=north east] {$O$};
\fill (-5.47,-1.29) circle (2pt);
\draw (-5.47,-1.29) node[anchor=north west] {$A$};
\fill (-5.12,.68) circle (2pt);
\draw (-5.12,.68) node[anchor=south west] {$B$};
\fill (-6.32,1.88) circle (2pt);
\draw (-6.32,1.88) node[anchor=south west] {$C$};
\fill (-6.83,.47) circle (2pt);
\draw (-6.83,.47) node[anchor=south east] {$O_\eps$};
\draw[dashed] (-6.83,.47) -- (-5.12,.68);
\draw[line width=1pt, rotate=10] (0,0) ellipse (3 and 1.6);
\draw (-1.5,1.7) node[anchor=west] {$\widetilde\C$};
\fill (0,0) circle (2pt);
\draw (0,0) node[anchor=south east] {$O$};
\draw (0,0) -- (2.72,.98);
\fill (2.72,.98) circle (2pt);
\draw[->] (2.72,.98) -- (3.43,1.68);
\draw (3.15,1.2) node[anchor=south east] {$\nu$};
\draw (2.65,.9) node[anchor=north] {$\widehat B$};
\end{tikzpicture}
\caption{The situation in Step~III.}\label{Figstep3}
\end{figure}
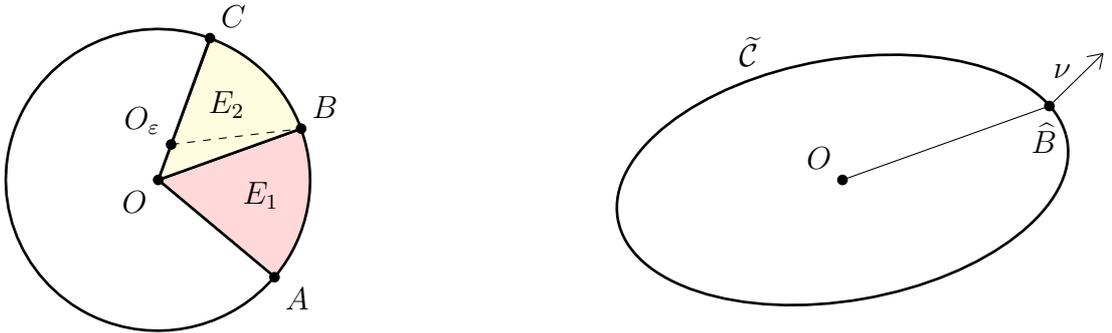
The situation is depicted in Figure~\ref{Figstep3}, left, where the two slices are denoted by $E_1$ and $E_2$ just to fix the ideas. Without loss of generality we may assume that, as in the figure, the points are ordered from $A$ to $C$ in the counterclockwise sense. Outside of the sector $AOC$ there might be further radii, not depicted in the figure. Let us denote by $\widetilde\C$ the unit ball corresponding to $\tilde\h$. Calling $\nu$, as in Figure~\ref{Figstep3}, right, the outer normal to $\widetilde\C$ at $\widehat B=B/\tilde\h(OB)$, for any direction $\eta\in\S^1$ one has
\begin{equation}\label{1order}
\tilde \h(O\widehat B+\eps\eta)
=\tilde \h\big(O\widehat B + \eps (\eta\cdot\nu) \nu\big) + o(\eps)
=1+\eps \tilde\h(OB)\, \frac{\eta\cdot\nu}{OB\cdot \nu}+ o(\eps)\,.
\end{equation}
Set then, as in the figure, $O_\eps=\eps C$ for some small $\eps>0$, and consider the cluster $\F$ obtained from $\E'$ by substituting the radius $OB$ with the segment $O_\eps B$. Notice that the difference $\overline P(\F)-\overline P(\E')$ is only determined by the different contribution of $OB$ and $O_\eps B$; indeed, the segment $OO_\eps$ contributes to $\overline P(\F)$ exactly as to $\overline P(\E')$, because on one side of the segment nothing happens, while on the other side the set $E_2$ is replaced by $E_1$, and this does not make any difference since both are colored. Therefore,
\[
\overline P(\F)- \overline P(\E') = \tilde\h(O_\eps B) - \tilde\h(OB) = \tilde\h(OB -\eps OC) - \tilde\h(OB)\,.
\]
Keeping in mind the first order expansion~(\ref{1order}), and observing that $OB\cdot \nu>0$ by convexity of $\widetilde \C$, we derive that $\overline P(\F)<\overline P(\E')$ for $0<\eps\ll 1$ if $OC\cdot \nu>0$. The step is then concluded in this case. Since we can perform the same argument with the segment $OA$ in place of $OC$, the step is proved unless
\begin{align}\label{sfaig}
OA \cdot \nu\leq 0\,, && OC \cdot \nu\leq 0\,.
\end{align}
And in turn, we can observe that~(\ref{sfaig}) is impossible. Indeed, the set $\{\eta\in\R^2:\, \eta\cdot \nu\leq 0\}$ is a half-space. And since $\angle AOC<\pi$, if this half-space contains both $OA$ and $OC$ then it must contain also $OB$, while as already observed $OB\cdot \nu>0$.

\step{IV}{Proof with an angle $\angle BOD<\pi$ containing a single radius.}
The next step consists in proving the thesis if there are two consecutive slices making together an angle strictly less than $\pi$. Notice that this is exactly what we have done in Step~III, except for the fact that we assumed there both slices to be colored. In this step we have then only to consider the case when one of the two slices is white. We assume the radii to be $BO,\, CO$ and $DO$, and without loss of generality we assume the points $B,\,C$ and $D$ to be ordered in the counterclockwise sense, and the slice between the radii $BO$ and $CO$ to be the white one, as in Figure~\ref{Figstep4}.\par
\begin{figure}[htbp]
\begin{tikzpicture}[>=>>>]
\filldraw[fill=white, draw=black, line width=1pt] (0,0) -- (-1.97,-.35) arc(190:260:2) -- cycle;
\filldraw[fill=green!15!white, draw=black, line width=1pt] (0,0) -- (-.35,-1.97) arc(260:340:2) -- cycle;
\draw[line width=1pt] (1.88,-.68) arc(-20:190:2);
\fill (0,0) circle (2pt);
\draw (0,0) node[anchor=south west] {$O$};
\fill (-1.97,-.35) circle (2pt);
\draw (-1.97,-.35) node[anchor=north east] {$B$};
\fill (-.35,-1.97) circle (2pt);
\draw (-.35,-1.97) node[anchor=north east] {$C$};
\fill (1.88,-.68) circle (2pt);
\draw (1.88,-.68) node[anchor=north west] {$D$};
\fill (-.66,-.12) circle (2pt);
\draw (-.66,-.12) node[anchor=south east] {$\widetilde B$};
\draw[dashed] (-.66,-.12) -- (-.35,-1.97);
\draw[<->] (0,0.3) -- (-.66,.18);
\draw (-.33,.28) node[anchor=south] {$\eps$};
\draw (-1.05,-.7) node {$E_0$};
\draw (.7,-1.1) node {$E_3$};
\filldraw[fill=white, draw=black, line width=1pt] (8,0) -- (6.03,-.35) arc(190:260:2) -- cycle;
\filldraw[fill=green!15!white, draw=black, line width=1pt] (8,0) -- (7.65,-1.97) arc(260:340:2) -- cycle;
\draw[line width=1pt] (9.88,-.68) arc(-20:190:2);
\fill (8,0) circle (2pt);
\draw (8,0) node[anchor=south east] {$O$};
\fill (6.03,-.35) circle (2pt);
\draw (6.03,-.35) node[anchor=north east] {$B$};
\fill (7.65,-1.97) circle (2pt);
\draw (7.65,-1.97) node[anchor=north east] {$C$};
\fill (9.88,-.68) circle (2pt);
\draw (9.88,-.68) node[anchor=north west] {$D$};
\fill (8.66,.12) circle (2pt);
\draw (8.66,.12) node[anchor=south west] {$H$};
\draw[dashed] (8.66,.12) -- (7.65,-1.97);
\draw[<->] (8,0.3) -- (8.66,.42);
\draw[dashed] (8,0) -- (8.66,.12);
\draw (8.33,.40) node[anchor=south] {$\eps$};
\draw (6.95,-.7) node {$E_0$};
\draw (8.7,-1.1) node {$E_3$};
\fill (8.51,-.19) circle (2pt);
\draw (8.7,-.25) node[anchor=north] {$W$};
\end{tikzpicture}
\caption{The situation in Step~IV.}\label{Figstep4}
\end{figure}
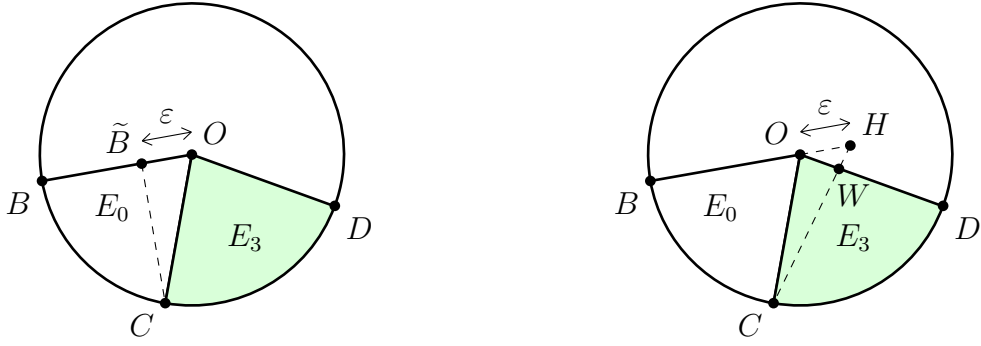
We define first a possible competitor $\F$ as in Figure~\ref{Figstep4}, left. Namely, for a small, positive $\eps$ we define $\widetilde B=\eps B$ and we let $\F$ be the cluster obtained by $\E'$ substituting the radius $OC$ with the segment $\widetilde BC$. This time, the difference between $\overline P(\E')$ and $\overline P(\F)$ is given not only by the different contribution of $OC$ and $\widetilde BC$, but also by the fact that the small segment $O\widetilde B$ is between a white and a colored slice in $\E'$, while it is between two colored slices in $\F$. Actually, the segment $O\widetilde B$ is not even in $\partial^*\F$ if the colored slice on the other side of $OB$ has the same color as the slice of the sector $COD$. Therefore,
\begin{equation}\label{estste4}
\overline P(\F) -\overline P(\E') \leq \h(\widetilde BC)-\h (OC) + \frac{\h(O\widetilde B)-\h(\widetilde B O)}2\,,
\end{equation}
and the inequality is strict if and only if the colored slice on the other side of $OB$ and the slice of the sector $COD$ have the same color. There is a constant $\kappa\in\R$ such that
\[
\h(\widetilde BC)-\h (OC) = \kappa \eps + o(\eps)
\]
(the exact value of $\kappa$ can be found as in~(\ref{1order}), but in this step this is not important). Hence, from~(\ref{estste4}) we get
\[
\overline P(\F) \leq \overline P(\E') + \eps \bigg(\kappa + \frac{\h(OB)-\h(BO)}2\bigg) + o(\eps) \,,
\]
so that the competitor $\F$ concludes the proof in this case unless
\begin{equation}\label{lastcase}
\kappa + \frac{\h(OB)-\h(BO)}2\geq 0\,.
\end{equation}
Let us then assume that this last inequality holds true, and let us define a different competitor, as in Figure~\ref{Figstep4}, right. More precisely, again for a small positive $\eps$ we define $H=-\eps B$, and we let $W$ be the point of intersection between the segments $HC$ and $OD$. The cluster $\F$ is then obtained by substituting the radius $OC$ with the segment $WC$. Arguing as before, and keeping in mind that the slice on the other side of $OD$ is surely colored by Step~II, we have this time
\begin{equation}\label{lastcomp}
\overline P(\F)-\overline P(\E')= \h(WC)-\h(OC) + \frac{\h(OW)-\h(WO)}2\,.
\end{equation}
Notice that
\[
\h(WC)-\h(OC) = \h(HC)-\h(OC)-\h(HW) = -\kappa \eps - \h(HW) +o(\eps)\,,
\]
so by~(\ref{lastcomp}) we get the thesis with some small $\eps>0$ if
\[
\lim_{\eps\searrow 0} \ \bigg(\kappa+\frac{\h(HW)}\eps + \frac{\h(WO)-\h(OW)}{2\eps}\bigg)>0\,,
\]
which in turn, thanks to~(\ref{lastcase}), is surely true if
\begin{equation}\label{lastst4}
\lim_{\eps\searrow 0} \frac{2\h(HW) + \h(WO)-\h(OW) + \h(OH)-\h(HO)}\eps>0\,.
\end{equation}
Consider now the triangle $WOH$ and observe that, by elementary geometric relations,
\begin{align*}
\angle WOH= \angle DOH = \pi - \angle BOD \,, &&
\angle HWO =\angle CWD > \angle COD \,, &&
\angle OHW = \angle BHC > \frac{\angle BOC}2 \,.
\end{align*}
Hence, the three angles of the triangle $WOH$ depend on $\eps$, but they are all greater than a strictly positive constant which does not depend on $\eps$. Since $\overline h$ is strictly convex, there exists then a constant $\delta_3>0$ such that
\begin{align*}
\h(HW)+\h(WO) \geq (1+\delta_3) \h(HO)\,, && \h(OH)+ \h(HW)\geq (1+\delta_3)\h(OW)\,.
\end{align*}
We deduce
\[\begin{split}
&\lim_{\eps\searrow 0} \frac{2\h(HW) + \h(WO)-\h(OW) + \h(OH)-\h(HO)}\eps\\
&\hspace{60pt}\geq \lim_{\eps\searrow 0} \frac{\delta_3\big(\h(HO)+\h(OW) \big)}\eps
\geq \lim_{\eps\searrow 0} \frac{\delta_3\h(HO)}\eps
=\delta_3\h(OB)>0\,,
\end{split}\]
so~(\ref{lastst4}) is established and the proof follows also in this case.

\step{V}{Conclusion.}
We are now ready to conclude the thesis. By Step~III and Step~IV, the only case which is left open is when there are exactly four radii, say $OA,\,OB,\,OC$ and $OD$, with the points $A,\,B,\,C,\,D$ ordered in the counterclockwise sense, and $\angle AOC=\angle BOD=\pi$, as in Figure~\ref{Figstep5}, left. Since by Step~II there can be at most one white slice, we assume that the slices corresponding to the sectors $AOB,\, BOC$ and $COD$ are colored.
\begin{figure}[htbp]
\begin{tikzpicture}[>=>>>]
\filldraw[fill=white, draw=black, line width=1pt] (0,0) -- (-1.97,.34) arc(170:290:2) -- cycle;
\filldraw[fill=yellow!15!white, draw=black, line width=1pt] (0,0) -- (.68,-1.88) arc(-70:-10:2) -- cycle;
\filldraw[fill=red!15!white, draw=black, line width=1pt] (0,0) -- (1.97,-.34) arc(-10:110:2) -- cycle;
\filldraw[fill=green!15!white, draw=black, line width=1pt] (0,0) -- (-.68,1.88) arc(110:170:2) -- cycle;
\fill (0,0) circle (2pt);
\draw (0,0) node[anchor=north east] {$O$};
\fill (-.68,1.88) circle (2pt);
\draw (-.68,1.88) node[anchor=south east] {$C$};
\fill (-1.97,.34) circle (2pt);
\draw (-1.97,.34) node[anchor=south east] {$D$};
\fill (1.97,-.34) circle (2pt);
\draw (1.97,-.34) node[anchor=north west] {$B$};
\fill (0.68,-1.88) circle (2pt);
\draw (.68,-1.88) node[anchor=north west] {$A$};
\fill (.66,-.12) circle (1.5pt);
\draw (.66,-.12) node[anchor=north] {$H$};
\fill (-.23,.63) circle (1.5pt);
\draw (-.23,.63) node[anchor=north east] {$K$};
\fill (.43,.51) circle (1.5pt);
\draw (.43,.51) node[anchor=south west] {$W$};
\draw[<->] (-.23,.63) -- (.43,.51);
\draw[<->] (.66,-.12) -- (.43,.51);
\draw (.54,.25) node[anchor=west] {$\eps$};
\draw (.10,.57) node[anchor=south] {$\eps$};
\draw[->] (2.7,.4) .. controls (4,.6) .. (5.3,.4);
\filldraw[fill=white, draw=black, line width=1pt] (8,0) -- (6.03,.34) arc(170:290:2) -- cycle;
\filldraw[fill=yellow!15!white, draw=black, line width=1pt] (8.43,.51) -- (8,0) -- (8.68,-1.88) arc(-70:-10:2) -- cycle;
\filldraw[fill=red!15!white, draw=black, line width=1pt] (8.43,0.51) -- (9.97,-.34) arc(-10:110:2) -- cycle;
\filldraw[fill=green!15!white, draw=black, line width=1pt] (8,0) -- (8.43,.51) -- (7.32,1.88) arc(110:170:2) -- cycle;
\fill (8,0) circle (2pt);
\draw (8,0) node[anchor=north east] {$O$};
\fill (7.32,1.88) circle (2pt);
\draw (7.32,1.88) node[anchor=south east] {$C$};
\fill (6.03,.34) circle (2pt);
\draw (6.03,.34) node[anchor=south east] {$D$};
\fill (9.97,-.34) circle (2pt);
\draw (9.97,-.34) node[anchor=north west] {$B$};
\fill (8.68,-1.88) circle (2pt);
\draw (8.68,-1.88) node[anchor=north west] {$A$};
\fill (8.66,-.12) circle (1.5pt);
\draw (8.66,-.12) node[anchor=north] {$H$};
\fill (7.77,.63) circle (1.5pt);
\draw (7.77,.63) node[anchor=north east] {$K$};
\fill (8.43,.51) circle (1.5pt);
\draw (8.43,.51) node[anchor=south west] {$W$};
\draw [dashed] (8,0) -- (7.32,1.88);
\draw [dashed] (8,0) -- (9.97,-.34);
\end{tikzpicture}
\caption{The situation in Step~V.}\label{Figstep5}
\end{figure}
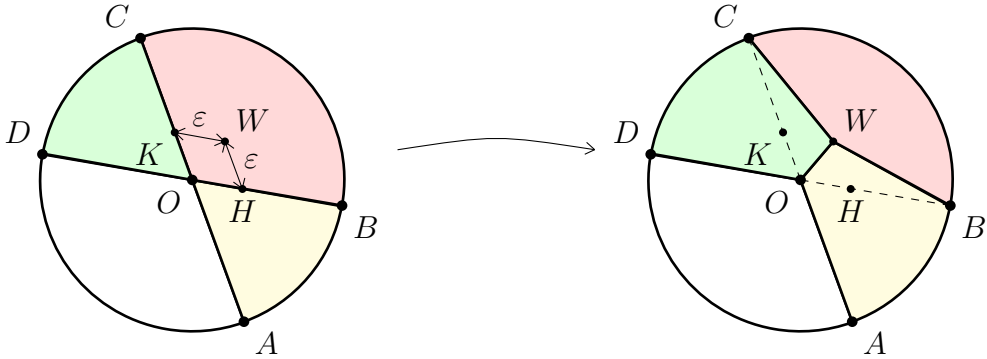
We are going to use only the fact that these slices are colored, the fact whether or not so is also the slice $DOA$ does not play any role. As in Step~III, let us call $\nu$ the direction of the outer normal at $B/\tilde\h(OB)$ to $\widetilde\C=\{\tilde\h\leq 1\}$. Since both the sectors $AOB$ and $BOC$ are colored, Step~III already gives the proof unless~(\ref{sfaig}) holds. As noticed in Step~III, (\ref{sfaig}) is in fact impossible if $\angle AOC<\pi$, and if $\angle AOC=\pi$ it holds only if $OA\cdot \nu =OC\cdot \nu=0$, which by the first order expansion~(\ref{1order}) implies that
\begin{equation}\label{lastcase1}
\tilde\h(OB-\eps OC)=\tilde\h(OB)+o(\eps)\,.
\end{equation}
Repeating the same argument in the union of the sectors $BOC$ and $COD$, which are also both colored and correspond to the angle $\angle BOD=\pi$, we get the thesis unless
\begin{equation}\label{lastcase2}
\tilde\h(OC-\eps OB) = \tilde\h(OC)+o(\eps)\,.
\end{equation}
To conclude, we have then only to find a suitable competitor under the assumption that~(\ref{lastcase1}) and~(\ref{lastcase2}) hold. In this final case, as in Figure~\ref{Figstep5}, right, we call $H=\eps B$, $K=\eps C$ and $W=\eps(B+C)=H+K$, and we define the cluster $\F$ substituting in $\partial\E'$ the radii $OB$ and $OC$ with the three segments $OW,\, WB$ and $WC$, so that in particular all the segments in $\partial\E'\Delta\partial\F$ are between two colored slices. We observe that
\[\begin{split}
\tilde\h(WB) &= \tilde\h(OB-OW) = \tilde\h\big((1-\eps)OB - \eps(OC)\big)
=(1-\eps) \tilde\h\bigg(OB - \frac \eps{1-\eps}\, OC\bigg)\\
&=(1-\eps) \tilde\h(OB) +o(\eps) = \tilde\h(HB)+o(\eps)\,,
\end{split}\]
where in the second last equality we have used~(\ref{lastcase1}). In the very same way, using~(\ref{lastcase2}), we have $\tilde\h(WC)=\tilde\h(KC)+o(\eps)$. Therefore, we get
\[\begin{split}
\overline P(\F)-\overline P(\E') &= \tilde\h(OW)+\tilde\h(WB)+\tilde\h(WC)-\tilde\h(OB)-\tilde\h(OC)\\
&= \tilde\h(OW)+\tilde\h(HB)+\tilde\h(KC)-\tilde\h(OB)-\tilde\h(OC)+o(\eps)\\
&=\eps \Big(\tilde\h(OB+OC)-\tilde\h(OB)-\tilde\h(OC)\Big) + o(\eps)\,,
\end{split}\]
and by the strict convexity of $\overline h$ we deduce $\overline P(\F)<\overline P(\E')$ for some small, positive $\eps$. The proof is then concluded.
\end{proof}

\begin{rmk}\label{fixdelta}
It is important to observe that the constant $\delta=\delta(\overline h)$ in the above proposition only depends on the norm $\overline h$. By continuity of $h$, we can then fix a constant $\delta>0$, depending only on $h$ and on $\textsf{D}$, such that $\delta\leq \delta(\overline h)$ for every $\overline h$ of the form $\overline h(v)=h(x,v)$ for some $x\in \textsf{D}$. We will apply Proposition~\ref{90prop} with such a choice.
\end{rmk}

We conclude this section by presenting a simple observation and an important consequence.

\begin{lem}\label{lemmagiorgio}
Let $\h:\R^2\to\R^+$ be a convex and positively $1$-homogeneous function, and for every path $\gamma:[0,1]\to\R^2$ of finite length let us call $\len(\gamma)$ the ``length of $\gamma$'' defined by
\begin{equation}\label{lengthcurves}
\len(\gamma) = \int_0^1 \h(\gamma'(\sigma))\,d\sigma\,.
\end{equation}
For any such path $\gamma$, then, one has
\[
\len(\gamma)\geq \len(\tilde\gamma)\,,
\]
where $\tilde\gamma:[0,1]\to\R^2$ is the affine path connecting $\gamma(0)$ with $\gamma(1)$.
%Moreover, if the ``unit ball'' $\{\h\leq 1\}$ is \emph{strictly} convex, then for every two constants $\kappa_1,\,\kappa_2>0$ there is a constant $\chi(\kappa_1,\kappa_2)>0$ such that, if $\gamma:[0,1]\to\R^2$ is a Lipschitz path parametrized by arclength, then
%\begin{align*}
%\bigg|\bigg\{t\in [0,1]:\, \frac{|\gamma'(t)-\tilde\gamma'(t)|}{|\tilde\gamma'(t)|}>\kappa_1 \bigg\}\bigg|>\kappa_2
%&& \Longrightarrow &&
%\frac{\len(\gamma)}{\len(\tilde\gamma)}\geq 1+ \chi(\kappa_1,\kappa_2)\,.
%\end{align*}
\end{lem}
\begin{proof}
%The first property
This is a direct application of Jensen lemma,
\[
\len(\gamma) = \int_0^1 \h(\gamma'(\sigma))\,d\sigma
\geq \h\bigg(\int_0^1 \gamma'(\sigma)\,d\sigma\bigg)
= \h\Big(\gamma(1)-\gamma(0)\Big)
=\len(\tilde\gamma)\,.
\]
\end{proof}

\begin{cor}\label{helpnolakes}
Let $\h$ and $\len$ be as in Lemma~\ref{lemmagiorgio}, and let $\tau_1,\,\tau_2:[0,1]\to\R^2$ be two injective paths of finite length which have no intersection except the points $P=\tau_1(0)=\tau_2(1)$ and $Q=\tau_1(1)=\tau_2(0)$. Then, there exists an injective path $\tau:[0,1]\to\R^2$ with $\tau(0)=P$ and $\tau(1)=Q$, which is entirely contained in the (closed) region enclosed by $\tau_1\cup\tau_2$, and such that, setting $\hat\tau:[0,1]\to\R^2$ as $\hat\tau(t)=\tau(1-t)$, one has
\begin{equation}\label{half}
\len(\tau) + \len(\hat\tau) \leq \len(\tau_1)+\len(\tau_2)\,.
\end{equation}
\end{cor}
\begin{proof}
By approximation, we can assume that the paths $\tau_1$ and $\tau_2$ are done by finitely many linear pieces, so that the region enclosed by $\tau_1\cup\tau_2$ is a closed polygon $\PP$. In addition, we can also assume that \emph{every} couple of vertices of $\PP$ (not only the couples of consecutive vertices) corresponds to a different direction, so in particular there are no three aligned vertices. We argue then by induction on the number $N$ of sides of $\PP$.\par

If $N=3$, then necessarily one of the paths, say $\tau_1$, is simply the segment $PQ$, and the other path is done by two linear pieces, say $QB$ and $BP$. In this case, it is enough to call $\tau$ the segment between $P$ and $Q$, and then~(\ref{half}) is obvious by Lemma~\ref{lemmagiorgio}, since
\[\begin{split}
\len(\tau_1)+\len(\tau_2) &= \len(PQ) + \len(QB)+\len(BP) \geq \len (PQ) + \len(QP) \\
&=\len(\tau) + \len(\hat\tau) \,.
\end{split}\]
Let us then assume that $N\geq 4$ and that the claim has been proven for all the polygons with strictly less sides than $N$. If there are two vertices $B,\, D$ in $\tau_1$ such that the open segment $BD$ is contained in the interior of $\PP$, then we can call $\tilde\tau_1$ the path obtained by $\tau_1$ by substituting the whole part between $B$ and $D$ with the segment $BD$, and $\tilde\tau_2=\tau_2$. The resulting polygon $\widetilde\PP$ is contained in $\PP$ and has strictly less than $N$ vertices. By assumption, we find then a path $\tau$ as in the claim for the polygon $\widetilde\PP$. The path $\tau$ is contained in $\widetilde\PP$, hence in $\PP$, and~(\ref{half}) holds true since, again by Lemma~\ref{lemmagiorgio},
\[
\len(\tau)+\len(\hat\tau) \leq \len(\tilde\tau_1) + \len(\tilde\tau_2) \leq \len(\tau_1)+\len(\tau_2)\,.
\]
In the same way we argue if the two vertices $B,\, D$ belong to $\tau_2$.\par
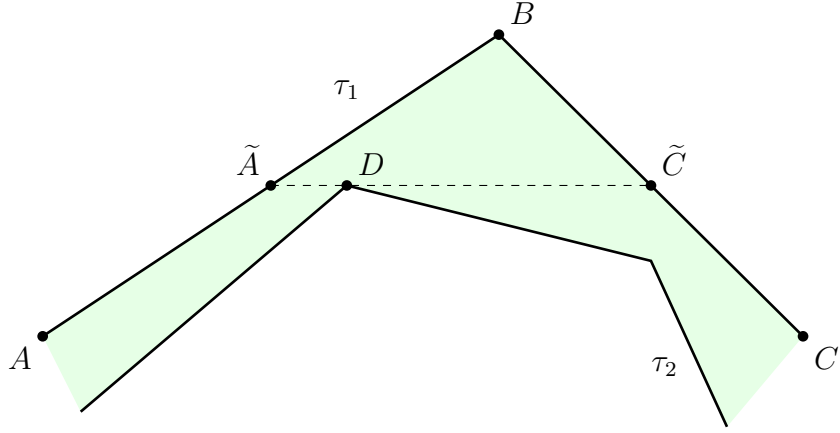
\begin{figure}[htbp]
\begin{tikzpicture}
\fill[fill=green!10!white] (0,0) -- (6,4) -- (10,0) -- (9,-1.2) -- (8,1) -- (4,2) -- (.5,-1) -- cycle;
\fill (0,0) circle (2pt);
\draw (0,0) node[anchor=north east] {$A$};
\fill (3,2) circle (2pt);
\draw (3,2) node[anchor=south east] {$\widetilde A$};
\fill (6,4) circle (2pt);
\draw (6,4) node[anchor=south west] {$B$};
\fill (8,2) circle (2pt);
\draw (8,2) node[anchor=south west] {$\widetilde C$};
\fill (10,0) circle (2pt);
\draw (10,0) node[anchor=north west] {$C$};
\fill (4,2) circle (2pt);
\draw (4,2) node[anchor=south west] {$D$};
\draw[line width=1pt] (0,0) -- (6,4) -- (10,0);
\draw[line width=1pt] (.5,-1) -- (4,2) -- (8,1) -- (9,-1.2);
\draw[dashed] (3,2) -- (8,2);
\draw (4,3) node[anchor=south] {$\tau_1$};
\draw (8.5,-.4) node[anchor=east] {$\tau_2$};
\end{tikzpicture}
\caption{A possible situation in Corollary~\ref{helpnolakes}.}\label{FigFra}
\end{figure}
Let us finally assume that there are no such vertices. A possible situation is depicted in Figure~\ref{FigFra}. We can take three consecutive vertices $A,\, B,\, C$ in one of the paths, say $\tau_1$, such that the angle at $B$ is less than $\pi$ (this is clearly possible since every polygon has at least three angles less than $\pi$, so at least one vertex of $\PP$ different from $P$ and $Q$ corresponds to an angle less than $\pi$). Since the open segment $AC$ cannot be contained in the interior of $\PP$ by assumption, the triangle $ABC$ contains other vertices of the polygon. In particular, there are two points $\widetilde A$ and $\widetilde C$ in $AB$ and $BC$ respectively, such that the open segment $\widetilde A \widetilde C$ is parallel to $AC$ and intersects $\partial\PP$ exactly at one point, say $D$. The point $D$ is necessarily a vertex of the polygon, and the open segment $BD$ is contained in the interior of $\PP$, thus by assumption $D$ must be contained in $\tau_2$. Let us then call $\tau_1^1$ the path obtained by taking the part of $\tau_1$ between $P$ and $\widetilde A$ and adding the segment $\widetilde A D$, and $\tau_1^2$ the segment $D\widetilde C$ together with the part of $\tau_1$ between $\widetilde C$ and $Q$. Moreover, subdivide $\tau_2$ in the part $\tau_2^2$ between $Q$ and $D$, and the part $\tau_2^1$ between $D$ and $P$. The paths $\tau_1^1$ and $\tau_2^1$ enclose a polygon $\PP^1$, while $\tau_1^2$ and $\tau_2^2$ enclose $\PP^2$. Both polygons are contained in $\PP$ and have strictly less sides than $\PP$, thus by inductive assumption we obtain a path $\tau^1$ in $\PP^1$ between $P$ and $D$ and a path $\tau^2$ in $\PP^2$ beween $D$ and $Q$, which satisfy the inequalities analogous to~(\ref{half}). The path $\tau$ obtained putting together $\tau^1$ and $\tau^2$ is then a path in $\PP^1\cup\PP^2\subseteq \PP$, and satisfies~(\ref{half}) since applying once again Lemma~\ref{lemmagiorgio} we have
\[\begin{split}
\len(\tau) + \len(\hat\tau) &= \len(\tau^1) + \len(\widehat{\tau^1}) + \len(\tau^2) + \len(\widehat{\tau^2})\\
&\leq \len(\tau_1^1) + \len(\tau_2^1) + \len(\tau_1^2) + \len(\tau_2^2)
\leq \len(\tau_1)+\len(\tau_2)\,.
\end{split}\]
\end{proof}

\subsection{Finitely many triple points\label{triplepoints}}

We now start our construction for proving Theorem~\mref{main}. Through this section and the following one, $\E$ is a fixed, minimal cluster, satisfying the assumptions of Theorem~\mref{main}, and $\textsf{D}$ is a fixed, closed ball. The aim of this section is to show several preliminary properties of $\E$, eventually establishing that $\partial^* \E$ only admits (in $\textsf{D}$) finitely many junction points, and all of them are triple points. This will be obtained in Lemma~\ref{dist>R6}.\par

We set $R_1=\min\{R_\beta,\,R_\eta\}$ (the constants $R_\beta$ and $R_\eta$ have been defined in Definition~\ref{defge}). In the following, we will define several different values of $R_i$ with $R_1\geq R_2\geq R_3\, \cdots$. Each of these constants will only depend on $\E$, $\textsf{D}$, $g$ and $h$.

Our first result is a simple observation following from the $\eps-\eps^\beta$ property and the $\eta$-growth condition, that one can use to build competitors. We will use it several times in the sequel.

\begin{lem}[Small ball competitor]\label{labello}
Let $B(x,r)\subseteq \textsf{D}$ be a ball such that $|B(x,r)|<\bar\eps/2$ and $r<R_1$, let $\E,\, \E'$ be any two clusters which coincide outside $B(x,r)$, and call $\eps=|\E|-|\E'|$. There exists another cluster $\E''$ such that $|\E''|=|\E|$, $\E''\cap B(x,r)=\E'\cap B(x,r)$ and
\begin{equation}\label{proplab}
P(\E'')\leq P(\E') + \Ceeb |\eps|^\beta \leq P(\E') + \Ceeb (2\Cgr r^\eta)^\beta \,.
\end{equation}
This inequality is actually true also with $\Ceeb[|\eps|]$ in place of $\Ceeb$.
\end{lem}
\begin{proof}
We start by noticing that
\[
|\eps| \leq \sum_{i=1}^m |\eps_i| \leq 2 |B(x,r)| < \bar\eps\,,
\]
hence we can apply the $\eps-\eps^\beta$ property to $\E$ with constant $\eps$ and point $x$. Hence, there is another cluster $\F$ such that $\F=\E$ inside $B(x,R_\beta)\supseteq B(x,r)$, and moreover $|\F|=|\E|+\eps$ and $P(\F)\leq P(\E) + \Ceeb[|\eps|] |\eps|^\beta$. We define then the cluster $\E''$ as the cluster which coincides with $\E'$ inside $B(x,r)$, and with $\F$ outside of $B(x,r)$. Its volume is then
\[
|\E''| = |\E' \cap B(x,r)| + |\F\setminus B(x,r)|
 = |\E\cap B(x,r)| + |\E'| - |\E| + |\E\setminus B(x,r)| + |\F| - |\E|
=|\E|\,.
\]
Keeping in mind the growth condition, we have
\[
|\eps| \leq 2 |B(x,r)| \leq 2 \Cgr r^\eta\,.
\]
As a consequence, the perimeter of $\E$ can be evaluated as
\[\begin{split}
P(\E'') &= P(\E'; B(x,r)) + P\big(\F;\R^2\setminus B(x,r)\big)\\
&= P(\E; B(x,r)) +P(\E')-P(\E) + P\big(\E;\R^2\setminus B(x,r)\big) + P(\F)-P(\E)\\
&\leq P(\E') + \Ceeb[|\eps|] |\eps|^\beta
\leq P(\E') + \Ceeb[|\eps|] \big(2\Cgr r^\eta\big)^\beta\,.
\end{split}\]
Keeping in mind that $\Ceeb[|\eps|]\leq \Ceeb[\bar\eps]\leq\Ceeb$, the proof is then concluded.
\end{proof}

\begin{lem}[Perimeter in a ball is controlled by radius]\label{13/2}
There exists a constant $R_2\leq R_1$ such that, for every $B(x,r)\subseteq \textsf{D}$ with $r<R_2$, one has
\begin{equation}\label{step1}
\haus^1(\partial^* \E\cap B(x,r)) < 7\,\frac{h_{\rm max}}{h_{\rm min}} \, r\,.
\end{equation}
\end{lem}
\begin{proof}
We let $R_2\leq R_1$ be so small that
\begin{align}\label{choiceR2}
\Cgr R_2^\eta <\frac{\bar\eps} 2\,, &&
\Ceeb \big( 2 \Cgr R_2^\eta\big)^\beta < \frac{h_{\rm max}}2\,R_2\,.
\end{align}
Notice that the first inequality is true for every $R_2$ small enough. The same is true for the second one if $\eta\beta>1$. Instead, if $\eta\beta=1$, the second inequality is true regardless of the value of $R_2$ thanks to~(\ref{Cperissmall}) as soon as we define
\begin{equation}\label{ceeb1}
\Ceeb^1 =\frac{h_{\rm max}}{2^{\beta+1}\Cgr^\beta} \,.
\end{equation}
Take now $r<R_2$ and $x\in\R^2$ as in the claim. Let us define the cluster $\E'$ by setting $E_1'=E_1\cup B(x,r)$ and $E_i'=E_i\setminus B(x,r)$ for every $2\leq i\leq m$. Clearly
\begin{equation}\label{quellaltro}
P(\E')\leq P(\E) - h_{\rm min} \haus^1(\partial^* \E \cap B(x,r)) + 2\pi r h_{\rm max}\,.
\end{equation}
Let us call $\eps\in\R^m$ the vector given by $\eps_i=|E_i\cap B(x,r)|$ for every $2\leq i\leq m$, and $\eps_1=-|B(x,r)\setminus E_1|$, so that $|\E|=|\E'|+\eps$. Notice that $|B(x,r)|\leq \Cgr r^\eta<\bar\eps/2$ by the first property in~(\ref{choiceR2}). Hence, we can apply Lemma~\ref{labello} to get another cluster $\E''$ satisfying~(\ref{proplab}), so that
\[
P(\E'') \leq P(\E') + \Ceeb (2\Cgr r^\eta)^\beta < P(\E') + \frac{h_{\rm max}}2\, r
\]
by the second property in~(\ref{choiceR2}), which is clearly valid with every $r<R_2$ in place of $R_2$. Putting this estimate together with~(\ref{quellaltro}), and recalling that $P(\E)\leq P(\E'')$ by minimality of $\E$ and since $|\E''|=|\E|$ by Lemma~\ref{labello}, we get
\[
\haus^1(\partial^*\E\cap B(x,r)) 
\leq 2\pi r \,\frac{h_{\rm max}}{h_{\rm min}} + \frac{h_{\rm max}}{2h_{\rm min}}\,r
<7\,\frac{h_{\rm max}}{h_{\rm min}} \, r\,,
\]
hence the proof is concluded.
\end{proof}

We can now show that there can be no ``islands'' in small balls, that is, if a set $E_i$ intersects a small ball then it must also intersect its boundary. Notice that, at least for the moment, $i$ cannot attain the value $0$, hence it is still possible that there is a empty hole (or ``lake'') compactly contained inside a small ball. We will rule out this possibility later.

\begin{lem}[No-islands]\label{noisland}
There exist a constant $K>0$, only depending on $\textsf{D},\,h$ and $\E$, and a constant $R_3\leq R_2$, such that for every $r< R_3$, every finite perimeter set $G\subseteq B(x,r)\subseteq \textsf{D}$ and every $1\leq i \leq m$, one has
\begin{equation}\label{effnoisl}
\haus^1\big(\partial^* (E_i \cap G)\big) \leq K \haus^1\big(E_i \cap \partial^* G\big)\,.
\end{equation}
In particular, if for some $1\leq i\leq m$ one has
\begin{equation}\label{intint}
|E_i \cap G|>0\,,
\end{equation}
then also
\begin{equation}\label{intbor}
\haus^1(E_i \cap \partial^* G)>0\,.
\end{equation}
\end{lem}
\begin{proof}
Let $R_3\leq R_2$ be a constant, to be precised later, take a set of finite perimeter $G$ contained in a ball $B(x,r)\subseteq \textsf{D}$ with $r<R_3$, and fix $1\leq i \leq m$. To get the thesis it is enough to prove~(\ref{effnoisl}). Indeed, if~(\ref{intint}) holds true, then $\haus^1(\partial^*(E_i\cap G))>0$, which by~(\ref{effnoisl}) gives~(\ref{intbor}).\par

Let us call $F=E_i\cap G$, and assume that $|F|>0$, since otherwise~(\ref{effnoisl}) is emptily true. We claim that, provided $K$ is large enough, if~(\ref{effnoisl}) is false then we can find a competitor, that is, a cluster $\E'$ such that
\begin{align}\label{competitor}
\E'=\E \hbox{ outside $F$}\,, && P(\E')\leq P(\E) - \frac{h_{\rm min}^2}{3mh_{\rm max}}\, \haus^1(\partial^* F)\,.
\end{align}
We shall first observe that the existence of such a cluster is impossible if $R_3$ has been taken small enough, so that the thesis will follow by proving the claim. By Lemma~\ref{disgusting}, from~(\ref{competitor}), right, we deduce
\begin{equation}\label{competitor2}
P(\E') \leq P(\E)- \frac{h_{\rm min}^2}{3mh_{\rm max}^2}\, P(F)
\leq P(\E) - \frac{h_{\rm min}^3}{3\Cgr^{1/\eta}mh_{\rm max}^2}\, |F|^{1/\eta}\,.
\end{equation}
Let us define $\eps=|\E|-|\E'|$, so that by~(\ref{competitor}), left, $|\eps|\leq 2|F|$. Applying Lemma~\ref{labello}, we get a cluster $\E''$ with $|\E''|=|\E|$ and
\[
P(\E'')\leq P(\E') + \Ceeb |\eps|^\beta \leq P(\E') + \Ceeb 2^\beta |F|^\beta\,.
\]
Putting this inequality together with~(\ref{competitor2}), by the optimality of $\E$ we find
\begin{equation}\label{aboine}
|F|^{\beta-1/\eta}\geq \frac{h_{\rm min}^3}{3\cdot 2^\beta \Ceeb \Cgr^{1/\eta} m h_{\rm max}^2}\,.
\end{equation}
We can again distinguish the case $\eta\beta=1$ and the case $\eta\beta>1$. If $\eta\beta=1$, then the above inequality is false thanks to~(\ref{Cperissmall}) if we define
\begin{equation}\label{ceeb2}
\Ceeb^2=\frac{h_{\rm min}^3}{3\cdot 2^\beta \Cgr^{1/\eta} m h_{\rm max}^2}\,,
\end{equation}
so we have already found the desired contradiction and the thesis follows simply by taking $R_3=R_2$. Instead, if $\eta\beta>1$, keeping in mind that
\[
|F|\leq |G|\leq |B(x,r)| \leq \Cgr r^\eta
\]
by the growth condition, the estimate~(\ref{aboine}) is clearly false if $r$ is small enough, so we can find some $R_3\leq R_2$ such that the desired contradiction follows also in this case. Summarizing, the thesis follows if we prove the existence of a cluster $\E'$ satisfying~(\ref{competitor}) under the assumption that~(\ref{effnoisl}) does not hold and with $K$ large enough.\par

For every $0\leq j \leq m,\, j\neq i$, we define $\Gamma_j := \partial^* F \cap \partial^* E_j$. Since
\[
\partial^* G\cap E_i\subseteq \partial^* F\subseteq \partial^* E_i \cup \big(\partial^* G\cap E_i\big)\,,
\]
and $\haus^1$-a.e.\ point of $\partial^* E_i$ belongs also to $\partial^* E_j$ for exactly one $0\leq j\leq m,\, j\neq i$, we have
\begin{equation}\label{subdiv}
\partial^* F =\big(\partial^* G \cap E_i\big)\,\cup\, \bigcup_{\doppio{j\in \{0,\,1,\,\dots\,,\,m\}}{j\neq i}}\ \Gamma_j
\end{equation}
up to negligible sets, and the $m+1$ sets are essentially disjoint. Consider now the inequality
\begin{equation}\label{maybeyes}
\haus^1(\Gamma_0)\geq \bigg(1-\frac{h_{\rm min}}{3h_{\rm max}}\bigg)\, \haus^1(\partial^* F)\,.
\end{equation}
The definition of the competitor $\E'$ will depend on whether or not this inequality holds true. First of all, we assume that the inequality holds. In this case, we let $\E'$ be the cluster such that $E_i'=E_i\setminus F$ and $E_j'=E_j$ for every $1\leq j\leq m,\, j\neq i$, so that~(\ref{competitor}), left, is true. In order to compare $P(\E)$ and $P(\E')$, we notice that $\partial^* \E' = \big(\partial^* \E \setminus \Gamma_0\big) \cup \big(\partial^* G\cap E_i\big)$, and also that, for each $1\leq j\leq m,\, j\neq i$, the set $\Gamma_j$ is contained both in $\partial^*\E$ and in $\partial^* \E'$, but its contribution to the perimeter is different. In fact, in $\partial^*\E$ the set $\Gamma_j$ is a common boundary between two ``colored'' sets, namely, $E_i$ and $E_j$, while in $\partial^*\E'$ it is common boundary between a colored and a white set, that is, $E'_j$ and $E'_0$. Consequently, making use of~(\ref{maybeyes}), we can estimate
\[\begin{split}
P(\E')- P(\E)&\leq - h_{\rm min} \haus^1(\Gamma_0) + h_{\rm max} \haus^1(\partial^* F\setminus \Gamma_0)\\
&= h_{\rm max}\haus^1(\partial^* F) - \big(h_{\rm min}+h_{\rm max}\big) \haus^1(\Gamma_0)\\
&\leq \bigg(h_{\rm max} - (h_{\rm min}+h_{\rm max})\bigg(1-\frac{h_{\rm min}}{3h_{\rm max}}\bigg)\bigg)\haus^1(\partial^* F)\\
&= h_{\rm min}\bigg(-\frac 23 + \frac{h_{\rm min}}{3h_{\rm max}}\bigg)\haus^1(\partial^* F)
\leq - \frac{h_{\rm min}}3\, \haus^1(\partial^* F)\\
&\leq - \frac{h_{\rm min}^2}{3mh_{\rm max}}\, \haus^1(\partial^* F)\,,
\end{split}\]
which is~(\ref{competitor}), right. We have then found the searched competitor if~(\ref{maybeyes}) holds.\par

Let us then finally assume that~(\ref{maybeyes}) is false. As a consequence, using~(\ref{subdiv}) and the fact that~(\ref{effnoisl}) is false, we have
\[
\sum_{{\doppio{j\in \{1,\,\dots\,,\,m\}}{j\neq i}}} \haus^1(\Gamma_j) = \haus^1(\partial^* F) - \haus^1(\partial^* G\cap E_i) - \haus^1(\Gamma_0)
> \bigg(\frac{h_{\rm min}}{3h_{\rm max}}- \frac 1K\bigg)\, \haus^1(\partial^* F)\,.
\]
Calling then $\ell={\rm argmax} \big\{\haus^1(\Gamma_j),\, 1\leq j\leq m,\, j\neq i\big\}$, we have
\begin{equation}\label{goodell}
\haus^1(\Gamma_\ell)\geq \frac 1 {m-1}\, \sum_{{\doppio{j\in \{1,\,\dots\,,\,m\}}{j\neq i}}} \haus^1(\Gamma_j)
\geq \frac{h_{\rm min}}{3(m-1/2)h_{\rm max}}\, \haus^1(\partial^* F)\,,
\end{equation}
where the last inequality is true if $K$ is large enough. In this case, we define $\E'$ the cluster such that $E_i'=E_i\setminus F$, $E_\ell'=E_\ell\cup F$, and $E_j'=E_j$ for every $j \in \{1,\, \dots\,,\, m\}\setminus \{i,\, \ell\}$. Notice that~(\ref{competitor}), left, holds true, and this time
\[
\partial^* \E' = \big(\partial^* \E \setminus \Gamma_\ell\big) \cup \big(\partial^* G\cap E_i\big)\,.
\]
Moreover, for every $1\leq j\leq m,\, j\notin \{i,\,\ell\}$ the contribution of $\Gamma_j$ to $P(\E)$ and to $P(\E')$ is the same (here we are using that $i\neq 0$). As a consequence, we have
\[
P(\E') \leq P(\E) - h_{\rm min} \haus^1(\Gamma_\ell) + h_{\rm max} \haus^1(\partial^* G\cap E_i) \leq P(\E) - \frac{h_{\rm min}^2}{3mh_{\rm max}}\haus^1(\partial^* F)\,,
\]
where the last inequality is true by~(\ref{goodell}), by the fact that~(\ref{effnoisl}) is false, and up to possibly increase the value of $K$. We have then proved~(\ref{competitor}), right, and the proof is completed.
\end{proof}

The same result that we have just proved is also true for the case $i=0$ (and we call this case ``no-lakes'' instead of ``no-island''), but our proof above does not work. We now show the no-lakes result in a simplified case, namely, for ``holes'' whose boundary entirely belongs to a same $\partial^* E_\ell$, with $1\leq \ell\leq m$. Later on, we will show it in full generality.

\begin{lem}[No-lakes, part 1]\label{nolakes}
Let $B(x,r)\subseteq \textsf{D}$ be a ball with $r< R_3$, and let $1\leq \ell\leq m$. There is no set $F\subseteq B(x,r)\setminus E_\ell$, thus in particular no set $F\subseteq B(x,r)\cap E_0$, with $|F|>0$ and $\partial^* F \subseteq \partial^* E_\ell$.
\end{lem}
\begin{proof}
We argue as in Lemma~\ref{noisland}, but the situation is now much simpler. Assume the existence of $x,\,r,\,\ell$ and $F$ as in the claim, and define $\E'$ the cluster such that $E_\ell'=E_\ell\cup F$, and $E_j'=E_j\setminus F$ for every $1\leq j\leq m,\, j\neq \ell$. By assumption, $\partial^*\E'=\partial^*\E\setminus (\partial^*F\cup F)$, and in particular
\[
P(\E')\leq P(\E) - h_{\rm min} \haus^1(\partial^* F)\,.
\]
Since this inequality is stronger than~(\ref{competitor}), we conclude exactly as in Lemma~\ref{noisland}.
\end{proof}

As a consequence of the last two results, we can observe a first regularity property for the sets $E_i,\, 1\leq i\leq m$, mild but useful.

\begin{lem}\label{lem1reg}
For every ball $B(x,r)\subseteq \textsf{D}$ with $r<R_3$ and for every $1\leq i\leq m$, the set $E_i\cap B(x,r)$ is an open set (taking the set of points with density $1$ as representative). Moreover, for every connected component $F$ of $E_i\cap B(x,r)$, there exists an injective curve $\gamma:\S^1\to\R^2$ of finite length such that $\partial^* F=\partial F=\gamma(\S^1)$ up to $\haus^1$-negligible subsets, and $\haus^1\big(\partial F\cap \partial B(x,r)\big)>0$.
\end{lem}
It is important to notice that, in the above result, the connectedness of $F$ should be in principle meant in the measure theoretic sense, see the Appendix. However, an immediate consequence of the result itself is that the measure theoretical connected components are actually connected in the topological sense.
\begin{proof}[Proof of Lemma~\ref{lem1reg}]%\proofof{Lemma~\ref{lem1reg}}
We can assume that $|E_i\cap B(x,r)|>0$, since otherwise the result is emptily true. Let then $F$ be either the whole set $E_i\cap B(x,r)$ or one of its connected components. Since $F\subseteq E_i$, then~(\ref{intint}) holds with $G=F$, hence also~(\ref{intbor}) is true, i.e., $\haus^1(E_i \cap \partial^* F)>0$. Observing that $\partial^* F\subseteq \partial^* E_i \cup \partial B(x,r)$ and $\haus^1(E_i\cap\partial^* E_i)=0$, we deduce that
\[
0<\haus^1(E_i \cap\partial^* F)\leq\haus^1(\partial^* F\cap \partial B(x,r))\,.
\]
Therefore, keeping in mind Theorem~\ref{qmpor}, Lemma~\ref{porope} and Lemma~\ref{noislPCC}, all we have to do is to check that $F$ is quasi-minimal and has no holes (in the sense of Definition~\ref{defholes}).\par 

We start proving that $F$ has no holes. By contradiction, assume the existence of $U\subseteq\R^2\setminus F$ with $\haus^2(U)>0$ and such that $\partial^* F =\partial^* U \cup \partial^*(F\cup U)$, which implies that $U\subseteq B(x,r)$. Up to $\haus^1$-negligible subsets, all points of $\partial^* U$ have density $1/2$ with respect to $U$, and also with respect to $F$ since $\partial^* U \subseteq \partial^* F$. Since $U\cap F=\emptyset$, this implies that $\haus^1$-a.e.\ point of $\partial^* U$ has density $1$ with respect to $U\cup F$. And since $U\cup F\subseteq B(x,r)$, we deduce that $\partial^* U\cap\partial B(x,r)$ is $\haus^1$-negligible, which recalling that $\partial^* U \subseteq \partial^* F \subseteq \partial^* E_i \cup \partial B(x,r)$ gives
\begin{equation}\label{puinpei}
\partial^* U\subseteq \partial^* E_i \qquad \haus^1\hbox{-a.e.}\,.
\end{equation}
Observe now that $U$ does not intersect $F$, but it could still intersect $E_i$. Nevertheless,
\begin{equation}\label{eiuuno}
\haus^1\big(\partial^*(E_i\cap U) \cap \partial^* U \big)=0\,.
\end{equation}
Indeed, up to $\haus^1$-negligible subsets, points of $\partial^*(E_i \cap U)$ have density $1/2$ with respect to $E_i\cap U$, and points of $\partial^* U\subseteq \partial^* F$ have density $1/2$ with respect to $F\subseteq E_i\setminus U$, so $\haus^1$-a.e.\ point of $\partial^*(E_i\cap U) \cap \partial^* U$ has density $1$ with respect to $E_i$, and by~(\ref{puinpei}) this gives~(\ref{eiuuno}), which in particular implies that the set $V=U\setminus E_i$ satisfies $|V|>0$. Summarizing, $V$ is a subset of $B(x,r)\setminus E_i$, and by construction and~(\ref{puinpei}) we have $\partial^* V\subseteq \partial^* U \cup \partial^* E_i\subseteq \partial^* E_i$. The existence of such a set $V$ is excluded by Lemma~\ref{nolakes}, thus we have proved that $F$ has no holes.\par

Hence, to conclude the proof, we only have to deal with the quasi-minimality. Since $F\subseteq B(x,r)$, it is enough to take a ball $B(z,\rho)$ intersecting $F$ and with $\rho<R_3$. To prove the quasi-minimality we have to find a constant $C_{qm}$, only depending on $\textsf{D},\, h$ and $\E$, such that for every set $H$ with $F\Delta H \comp B(z,\rho)$ one has
\begin{equation}\label{quasimin}
\haus^1\big(\partial^* F \cap B(z,\rho)\big)\leq C_{qm} \haus^1\big(\partial^* H \cap B(z,\rho)\big)\,.
\end{equation}
Let us call $G_1=F\setminus H$ and $G_2=H\setminus F$, and notice that $G_1,\,G_2 \comp B(z,\rho)$. We clearly have
\begin{equation}\label{subdiv2}
\partial^* F \cap B(z,\rho) \subseteq \partial^* G_1 \cup \partial^* G_2\cup \big(\partial^* H\cap B(z,\rho)\big)\,.
\end{equation}
Applying Lemma~\ref{noisland} to the set $G_1$ we get
\begin{equation}\label{estG1}
\haus^1(\partial^* G_1) = \haus^1\big(\partial ^*(G_1 \cap E_i)\big) \leq K \haus^1\big(E_i \cap \partial^* G_1\big)\leq K \haus^1\big(\partial^* H \cap B(z,\rho)\big)
\end{equation}
where the first equality holds since $G_1 = G_1\cap E_i$, and the last inequality is true since $\partial^* G_1 \subseteq \big(\partial^* E_i \cup \partial^* H \big)\cap B(z,\rho)$, and then $\haus^1$-a.e. point of $\partial^* G_1\cap E_i$ cannot be in $\partial^* E_i$, so it must be in $\partial^* H\cap B(z,\rho)$.\par

Let us now pass to consider $G_2$. Let us call $\widetilde \E$ the cluster such that $\widetilde E_i=E_i \cup G_2$ and $\widetilde E_j=E_j\setminus G_2$ for every $j\neq i$. As already observed while proving Lemma~\ref{nolakes}, we get a contradiction with the same argument of Lemma~\ref{noisland} if
\[
P(\widetilde\E) \leq P(\E) - \frac{h_{\rm min}}2\, \haus^1(\partial^* G_2)\,.
\]
In fact, in the proof of Lemma~\ref{nolakes} we were using the same inequality with $h_{\rm min}$ in place of $h_{\rm min}/2$, but both work since both inequalities are stronger than~(\ref{competitor}). Therefore, we know that
\begin{equation}\label{wkt}
P(\widetilde\E) > P(\E) -\frac{h_{\rm min}}2\, \haus^1(\partial^* G_2)\,.
\end{equation}
Let us now observe that
\begin{equation}\label{qs2}
\big(\partial^* G_2 \cap \partial^* F\big) \cap \partial^* \widetilde \E=\emptyset\,.
\end{equation}
Indeed, $\haus^1$-a.e.\ $y\in \partial^* G_2\cap\partial^* F$ has density $1/2$ with respect to $F$. Moreover, it has density $1/2$ with respect to $G_2$, which does not intersect $F$, so density $1$ with respect to $F\cup G_2\subseteq \widetilde E_i$. Thus, $y\notin \partial^*\widetilde\E$ and~(\ref{qs2}) is established. Moreover, $\partial^*\widetilde\E\setminus\partial^*\E\subseteq \partial^* G_2$, which by~(\ref{qs2}) becomes
\[
\partial^*\widetilde\E\setminus\partial^*\E\subseteq \partial^* G_2\setminus \partial^*F\,.
\]
By~(\ref{wkt}) we obtain the estimate
\[\begin{split}
\frac{h_{\rm min}}2\, \haus^1(\partial^* G_2)&> P(\E)-P(\widetilde\E)
\geq h_{\rm min} \haus^1\big(\partial^* G_2 \cap \partial^* F\big)-h_{\rm max} \haus^1\big(\partial^* G_2 \setminus \partial^* F\big)\\
&=h_{\rm min} \haus^1\big(\partial^* G_2\big)-(h_{\rm min}+h_{\rm max}) \haus^1\big(\partial^* G_2 \setminus \partial^* F\big) \\
&\geq h_{\rm min} \haus^1\big(\partial^* G_2\big)-(h_{\rm min}+h_{\rm max}) \haus^1\big(\partial^* H \cap B(z,\rho)\big) \,,
\end{split}\]
which can be rewritten as
\[
\haus^1(\partial^* G_2)< \frac{2(h_{\rm min}+h_{\rm max})}{h_{\rm min}}\, \haus^1\big(\partial^* H \cap B(z,\rho)\big)\,.
\]
Putting this inequality together with~(\ref{estG1}), we obtain~(\ref{quasimin}) thanks to~(\ref{subdiv2}).
\end{proof}

Notice that, as a consequence of the above regularity result, for each ball $B(x,r)\subseteq \textsf{D}$ with $r< R_3$, and each $1\leq i \leq m$, the boundary of $E_i\cap B(x,r)$ is done by a countable union of closed, injective curves. Since all these curves have to reach $\partial B(x,r)$, they are actually finitely many if Vol'pert Theorem~\ref{volpert} holds for $B(x,r)$ and $\haus^0(\partial^* \E \cap \partial B(x,r))<+\infty$, which is true for almost each $r>0$.\par

Observe now that, when two sets of the minimal cluster have some common boundary, then two of these curves have some intersection. We show now that these intersections between different curves behave not too crazily.

\begin{lem}\label{noncrazyint}
There exists $R_4\leq R_3$ such that the following holds. Let $B(x,r)\subseteq \textsf{D}$ be a ball with $r< R_4$, let $\gamma_1,\,\gamma_2:\S^1\to \overline{B(x,r)}$ be two curves as in Lemma~\ref{lem1reg}, not necessarily different, and let $\tau_1,\,\tau_2:[0,1]\to B(x,r)$ be two injective subpaths of $\gamma_1$ and $\gamma_2$ such that
\begin{align*}
\tau_1(0)=\tau_2(1)\,, && \tau_1(1)=\tau_2(0)\,.
\end{align*}
Then the paths $\tau_1$ and $\tau_2$ coincide, that is, $\tau_1((0,1))=\tau_2((0,1))$. More in general, if
\begin{align*}
\tau_1\big((0,1)\big)\cap \tau_2\big((0,1)\big) = \emptyset\,, && \tau_1(0)=\tau_2(1)\,,
\end{align*}
then
\begin{equation}\label{thesisnci}
\big|\tau_1(0)-\tau_1(1)\big| \leq C_1 |\tau_1(1)-\tau_2(0)|
\end{equation}
for some constant $C_1> 1$ depending only on $\textsf{D},\,h$ and $\E$.
\end{lem}
Before giving the proof of this result, we briefly explain its meaning, also with the aid of Figure~\ref{Fignoncrazy}.
\begin{figure}[htbp]
\begin{tikzpicture}[>=>>>,smooth cycle]
\filldraw[fill=red!35!white, draw=black, line width=.75pt] plot [tension=1] coordinates {(-2,0) (-3.5,1) (-2.5,2.5) (-0.5,2.5) (3,2.5) (3.5,0) (1.5,.5)};
\filldraw[fill=blue!35!white, draw=black, line width=.75pt] plot [tension=1] coordinates {(-2,0) (-3.5,-2) (-2.5,-3) (3,-3) (3.5,-3) (2.5,-1.5) (1.5,0.3) (0,-2)};
\filldraw[fill=white, draw=white] (-4,0) -- (-3,0) arc(180:0:3) -- (4,0) -- (4,3.5) -- (-4,3.5) -- cycle;
\filldraw[fill=white, draw=white] (-4,0) -- (-3,0) arc(-180:-0:3) -- (4,0) -- (4,-3.5) -- (-4,-3.5) -- cycle;
\fill (-2,0) circle (2pt);
\fill (1.5,.5) circle (2pt);
\fill (1.5,0.3) circle (2pt);
\draw[line width=1pt] (0,0) circle (3);
\draw (0,1.5) node {$E_j$};
\draw (-1.5,-1.5) node {$E_k$};
\draw (-2,0) node[anchor=south] {$P$};
\draw (1.5,0.6) node[anchor=south] {$Q_1$};
\draw (1.5,0.2) node[anchor=north] {$Q_2$};
\draw[->] (-.2,0.4) -- (0.45,0.55);
\draw (-.5,0.2) node[anchor=south] {$\tau_1$};
\draw[->] (.55,-1.2) -- (0.25,-1.8);
\draw (0,-1.9) node[anchor=south] {$\tau_2$};
\end{tikzpicture}
\caption{The curves $\tau_1$ and $\tau_2$ and the points $P,\, Q_1$ and $Q_2$ in Lemma~\ref{noncrazyint}.}\label{Fignoncrazy}
\end{figure}
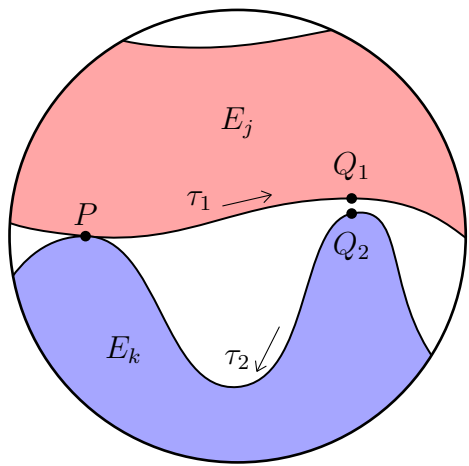
Let $\gamma_1$ and $\gamma_2$ be two curves as in Lemma~\ref{lem1reg}, and let us select two subpaths, $\tau_1$ of $\gamma_1$ and $\tau_2$ of $\gamma_2$, both entirely contained in the interior of the ball $B(x,r)$. The first part of the claim says that if the two paths have the same endpoints, then they have to coincide. In other words, if two curves $\gamma_1$ and $\gamma_2$ have two points in common, then they remain together between them. The second part of the claim considers a more general situation, namely, when $\tau_1$ and $\tau_2$ have disjoint interiors and one common endpoint, called $P$ in the figure (in particular, $\tau_1$ and $\tau_2$ could be consecutive subpaths of a same curve $\gamma_1=\gamma_2$). The other endpoints are called respectively $Q_1$ and $Q_2$. The inequality~(\ref{thesisnci}) then says that $Q_1$ and $Q_2$ cannot be too close, with respect to the distance between $P$ and $Q_1$. We have then to exclude the situation depicted in the figure, where $Q_1$ and $Q_2$ are very close to each other.\par
Notice that the second part of the statement implies the first one. Indeed, if $Q_1=Q_2$ then~(\ref{thesisnci}) gives $P=Q_1$, which is impossible since $\tau_1$ is an injective path. This means that if $\tau_1$ and $\tau_2$ have disjoint interiors then they cannot have the same endpoints. And in turn, this implies the first part of the statement, since two different paths with the same endpoints clearly admit two subpaths with the same endpoints and disjoint interiors.

\begin{proof}[Proof of Lemma~\ref{noncrazyint}]%\proofof{Lemma~\ref{noncrazyint}}
Thanks to the above discussion, we only prove the second part of the statement. Let $R_4\leq R_3$ be a small constant, which will be precised later. For simplicity of notation, as in Figure~\ref{Fignoncrazy} we set
\begin{align*}
P=\tau_1(0)=\tau_2(1)\,, && Q_1=\tau_1(1)\,, && Q_2 = \tau_2(0)\,, && d=|Q_1-Q_2|\,,
\end{align*}
so that~(\ref{thesisnci}) can be rewritten as $|P-Q_1|\leq C_1 d$. We limit ourselves to show the second part of the thesis, that is, that~(\ref{thesisnci}) holds if $\tau_1$ and $\tau_2$ have disjoint interiors. Indeed, this implies that in the case of disjoint interiors it is impossible that $Q_1=Q_2$. And as a consequence, if $Q_1=Q_2$ then $\tau_1$ and $\tau_2$ have to coincide, because otherwise there are further subpaths of $\tau_1$ and $\tau_2$ with disjoint interiors and the same endpoints, which has been excluded.

For brevity, and with a small abuse of notation, the image of a path $\tau:[0,1]\to\R^2$ will be sometimes denoted as $\tau$ instead of $\tau([0,1])$. We first assume that, as in the figure,
\begin{equation}\label{justforasec}
\hbox{\it the interior of the segment $Q_1Q_2$ does not intersect $\tau_1\cup \tau_2$}\,,
\end{equation}
at the end we will easily remove this assumption. Putting together $\tau_1,$ the segment $Q_1Q_2$, and $\tau_2$, we obtain then an injective, closed path in $B(x,r)$, which encloses a closed region that we call $G$. Without loss of generality we assume that this path is percurred clockwise, as in the figure. We also call $1\leq j,\,k\leq m$ the two indices such that the set enclosed by $\gamma_1$ (resp., $\gamma_2$) is contained in $E_j$ (resp., $E_k$). Notice that $j$ and $k$ are not necessarily different, in particular if $\gamma_1=\gamma_2$ then of course $j=k$. Observe that $\haus^1$-a.e.\ point of $\tau_1\cup\tau_2$ has density $1/2$ with respect to $E_j$ or $E_k$, hence it cannot have density $1$ with respect to any of the sets of the cluster. As a consequence, for every $1\leq i\leq m$, regardless whether or not $i$ coincides with one between $j$ and $k$, applying~(\ref{effnoisl}) of Lemma~\ref{noisland} we get
\begin{equation}\label{gammasmall}
\haus^1(\partial^* (E_i\cap G)\cap (\tau_1\cup\tau_2)) \leq \haus^1(\partial^*(E_i\cap G)) \leq K \haus^1(E_i\cap \partial^* G)=K \haus^1(E_i\cap Q_1Q_2)\,.
\end{equation}
We can now easily reduce ourselves to the case when
\begin{align}\label{notinside}
\haus^1\big(\tau_1 \cap \partial^*(E_j\cap G)\big)=0\,, && \haus^1\big(\tau_2 \cap \partial^*(E_k\cap G)\big)=0\,.
\end{align}
Indeed, the path $\tau_1$ is part of the boundary of a connected component of $E_j\cap B(x,r)$, and then $E_j$ is ``on one side of $\tau_1$''. In other words, $\haus^1$-a.e.\ point of $\tau_1$ has density $1/2$ with respect to $E_j$ and, up to $\haus^1$-negligible sets, either all points of $\tau_1$ have density $1/2$ with respect to $E_j\setminus G$, or they all have density $1/2$ with respect to $E_j\cap G$. The first case corresponds to the left assumption in~(\ref{notinside}), and it is the one depicted in Figure~\ref{Fignoncrazy}. As a consequence, if the left property of~(\ref{notinside}) fails, then by~(\ref{gammasmall}) with $i=j$ we have
\[
|\tau_1(0)-\tau_1(1)| \leq \haus^1(\tau_1) \leq \haus^1\big(\partial^* (E_j \cap G)\cap (\tau_1\cup\tau_2)\big) \leq K |Q_1-Q_2| = K |\tau_1(1)-\tau_2(0)|\,,
\]
then~(\ref{thesisnci}) is already proved with $C_1=K$. Similarly, if the right property of~(\ref{notinside}) fails, then
\[
|\tau_1(0)-\tau_1(1)| =|\tau_2(1)-\tau_1(1)|
\leq \haus^1(\tau_2)+ |\tau_1(1)-\tau_2(0)| \leq (K+1)|\tau_1(1)-\tau_2(0)|\,,
\]
hence we have again the validity of~(\ref{thesisnci}) with $C_1=K+1$. Therefore, from now on and without loss of generality we assume that~(\ref{notinside}) is true.\par
Putting together all the bounds~(\ref{gammasmall}) varying $1\leq i\leq m$, we have
\begin{align}\label{Gammanotsolong}
\haus^1(\Gamma) \leq Kd\,, && \hbox{where} && \Gamma=\bigcup\nolimits_{i=1}^m \partial^* (E_i\cap G) \cap (\tau_1\cup \tau_2)\,.
\end{align}
This bound gives an important information. Namely, if $d\ll |Q_1-P|$, then $\Gamma$ is only a very small portion of $\tau_1\cup\tau_2$. In other words, it is possible that some curves $\gamma_h$ as in Lemma~\ref{lem1reg} enter in $G$, at least if $Q_1\neq Q_2$ (such curves are not shown in Figure~\ref{Fignoncrazy} to keep the figure clear). Nevertheless, these curves cover only a small part of $\tau_1\cup\tau_2$. As a consequence, the greatest part of $\tau_1\cup\tau_2$ does not belong to $\Gamma$, and then by construction it is done by points which belong to $\partial^* E_0$.\par

We call now $\overline h:\R^2\to\R^+$ the function given by $\overline h(v)=h(x,v)$ for every $v\in\R^2$ (keep in mind that $x$ is the center of the disk), we set $\h:\R^2\to\R$ as $\h(\nu)=\bar h(\hat \nu)$ where $\hat \nu$ is the angle obtained by rotating $\nu$ of $90^\circ$ clockwise, and we denote by $\len$ the length of curves given by~(\ref{lengthcurves}). Applying then Corollary~\ref{helpnolakes} to the paths $\tau_1\cup Q_1Q_2$ and $\tau_2$, which are two injective paths, disjoint except for the common endpoints (i.e., $P$ and $Q_2$), we get an injective path $\tau:[0,1]\to G$ connecting $P$ and $Q_2$ and satisfying~(\ref{half}), which reads as
\begin{equation}\label{halfthistime}
\len(\tau) + \len(\hat\tau) \leq \len(\tau_1\cup Q_1Q_2)+\len(\tau_2)
\leq \len(\tau_1)+\len(\tau_2)+h_{\rm max}d\,.
\end{equation}
We call $G_1$ and $G_2$ the parts of $G$ enclosed by $\tau_1\cup Q_1Q_2\cup \hat\tau$ and by $\tau\cup\tau_2$ respectively, and we are in a position to define the competitor $\E'$. In fact, we set $E_i'= E_i\setminus G$ for every $i\notin \{j,\,k\}$. Moreover, if $j\neq k$ then we set $E_j'=E_j \cup G_1$ and $E'_k=E_k \cup G_2$, while if $j=k$ we set $E'_j = E_j \cup G$. Notice that, in this way, and keeping in mind~(\ref{notinside}), we remove from $\partial^*\E$ both $\tau_1$ and $\tau_2$, as well as the part of $\partial^*\E$ which was in the interior of $G$, if any. Conversely, we add a part of the segment $Q_1Q_2$ and, if $j\neq k$, the path $\tau$ (that could contain parts of $\tau_1$ and $\tau_2$, which would then be re-added to the boundary). In particular, notice that if $j\neq k$ then the path $\tau$ is added to $\partial^*\E'$, and it is common boundary between $E_j'$ and $E_k'$. Instead, if $j=k$, then $\tau$ is simply not added to $\partial^*\E'$.\par

Recalling that $\omega$ is the modulus of continuity of $h$ in the first variable inside $\textsf{D}$, we select $R_4$ so small that $\omega(R_4)< h_{\rm min}/6$, thus for every $y\in B(x,r)$ and $\nu\in\S^1$ we have
\[
\frac 56\, \overline h(\nu) \leq \overline h(\nu)- \frac{h_{\rm min}}6<\overline h(\nu)- \omega(r) \leq h(y,\nu) \leq \overline h(\nu) + \omega(r)< \overline h(\nu) + \frac{h_{\rm min}}6 \leq \frac 76\,\overline h(\nu)\,.
\]
As a consequence, keeping in mind that $\tau_1\cup Q_1Q_2\cup\tau_2$ is a clockwise and closed path, and that points of $\tau_1\setminus \Gamma$ (resp., $\tau_2\setminus\Gamma$) belong to $\partial^* E_j\cap\partial^* E_0$ (resp., $\partial^* E_k\cap \partial^* E_0$), we use~(\ref{halfthistime}) and~(\ref{Gammanotsolong}) to get
\begin{equation}\label{oneesti}\begin{split}
P(\E')-P(\E)&\leq \frac 76\,\frac{\len(\tau)+\len(\hat\tau)}2 - \frac 56\, \big(\len(\tau_1) +\len(\tau_2)\big)+ h_{\rm max} \big(\haus^1(\Gamma)+d\big)\\
&\leq -\frac 14\, \big(\len(\tau_1) +\len(\tau_2)\big)+ h_{\rm max} \bigg(\haus^1(\Gamma)+\frac{19}{12}\,d\bigg)\\
& \leq -\frac {h_{\rm min}}4\, \haus^1(\tau_1\cup\tau_2) + h_{\rm max} (K+2)\,d\,.
\end{split}\end{equation}
Let us call again $\eps=|\E|-|\E'|$, so that by construction, by the (Euclidean) isoperimetric inequality we have
\[
|\eps|^\beta \leq (2 |G|)^\beta \leq \frac 1{(2\pi)^\beta} \, \big(\haus^1(\tau_1 \cup\tau_2)+d\big)^{2\beta}\leq 
\haus^1(\tau_1 \cup\tau_2)^{2\beta}+d^{2\beta}
\leq \haus^1(\tau_1 \cup\tau_2)^{2\beta}+d\,.
\]
Notice that, in the last inequality, we have used the fact that $\beta\geq 1/2$, which is true since $\eta\beta\geq 1$ and the $\eta$-growth condition implies that $\eta\leq 2$, and the fact that $d<1$, which is obvious as soon as $R_4<1/2$. We can now apply Lemma~\ref{labello} to find a cluster $\E''$ with $|\E''|=|\E|$ and such that, also by~(\ref{oneesti}),
\[\begin{split}
P(\E'') &\leq P(\E') + \Ceeb |\eps|^\beta\\
&\leq P(\E) -\frac{h_{\rm min}}4\,\haus^1(\tau_1 \cup\tau_2)+ \Big(h_{\rm max} (K+2)+ \Ceeb\Big)\,d + \Ceeb \haus^1(\tau_1 \cup\tau_2)^{2\beta}\,.
\end{split}\]
The optimality of $\E$ implies then that
\[
\frac{h_{\rm min}}4\,\haus^1(\tau_1 \cup\tau_2)\leq \Big(h_{\rm max} (K+2)+ \Ceeb\Big)\,d + \Ceeb \haus^1(\tau_1 \cup\tau_2)^{2\beta}\,.
\]
We now claim that
\begin{equation}\label{sinceopt}
\frac{h_{\rm min}}5\,\haus^1(\tau_1 \cup\tau_2)\leq \Big(h_{\rm max} (K+2)+ \Ceeb\Big)\,d\,.
\end{equation}
This inequality follows from the previous one as soon as we show that
\[
\Ceeb \haus^1(\tau_1 \cup\tau_2)^{2\beta} \leq \frac{h_{\rm min}}{20}\,\haus^1(\tau_1 \cup\tau_2)\,.
\]
If $\beta>1/2$, this last bound clearly holds as soon as $\haus^1(\tau_1 \cup\tau_2)$ is small enough, and in turn, since $\tau_1\cup\tau_2\subseteq\partial^*\E$, this is true by Lemma~\ref{13/2}, up to possibly decreasing $R_4$. Instead, if $\beta=1/2$, and then necessarily $\eta=2$ and $\eta\beta=1$, the bound holds by~(\ref{Cperissmall}) as soon as we define
\begin{equation}\label{ceeb3}
\Ceeb^3=\frac{h_{\rm min}}{20}\,.
\end{equation}
We have then shown the validity of~(\ref{sinceopt}), which implies~(\ref{thesisnci}) since $|\tau_1(1)-\tau_2(0)|=d$ and $|\tau_1(0)-\tau_1(1)|=|P-Q_1|\leq \haus^1(\tau_1)$. Summarizing, we have concluded the proof under the additional assumption~(\ref{justforasec}).\par

We are then only left to consider the case when the interiors of $\tau_1$ and $\tau_2$ are disjoint and the open segment $Q_1Q_2$ intersects points of $\tau_1\cup\tau_2$. In this case, since $\tau_1([0,1])$ and $\tau_2([0,1])$ are closed, we can define $\widetilde Q_1$ and $\widetilde Q_2$ two points with minimal distance among the pairs in $\tau_1\cap Q_1Q_2$ and $\tau_2\cap Q_1Q_2$. We can then call $\tilde \tau_1$ (resp., $\tilde\tau_2$) the subpath of $\tau_1$ (resp., $\tau_2$) between $P$ and $\widetilde Q_1$ (resp., between $\widetilde Q_2$ and $P$). Since by minimality the open segment $\widetilde Q_1\widetilde Q_2$ does not intersect $\tilde\tau_1\cup\tilde\tau_2 \subseteq \tau_1\cup\tau_2$, we know the validity of~(\ref{thesisnci}) for $\tilde\tau_1$ and $\tilde\tau_2$. Therefore, we easily deduce the validity also for $\tau_1$ and $\tau_2$, since using the minimality of $\widetilde Q_1$ and $\widetilde Q_2$ we get
\[\begin{split}
\big|\tau_1(0)-\tau_1(1)\big| &= |P-Q_1|
\leq |P-\widetilde Q_1| + |\widetilde Q_1 - Q_1|
\leq C_1 |\widetilde Q_1 - \widetilde Q_2 | + |\widetilde Q_1 - Q_1|\\
&\leq C_1 |\widetilde Q_1 - Q_2 | + |\widetilde Q_1 - Q_1|
\leq 
 C_1 |Q_1-Q_2| = C_1 |\tau_1(1)-\tau_2(0)|\,.
\end{split}\]
\end{proof}

An immediate corollary of the above result is the following one, which generalizes a part of the claim of Lemma~\ref{lem1reg} also to $E_0$.

\begin{cor}\label{correg}
Let $B(x,r)\subseteq \textsf{D}$ be a ball with $r<R_4$ and $0<\haus^0(\partial^*\E\cap\partial B(x,r))<\infty$, and let $F$ be a connected component of $E_0\cap B(x,r)$. There exists an injective map $\gamma:\S^1\to\R^2$ of finite length such that $\partial^* F=\partial F=\gamma(\S^1)$ up to $\haus^1$-negligible subsets.
\end{cor}
\begin{proof}
Let us take the collection $\{\gamma_j\}$ of all the curves given by Lemma~\ref{lem1reg}, which parametrize the boundaries of all the connected components of the sets $E_i\cap B(x,r)$ with $1\leq i \leq m$. Since all these curves have to reach $\partial B(x,r)$, by assumption and also using Vol'pert Theorem~\ref{volpert} we deduce that they are finitely many. For each $j$, the set $\gamma_j\cap B(x,r)$ is a disjoint, finite union if injective subpaths $\gamma_j^h:(0,1)\to B(x,r)$, each of them having both endpoints in $\partial B(x,r)$. Hence, altogether there are only finitely many paths $\gamma_j^h$. Notice that $\partial^*\E\cap B(x,r)$ coincides with the union of these paths. Moreover, $\haus^1$-a.e.\ point $z\in\partial^*\E\cap B(x,r)$ belongs to exactly two boundaries $\partial^* E_i \cap B(x,r)$ with $0\leq i \leq m$. If both boundaries correspond to indices $i\neq 0$, then $z$ belongs to exactly two of the paths $\gamma_j^h$, and $z\notin \partial^* E_0$. Conversely, if one of the two boundaries corresponds to $i=0$, thus in particular $z\in \partial^* E_0$, then $z$ belongs to exactly one of the paths $\gamma_j^h$. Therefore, $\partial^* E_0\cap B(x,r)$ coincides $\haus^1$-a.e.\ with the points belonging to exactly one of the paths $\gamma_j^h$.\par

Notice now that, by construction and thanks to Lemma~\ref{noncrazyint}, the intersection between any two of the paths $\gamma_j^h$ is either empty, or a single point, or a common closed subpath. As a consequence, $\partial^* E_0\cap B(x,r)$ coincides $\haus^1$-a.e.\ with the union of finitely many injective curves $\tau_k:(0,1)\to B(x,r)$ of finite length, which are the parts of the paths $\gamma_j^h$ which do not belong to any other of the paths. By construction, the curves $\tau_k$ are pairwise disjoint. As a consequence, every endpoint of each curve $\tau_k$ must be also endpoint of some of the other curves, and then the claim immediately follows.
\end{proof}

Notice that, by Lemma~\ref{lem1reg} and Corollary~\ref{correg}, we now know that the sets $\partial\E$ and $\partial^*\E$ coincide in $\textsf{D}$ up to $\haus^1$-negligible subsets. We can now show the existence of arbitrarily small circles around each point of $\textsf{D}$ with at most three points of $\partial\E$.

\begin{lem}[At most three points]\label{atmost3}
There exist $R_5\leq R_4$ and $C_2>2$, only depending on $h$, $\textsf{D}$ and $\E$, such that for every ball $B(x,r)\subseteq \textsf{D}$ with $r< R_5$ there is $r/C_2 \leq \rho\leq r$ such that
\begin{equation}\label{thencons}
\#\Big(\partial\E \cap \partial B(x,\rho) \Big) \leq 3\,.
\end{equation}
\end{lem}
\begin{proof}
As already done several times, we start by defining $\overline h:\R^2\to\R^+$ as $\overline h(v)=h(x,v)$, and $\overline P$ the perimeter obtained by substituting $h$ with $\overline h$ in~(\ref{weightedvolper}). Moreover, we call again $\h:\R^2\to \R^+$ the function such that $\h(\nu)=\overline h(\hat \nu)$ where $\hat \nu$ is the angle obtained by rotating $\nu$ of $90^\circ$ clockwise, and we denote by $\len$ the length of curves given by~(\ref{lengthcurves}).\par

Let now $M$ be the smallest integer strictly larger than $1+7\,\frac{h_{\rm max}}{h_{\rm min}}$. Let $R_5\leq R_4$ and $C_2>2$ be two constants to be specified later, and let us take a ball $B(x,r)\subseteq \textsf{D}$ with $r< R_5$. By Lemma~\ref{13/2}, there exists a radius $r/M \leq r_0\leq r$ such that $\partial B(x,r_0)\cap \partial\E$ is made by at most $M$ points. More precisely, $r_0$ is a Lebesgue point of the function $\rho\mapsto\haus^0(\partial B(x,\rho)\cap \partial\E)$, and the value of this function at $\rho=r_0$ is at most $M$.\par

By Lemma~\ref{lem1reg}, for every $1\leq i\leq m$ the boundary of each connected component of the set $E_i\cap B(x,r_0)$ is a closed curve of finite length intersecting $\partial B(x,r_0)$. Intersecting these curves with the interior of the ball, we have finitely many paths of finite length inside $B(x,r_0)$ with both endpoints on $\partial B(x,r_0)$. Let us denote by $\psi_j:(0,1)\to B(x,r_0)$ these paths. We observe that they are at most $M$. Indeed, every path has two endpoints in $\partial B(x,r_0)\cap \partial\E$, and the fact that $\rho$ is a Lebesgue point of $\rho\mapsto \haus^0(\partial\E\cap \partial B(x,\rho))$ implies that each point of $\partial B(x,r_0)\cap \partial\E$ can be endpoint of at most two paths. For every $j$, we call
\[
\rho_j = \min \big\{|\psi_j(t)-x|,\, 0<t< 1\big\}\in [0,r_0)\,.
\]
Remember that different paths may have parts in common, actually $\haus^1$-a.e.\ point of $\partial\E\setminus\partial E_0$ belongs to two different paths, as already observed in Corollary~\ref{correg} (where we called the paths $\gamma_j^h$ instead of $\psi_j$). However, by Lemma~\ref{noncrazyint}, the intersection between any two of the paths is either empty, or a common closed subpath, which might be a single point. Let then $\psi_k$ and $\psi_l$ be two paths such that $\psi_k\cap\psi_l\neq \emptyset$, let us call $\gamma:[0,1]\to\R^2$ a parametrization of $\psi_k\cap\psi_l$, either injective or constant, and set $\rho^\pm_{k,l}\in [0,r_0)$ as
\begin{align*}
\rho^-_{k,l} = |\gamma(0)-x|\,, &&
\rho^+_{k,l} =|\gamma(1)-x|\,.
\end{align*}
Notice that the orientation of $\gamma$ is not univoquely determined, but the pair $\rho^\pm_{k,l}$ is well-defined. Let now $H>2C_1+1$ be a large constant, only depending on $\textsf{D},\,h$ and $\E$ and to be specified later. Since the paths $\psi_j$ are at most $M$, the constants $\rho_j$ and $\rho_{k,l}^\pm$ are at most $M^2$. As a consequence, we can fix
\begin{equation}\label{bigr1}
\frac{r_0}{H^{2M^2+3}} < r_1 < \frac {r_0}H
\end{equation}
such that for every $j$, and for every pair $(k,l)$ such that $\psi_k\cap\psi_l\neq\emptyset$ we have
\begin{align}\label{nothap}
\rho_j \notin \bigg[ \frac{r_1}H, Hr_1\bigg]\,, &&
\rho_{k,l}^\pm \notin \bigg[ \frac{r_1}H,H r_1\bigg]\,.
\end{align}
Roughly speaking, this means that nothing ``special'' happens for a while around the circle $\partial B(x,r_1)$. More precisely, every path which enters in the ball $B(x,Hr_1)$ has to enter also in the much smaller ball $B(x,r_1/H)$, and the intersection between every two paths has to start/end either outside of the ball $B(x,Hr_1)$, or inside the small ball $B(x,r_1/H)$. Up to renumbering, we assume that $\psi_j\cap B(x,Hr_1)\neq \emptyset$ if and only if $1\leq j \leq M^-$ for some $M^-\leq M$.\par

Fix a path $\psi_j$ with $1\leq j\leq M^-$. By~(\ref{nothap}), this implies that also $\psi_j \cap B(x,r_1/H) \neq \emptyset$. We can then set $0<t^0_j < t^1_j< t^2_j <t^3_j <1$ as
\begin{align*}
t^0_j &= \inf \big\{t\in [0,1]:\, |\psi_j(t)-x| \leq r_1 \big\}\,, &
t^1_j &= \inf \Big\{t\in [0,1]:\, |\psi_j(t)-x| \leq \frac{r_1}H \Big\}\,\\
t^2_j &= \sup \Big\{t\in [0,1]:\, |\psi_j(t)-x| \leq \frac{r_1}H \Big\}\,, &
t^3_j &= \sup \big\{t\in [0,1]:\, |\psi_j(t)-x| \leq r_1 \big\}\,.
\end{align*}
We can easily notice that
\begin{equation}\label{entesc}
|\psi_j(t) - x | < H r_1 \quad \forall\, t \in [t^0_j ,t^3_j]\,.
\end{equation}
Indeed, assume that $|\psi_j(t)-x|\geq H r_1$ for some $t^0_j \leq t \leq t^3_j$. Then, we apply Lemma~\ref{noncrazyint}, being $\tau_1$ and $\tau_2$ the restrictions of $\psi_j$ to $[t,t^3_j]$ and $[t^0_j,t]$ respectively, so that~(\ref{thesisnci}) implies
\[
\big|\tau_1(0)-\tau_1(1)\big| \leq C_1 |\tau_1(1)-\tau_2(0)|
= C_1 |\psi_j(t^3_j)-\psi_j(t^0_j)| \leq 2 C_1 r_1\,,
\]
and this gives a contradiction since
\[
\big|\tau_1(0)-\tau_1(1)\big|=|\psi_j(t)-\psi_j(t^3_j)|\geq (H-1)r_1 
\]
and $H>2C_1+1$. We call $\psi_j^{\rm int}$ the ``interior part'' of $\psi_j$, that is, the restriction of $\psi_j$ to $[t^0_j,t^3_j]$, and we call also $\psi_j^{\rm int,1}$ and $\psi_j^{\rm int,2}$ the restrictions of $\psi_j$ to $[t^0_j,t^1_j]$ and to $[t^2_j,t^3_j]$ respectively, which are two disjoint subpaths of $\psi_j^{\rm int}$.\par

Now, keep in mind that $\haus^1$-a.e. point of $\partial\E$ belongs to the boundary of exactly two of the sets $E_i$, $0\leq i \leq m$. In particular, there exist an index $1\leq i\leq m$ such that $\psi_j\subseteq \partial E_i$, and there exists another index $0\leq i'\leq m$ with $i'\neq i$ such that $\psi_j(t^0_j)\in \partial E_{i'}$. We subdivide the indices $1,\,2,\, \cdots\,,\, M^-$ into the two subsets $\I^1_1$ and $\I^1_2$ by saying that $j\in \I^1_1$ if $i'\geq 1$, and $j\in \I^1_2$ if $i'=0$. Now, we claim that
\begin{equation}\label{remtog}
\psi_j\big( [t^0_j , t^1_j]\big) \subseteq \partial E_{i'}\,.
\end{equation}
To prove this property, we first assume that $j\in\I^1_1$, that is, $1\leq i' \leq m$. As a consequence, $\psi_j(t^0_j)$ is a point of some curve $\psi_{j'}$, and then $\psi_j \cap \psi_{j'}$ is a non-empty subpath of $\psi_j$, in particular $\psi_j\cap\psi_{j'}$ contains $\psi_j\big( [t^0_j,\bar t] \big)$ for some maximal value of $t^0_j \leq \bar t \leq 1$. By~(\ref{nothap}) we know that $|\psi_j(\bar t)-x|\notin [r_1/H,Hr_1]$, since $|\psi_j(\bar t)-x| \in \{\rho^-_{j,j'},\,\rho^+_{j,j'}\}$. We have then either that $|\psi_j(\bar t)-x|<r_1/H$, and then $\bar t\geq t^1_j$ by definition of $t^1_j$, or $|\psi_j(\bar t)-x|>Hr_1$, and then $\bar t\geq t^3_j$ by~(\ref{entesc}). In both cases, $\bar t\geq t^1_j$, and this shows~(\ref{remtog}) if $i'\geq 1$. In addition, this also shows that the point $\psi_j(t^0_j)$ is a ``special point'' also for $\psi_{j'}$, namely, it coincides with either $\psi_{j'}(t^0_{j'})$ or $\psi_{j'}(t^3_{j'})$.\par

Suppose now that $j\in \I^1_2$, i.e., $i'=0$, so that $\psi_j(t^0_j)$ does not belong to any $\psi_k$ with $k\neq j$. In particular, the point $\psi_j(t^0_j)$ belongs to a connected component of $E_0\cap B(x,r_0)$. As before, we have a maximal $\bar t\geq t^0_j$ such that $\psi_j\big([t^0_j,\bar t]\big)\subseteq \partial E_0$, and proving~(\ref{remtog}) in this case again reduces to showing that $\bar t\geq t^1_j$. By construction, either $\bar t=1$, and then~(\ref{remtog}) is already proved, or $\bar t<1$, and then $\psi_j(\bar t)\in \psi_{j'}$ for some $j'$, and in particular $|\psi_j(\bar t)-x|$ equals either $\rho^-_{j,j'}$ or $\rho^+_{j,j'}$. Exactly as before, this implies that $|\psi_j(\bar t)-x|\notin [r_1/H,Hr_1]$, and this proves $\bar t\geq t^1_j$. The property~(\ref{remtog}) is then proved. Of course, in the very same way, we have that $\psi_j(t^3_j)\in \partial E_i \cap \partial E_{i''}$, and that
\begin{equation}\label{remtog2}
\psi_j\big( [t^2_j , t^3_j]\big) \subseteq \partial E_{i''}\,.
\end{equation}
Moreover, we subdivide the indices $1,\,2,\,\cdots\,,\, M^-$ also into the two subsets $\I^2_1$ and $\I^2_2$, by saying that $j\in \I^2_1$ if $i''\geq 1$, and $j \in \I^2_2$ if $i''=0$.\par

We are now ready to define the competitor cluster $\E'$. In fact, for every $j$ we call $\widehat\psi_j^{\rm int, 1}$ the segment connecting $\psi_j(t^0_j)$ and $x$, and $\widehat\psi_j^{\rm int, 2}$ the segment connecting $x$ and $\psi_j(t^3_j)$. By construction, in particular keeping in mind~(\ref{remtog}) and~(\ref{remtog2}), we obtain that the set
\[
\partial\E \setminus \Big( \bigcup\nolimits_{j=1}^{M^-} \psi_j^{\rm int}\Big) \cup \Big( \bigcup\nolimits_{j=1}^{M^-} \widehat\psi_j^{\rm int, 1}\cup \widehat\psi_j^{\rm int, 2}\Big)
\]
is the boundary of a uniquely determined cluster, that we call $\E'$. In particular, we can observe that the ``colors'' of the boundaries of $\E'$ coincide with those of $\E$ in $B(x,Hr_1)\setminus B(x,r_1/H)$. More precisely, fix any $1\leq j \leq M^-$ and call $i,\,i'$ and $i''$ as before, so that $\psi_j^{\rm int,1}\subseteq\partial E_i \cap \partial E_{i'}$, and $\psi_j^{\rm int,2}\subseteq\partial E_i \cap \partial E_{i''}$. Then, also the path $\widehat\psi_j^{\rm int,1}$ is contained in $\partial E_i'\cap \partial E_{i'}'$, and $\widehat\psi_j^{\rm int,2}$ is contained in $\partial E_i'\cap \partial E_{i''}'$. And finally, by the definition~(\ref{PbE}) of the perimeter and by construction this implies that
\begin{equation}\label{2000}\begin{split}
\overline P(\E') &- \overline P(\E) \leq 
\sum_{j\in\I^1_1} \frac{\len(\widehat\psi_j^{\rm int,1})- \len(\psi_j^{\rm int,1})}2 +
\sum_{j\in\I^2_1} \frac{\len(\widehat\psi_j^{\rm int,2})- \len(\psi_j^{\rm int,2})}2\\
&
+\sum_{j\in\I^1_2} \Big(\len(\widehat\psi_j^{\rm int,1})- \len(\psi_j^{\rm int,1})\Big) +
\sum_{j\in\I^2_2} \Big(\len(\widehat\psi_j^{\rm int,2})- \len(\psi_j^{\rm int,2})\Big)\,.
\end{split}\end{equation}
We can easily estimate the terms of the above inequality. Indeed, for each $1\leq j\leq M^-$, $\widehat\psi_j^{\rm int,1}$ is the segment between $\psi_j(t^0_j)$ and $x$, while $\psi_j^{\rm int,1}$ is a path between $\psi_j(t^0_j)$ and some point $\psi_j(t^1_j)$ having distance $r_1/H$ from $x$, and then by Lemma~\ref{lemmagiorgio} we have that
\begin{align*}
\len(\widehat\psi_j^{\rm int,1})\leq \len(\psi_j^{\rm int,1})  + \frac{h_{\rm max}}H\, r_1 \,, &&
\len(\widehat\psi_j^{\rm int,2})&\leq \len(\psi_j^{\rm int,2})+ \frac{h_{\rm max}}H\, r_1 \,.
\end{align*}
Inserting these estimates in~(\ref{2000}), we obtain that
\begin{equation}\label{arrfra}
\overline P(\E')-\overline P(\E) \leq \frac{2 M h_{\rm max}}H\, r_1\,.
\end{equation}
We claim now that $\partial B(x,r_1)\cap \partial\E$ contains at most $3$ points. This will prove~(\ref{thencons}) with $\rho=r_1$, and recalling~(\ref{bigr1}) and the fact that $r/M \leq r_0 \leq r$ this will conclude the thesis with $C_2=M H^{2M^2+3}$. Suppose by contradiction that the claim is false, that is, $\partial B(x,r_1)\cap \partial\E$ contains at least $4$ points. Then, applying Proposition~\ref{90prop} to the cluster $\E'$ in the ball $B(x,r_1)$ with the distance $\overline h$, we obtain another cluster $\F$ which equals $\E'$ outside of $B(x,r_1)$ and such that $\overline P(\F) \leq \overline P(\E')- \delta r_1$, so that~(\ref{arrfra}) gives
\begin{equation}\label{prarr}
\overline P(\F) \leq \overline P(\E) - \frac \delta 2\, r_1
\end{equation}
as soon as $H > 4 M h_{\rm max}/\delta$. Keep in mind that, as observed in Remark~\ref{fixdelta}, the constant $\delta$ only depends on $h$ and $\textsf{D}$, but not on $x$ or $r$.\par

Since $\E$ and $\F$ coincide outside $B(x,Hr_1)$, we can estimate
\begin{equation}\label{ason}\begin{split}
\haus^1(\partial \F\cap B(x,H r_1)) &\leq \frac 1{h_{\rm min}}\, \overline P(\F; B(x,H r_1)) \leq \frac 1{h_{\rm min}}\, \overline P(\E; B(x,H r_1))\\
&\leq \frac {h_{\rm max}}{h_{\rm min}}\haus^1(\partial \E\cap B(x,H r_1))\,.
\end{split}\end{equation}
Let now $R_5\leq R_4$, only depending on $h,\,\textsf{D}$ and $\E$, be a constant such that
\[
\omega(R_5) < \frac{\delta h_{\rm min}^2}{28 H h_{\rm max}(h_{\rm min}+h_{\rm max})}\,.
\]
As a consequence, by Lemma~\ref{13/2}, (\ref{prarr}) and~(\ref{ason}) we have
\[\begin{split}
P(\F)-P(\E)&\leq \overline P(\F) - \overline P(\E) + \omega(r) \Big(\haus^1\big(\partial \E\cap B(x,Hr_1)\big)+\haus^1\big(\partial \F\cap B(x,H r_1)\big)\Big)\\
&\leq -\frac \delta 2\, r_1+ 7H\omega(r) \bigg( 1 + \frac {h_{\rm max}}{h_{\rm min}} \bigg) \, \frac{h_{\rm max}}{h_{\rm min}}\, r_1
\leq -\frac\delta 4\, r_1 \,.
\end{split}\]
Applying then Lemma~\ref{labello} we obtain a cluster $\E''$ with $|\E''|=|\E|$ such that
\[
P(\E'')\leq P(\F) + 2^\beta \Ceeb\Cgr^\beta (Hr_1)^{\eta\beta}
\leq P(\E)+ 2^\beta \Ceeb\Cgr^\beta H^{\eta\beta} r_1^{\eta\beta} - \frac \delta 4\, r_1\,.
\]
We can argue now as already done several times. Indeed, the contradiction $P(\E'')<P(\E)$, which concludes the proof, follows up to possibly further decrease $R_5$ if $\eta\beta>1$. Instead, if $\eta\beta=1$, it follows by~(\ref{Cperissmall}) as soon as we define
\begin{equation}\label{ceeb4}
\Ceeb^4=\frac\delta{2^{\beta+2} \Cgr^\beta H}\,.
\end{equation}
\end{proof}

Thanks to the above result, we can now show the ``no-lakes'' lemma in full generality.
\begin{lem}[No-lakes, general case]\label{nolakesgen}
There is $R_6\leq R_5$ such that, for every ball $B(x,r)\subseteq \textsf{D}$ with $r<R_6$, no connected component of $E_0$ can be compactly contained in $B(x,r/C_2)$, where $C_2$ is as in Lemma~\ref{atmost3}.
\end{lem}
\begin{proof}
Let $R_6\leq R_5$ be a constant to be specified later, let $B(x,r)\subseteq \textsf{D}$ and assume that $G\comp B(x,r/C_2)$ is the closure of a connected component of $E_0$. We can reduce ourselves to the case that
\begin{equation}\label{estidiam}
{\rm diam}(G) \geq \frac r{2C_2}\,.
\end{equation}
Indeed, otherwise let $x'$ be any point internal to $G$, and let $r'=C_2{\rm diam}(G)$. Then $r'\leq r/2\leq R_6$ and $G\comp B(x',r'/C_2)\subseteq B(x',r')\subseteq B(x,r)\subseteq \textsf{D}$, thus we can replace $B(x,r)$ with $B(x',r')$ for which the analogous of~(\ref{estidiam}) clearly holds. Hence, we assume without loss of generality that~(\ref{estidiam}) holds true.

Applying Lemma~\ref{atmost3}, we find $r/C_2\leq \rho\leq r$ such that $\partial\E\cap \partial B(x,\rho)$ has at most three points. Using Lemma~\ref{lem1reg} as already done in Corollary~\ref{correg} and Lemma~\ref{atmost3}, we find then at most three paths $\psi_j:(0,1)\to B(x,\rho)$, $j\in \{1,\,2,\,3\}$, of finite length and with $\psi_j(0),\, \psi_j(1)\in \partial B(x,\rho)$, whose images contain the whole $\partial\E\cap B(x,\rho)$, thus in particular $\partial G$. Since any two of the paths $\psi_j$ may intersect only in a connected subpath by Lemma~\ref{noncrazyint}, we deduce that the intersection of any $\psi_j$ with $\partial G$ is a connected path. As a consequence, $\partial G$ is the union of at most three connected pieces, each contained in one of the $\psi_j$. Actually, these pieces have to be exactly three. In fact, it cannot be one since the paths $\psi_j$, being injective, may not contain loops, and they cannot be two because otherwise two paths $\psi_j$ would have a non connected intersection.\par

As a consequence, the paths $\psi_j$ have to be exactly three, thus also $\partial B(x,\rho)\cap \partial\E$ contains exactly three points, and each of them is endpoint of two different paths $\psi_j$. This shows that $\partial E_0$ does not intersect $\partial B(x,\rho)$. Moreover, each two of the three pahts $\psi_1,\,\psi_2,\,\psi_3$ have a non-empty intersection, which is a closed common subpath, with one endpoint in $\partial G$ and the other one in $\partial B(x,\rho)$. In particular, $E_0\cap B(x,\rho)=G$.\par

Keep in mind that each $\psi_j$ is part of the boundary of a connected component of $E_{\ell(j)}\cap B(x,\rho)$ for some $1\leq \ell(j)\leq m$. Since every two of the paths have a non-negligible intersection, we deduce that $\ell(1),\,\ell(2)$ and $\ell(3)$ are different indices. For simplicity of notation, and without loss of generality, we assume that $\ell(j)=j$ for each $j\in\{1,\,2,\,3\}$. Just to fix the ideas, we also call
\begin{align*}
A = \psi_1(0) = \psi_2(1)\,, && B=\psi_3(0) = \psi_1(1)\,, && C=\psi_2(0)=\psi_3(1)\,,
\end{align*}
and let us call $A',\,B',\,C'$ the other endpoints of the intersections $\psi_1\cap \psi_2$, $\psi_1\cap\psi_3$ and $\psi_2\cap\psi_3$ respectively. The situation is depicted in Figure~\ref{Fignolakes}, left.\par
\begin{figure}[htbp]
\begin{tikzpicture}[>=>>>>,smooth cycle,scale=.69]
\filldraw[fill=blue!35!white, draw=black, line width=.75pt] plot [tension=1] coordinates {(-2,0) (-3.2,.4) (-3.4,.6) (-3.2,0) (-3,-2) (-2,-1.8) (-1,-1.4) (-0.2,-1) (-1,-.3)};
\filldraw[fill=green!35!white, draw=black, line width=.75pt] plot [tension=1] coordinates {(-1,-1.4) (-2.5,-2.3) (-1,-3) (3,-3) (3.5,-.8) (2.99,-.2) (1.5,.5) (0.5,-1.5)};
\filldraw[fill=red!35!white, draw=black, line width=.75pt] plot [tension=1] coordinates {(-2,0) (-3.5,1) (-2.5,2.5) (-0.5,3.1) (3,2.7) (3.5,0) (1.5,.5) (-0.2,1)};
\filldraw[fill=white, draw=white] (-4,0) -- (-3,0) arc(180:0:3) -- (4,0) -- (4,3.5) -- (-4,3.5) -- cycle;
\filldraw[fill=white, draw=white] (-4,0) -- (-3,0) arc(-180:-0:3) -- (4,0) -- (4,-3.5) -- (-4,-3.5) -- cycle;
\fill (1.5,.5) circle (2pt);
\fill (-2,0) circle (2pt);
\fill (-1,-1.4) circle (2pt);
\fill (2.99,-0.21) circle (2pt);
\fill (-2.99,0.21) circle (2pt);
\fill (-2.33,-1.89) circle (2pt);
\draw[line width=1pt] (0,0) circle (3);
\draw (0,2) node {$E_1$};
\draw (-2,-.9) node {$E_2$};
\draw (1.2,-2.2) node {$E_3$};
\draw (-2.99,.21) node[anchor=south east] {$A$};
\draw (2.99,-.21) node[anchor=north west] {$B$};
\draw (-2.33,-1.89) node[anchor=north east] {$C$};
\draw (-2,0) node[anchor=south east] {$A'$};
\draw (1.5,.5) node[anchor=south west] {$B'$};
\draw (-1,-1.4) node[anchor=north] {$C'$};
\draw (0,1) node[anchor=south east] {$\psi_1$};
\draw (-1.3,-1) node[anchor=south west] {$\psi_2$};
\draw (1.2,-.8) node[anchor=north] {$\psi_3$};
\draw (0,0) node[anchor=north] {$G$};
\draw[->>>>] (-0.05,1) -- (0,1.01);
\draw[->>>>] (.95,-.79) -- (0.948,-.8);
\draw[->>>>] (-1.296,-.179) -- (-1.3,-.178);
\filldraw[fill=blue!35!white, draw=black, line width=.75pt] (8.01,.21) -- (9,0) -- (10,-1.4) -- (8.67,-1.89) arc(219:176:3);
\filldraw[fill=green!35!white, draw=black, line width=.75pt] (8.67,-1.89) -- (10,-1.4) -- (12.5,.5) -- (13.99,-.21) arc(356:219:3);
\filldraw[fill=red!35!white, draw=black, line width=.75pt] (13.99,-.21) -- (12.5,.5) -- (9,0) -- (8.01,.21) arc (176:-4:3);
\draw[line width=1] (11,0) circle (3);
\fill (12.5,.5) circle (2pt);
\fill (9,0) circle (2pt);
\fill (10,-1.4) circle (2pt);
\fill (13.99,-0.21) circle (2pt);
\fill (8.01,0.21) circle (2pt);
\fill (8.67,-1.89) circle (2pt);
\draw (11,2) node {$\widetilde E_1$};
\draw (9,-.9) node {$\widetilde E_2$};
\draw (12.2,-2.2) node {$\widetilde E_3$};
\draw (8.01,.21) node[anchor=south east] {$A$};
\draw (13.99,-.21) node[anchor=north west] {$B$};
\draw (8.67,-1.89) node[anchor=north east] {$C$};
\draw (9.1,0.1) node[anchor=south east] {$A'$};
\draw (12.5,.5) node[anchor=south west] {$B'$};
\draw (10,-1.4) node[anchor=north] {$C'$};
\draw[->] (4,-.4) .. controls (5.5,-.2) .. (7,-.4);
\end{tikzpicture}
\caption{The paths $\psi_1,\,\psi_2$ and $\psi_3$ and the points $A,\,B,\,C,\,A',\,B',\,C'$ in Lemma~\ref{nolakesgen}.}\label{Fignolakes}
\end{figure}
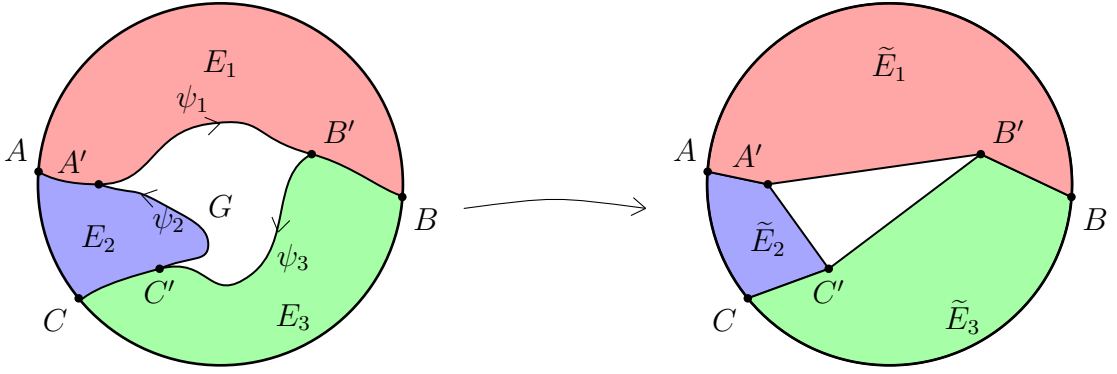
Notice now that ${\rm diam}(G)=|Q-P|$ for two points $P,\,Q\in\partial G$. Let us assume, just to fix the ideas, that $P\in \psi_1\cap \partial G$, and that $Q\in(\psi_1\cup\psi_3)\cap\partial G$. We can then apply Lemma~\ref{noncrazyint} to the paths $\tau_1$ and $\tau_2$ given by the restriction of $\psi_1$ between $P$ and $B'$, and between $A'$ and $P$ respectively, so that~(\ref{thesisnci}) gives $|P-B'| \leq C_1 |A'-B'|$. In the very same way, if $Q\in \psi_1$ we get $|Q-B'|\leq C_1 |A'-B'|$, while if $Q\in\psi_3$ we get $|Q-B'|\leq C_1|B'-C'|$. As a consequence, also by~(\ref{estidiam}) we deduce
\begin{equation}\label{pernotlo}
\frac r{2C_2} \leq {\rm diam}(G) \leq C_1\big(|A'-B'|+|B'-C'|+|C'-A'|\big)\,.
\end{equation}
We set now once again $\overline h:\R^2\to\R^+$ as $\overline h(\nu)=h(x,\nu)$, $\overline P$ the perimeter obtained substituting $h$ with $\overline h$ in~(\ref{weightedvolper}), $\h:\R^2\to\R^+$ as $\h(\nu)=\overline h(\hat\nu)$, being $\hat\nu$ the angle obtained rotating $\nu\in\R^2$ of $90^\circ$ clockwise, and $\len(\gamma)$ the lengt of any path $\gamma$ as in~(\ref{lengthcurves}). As in Figure~\ref{Fignolakes}, right, we define then a competitor $\widetilde \E$ by replacing the path $\psi_1$ with the union of the segments $AA',\, A'B'$ and $B'B$, the path $\psi_2$ with the segments $CC',\, C'A'$ and $A'A$, and the path $\psi_3$ with the segments $BB',\, B'C'$ and $C'C$. By Lemma~\ref{lemmagiorgio} we have $\overline P(\widetilde \E) \leq \overline P(\E)$. \par

Keeping in mind that $h$ is strictly convex in the second variable (in the sense of Definition~\ref{strictconv}), and then the unit ball corresponding to $\h$ is strictly convex, we have then a constant $\delta'>0$ such that
\[
\len(A'C')\leq \len (A'B')+\len(B'C') - 8\delta' h_{\rm max} {\rm dist}(B',A'C')\,.
\]
Putting this estimate together with the analogous ones for $C'B'$ and $B'A'$, we get
\[\begin{split}
\len(A'C') + \len(C'B')+\len(B'A') &\leq 2\big(\len(A'B')+\len(B'C')+\len(C'A')\big)\\
&\qquad- 4\delta' h_{\rm max}\big(|A'-B'|+|B'-C'|+|C'-A'|\big)\\
&\leq 2(1-2\delta')\big(\len(A'B')+\len(B'C')+\len(C'A')\big) \,,
\end{split}\]
which can be rewritten as
\[
\left(\!
\begin{array}{c}
\len(A'C')+\len(C'B')\\
+\\
\len(B'A')+\len(A'C')\\
+\\
\len(C'B')+\len(B'A')\\
\end{array}
\!\right)
\leq
\left(\!
\begin{array}{c}
(1-2\delta') \big(2\len(A'B')+\len(B'C')+\len(C'A')\big)\\
+\\
(1-2\delta') \big(2\len(B'C')+\len(C'A')+\len(A'B')\big)\\
+\\
(1-2\delta') \big(2\len(C'A')+\len(A'B')+\len(B'C')\big)
\end{array}
\!\right)\,.
\]
Hence, up to exchange the letters, we assume that
\[
\len(A'C')+\len(C'B') \leq (1-2\delta') \big(2\len(A'B')+\len(B'C')+\len(C'A')\big)\,.
\]
We are now in a position to define a second competitor, $\E'$, simply adding the triangle $A'B'C'$ to $\widetilde E_1$, that is, we set $E_j'=\widetilde E_j$ for every $j\neq 1$, and $E_1'=\widetilde E_1\cup A'B'C'$. By definition, we have then
\[\begin{split}
\overline P(\E') &= \overline P(\widetilde\E) -\len(A'B')+\frac{\len(A'C')-\len(C'A')+\len(C'B')-\len(B'C') } 2\\
&\leq \overline P(\E) -\delta'\, \big(2\len(A'B')+\len(B'C')+\len(C'A')\big)
\leq \overline P(\E) -\delta'\, \frac{h_{\rm min}}{2C_1C_2}\,r\,,
\end{split}\]
where in the last inequality we have also used~(\ref{pernotlo}).\par

The conclusion is now standard. Since $\E=\E'$ outside $B(x,r)$, and Lemma~\ref{13/2} gives a bound of $\haus^1(\partial\E\cap B(x,r))$ in terms of $r$, as soon as $R_6\leq R_5$ is small enough we have $\omega(r)$ so small that the above inequality implies
\[
P(\E')\leq P(\E) - \frac{\delta' h_{\rm min}}{3C_1C_2}\,r\,.
\]
Then, Lemma~\ref{labello} provides a further cluster $\E''$ with $|\E''|=|\E|$ such that
\[
P(\E'') \leq P(\E') + 2^\beta\Ceeb\Cgr^\beta r^{\eta\beta}
\leq P(\E) + 2^\beta\Ceeb\Cgr^\beta r^{\eta\beta}- \frac{\delta' h_{\rm min}}{3C_1C_2}\,r\,.
\]
And finally, the last inequality gives the searched contradiction $P(\E'')<P(\E)$ if $\eta\beta>1$ up to possibly furher reduce $R_6$, while if $\eta\beta=1$ the contradiction comes by~(\ref{Cperissmall}) defining the constant
\begin{equation}\label{ceeb5}
\Ceeb^5=\frac{\delta' h_{\rm min}}{3\cdot 2^\beta C_1 C_2\Cgr^\beta}\,,
\end{equation}
which again only depends on $h,\, A'$ and $\E$. The proof is then concluded.
\end{proof}

We can conclude this section by giving the definition of triple points and showing that there are only finitely many of them.

\begin{defn}[Triple points]\label{def3p}
We say that $x\in\R^2$ is a \emph{triple point} if
\[
\lim_{r\searrow 0}\ \# \Big\{ 0\leq i \leq m:\, \big| E_i\cap B(x,r)\big|>0\Big\} = 3\,.
\]
Notice that, by Lemma~\ref{atmost3}, the no-islands Lemma~\ref{noisland} and the no-lakes Lemma~\ref{nolakesgen}, this is equivalent to ask that the above limit is at least $3$.
\end{defn}

\begin{lem}[Finitely many triple points]\label{dist>R6}
Let $B(x,r)\subseteq \textsf{D}$ be a ball with $r<R_6$ and being $x$ a triple point. Then, there exists $r/C_2\leq \rho\leq r$ such that $\partial\E\cap B(x,\rho)$ consists of three paths of finite length, connecting $\partial B(x,\rho)$ with $x$, disjoint except at $x$. As a consequence, there is no other triple point in $B(x,\rho)$, so in particular triple points in $\textsf{D}$ are locally finitely many.
\end{lem}
\begin{proof}
First of all, we can apply Lemma~\ref{atmost3} to find $r/C_2\leq \rho\leq r$ such that Vol'pert Theorem~\ref{volpert} holds true for $B(x,\rho)$, $\rho$ is a Lebesgue point of the function $s\mapsto \haus^0(\partial\E\cap \partial B(x,s))$, and the value of this function at $\rho$ is at most $3$. Vol'pert Theorem implies that $\partial B(x,\rho)$ is the essentially disjoint union of either one, or two, or three open arcs belonging to different sets $E_i$. The no-islands Lemma~\ref{noisland} and the no-lakes Lemma~\ref{nolakesgen} ensure that $|E_i\cap B(x,\rho)|>0$ if and only if one of the above open arcs belong to $E_i$. Since $x$ is a triple point, we deduce that $\partial\E\cap \partial B(x,\rho)$ consists of exactly three points, call them $A,\, B,\, C$ for simplicity, that there are three distinct indices $0\leq j_1,\,j_2,\, j_3\leq m$ such that the arcs $AB,\, BC,\, CA$ of $\partial B(x,\rho)$ belong to $E_{j_1},\, E_{j_2},\, E_{j_3}$ respectively, and that $|E_i\cap B(x,\rho)|>0$ if and only if $i\in \{j_1,\,j_2,\, j_3\}$.\par

Lemma~\ref{lem1reg} and Corollary~\ref{correg} imply that each connected component of $E_i\cap B(x,\rho)$ has a boundary which is an injective, closed curve of finite length, and such a curve must reach $\partial B(x,\rho)$ by Lemma~\ref{noisland} and Lemma~\ref{nolakesgen}. By construction, each of the points $A,\,B$ and $C$ can be contained in at most two different curves among the boundaries of the connected components of $E_i\cap B(x,\rho)$. Therefore, for each $i\in \{j_1,\,j_2,\,j_3\}$ we have that $E_i\cap B(x,\rho)$ is made by a single connected component, and the boundary of such a connected component is the union of an arc of $\partial B(x,\rho)$ (in particular one between $AB,\,BC$ and $CA$) and a path contained in the interior of $B(x,\rho)$. We call $\psi_1,\,\psi_2$ and $\psi_3$ these three arcs, and we set $\tau_1=\psi_1\cap\psi_2,\, \tau_2 = \psi_2\cap\psi_3$ and $\tau_3=\psi_3\cap\psi_1$.\par

Keep in mind that $\haus^1$-a.e.\ point of $\partial \E$ belongs to exactly two different boundaries $\partial E_\ell$, with $0\leq \ell\leq m$, hence in particular $\haus^1$-a.e.\ point of $\partial\E\cap B(x,\rho)$ belongs to exactly two of the paths $\psi_1,\,\psi_2$ and $\psi_3$, that is, to one of the intersections $\tau_1,\,\tau_2,\,\tau_3$. However, Lemma~\ref{noncrazyint} implies that $\tau_1$ is an injective closed path, which is a common subpath of $\psi_1$ and $\psi_2$, and the same is true for $\tau_2$ and $\tau_3$. In other words, $\psi_1$ is the essentially disjoint union of the connected paths $\tau_1$ and $\tau_3$, $\psi_2$ is the essentially disjoint union of $\tau_1$ and $\tau_2$, and $\psi_3$ is the essentially disjoint union of $\tau_2$ and $\tau_3$. The three paths $\tau_1,\,\tau_2$ and $\tau_3$ meet then at some point $y\in B(x,\rho)$. Hence, we have proved that $\partial\E\cap B(x,\rho)$ is the union of the three paths $\tau_1,\,\tau_2$ and $\tau_3$, and these three paths connect the points $A,\,B,\,C \in \partial B(x,\rho)$ with the internal point $y$, and they are disjoint except for the common point $y$. Consequently, every point of $B(x,\rho)$ different from $y$ is \emph{not} a triple point, and this ensures that $y=x$ and concludes the proof.
\end{proof}

\subsection{Interface regularity\label{s:interface}}

This section is devoted to show the regularity of the boundary of the optimal cluster $\E$ away from the triple points, that is, where there are only two different sets. In this case, we show that $\partial\E$ is done by a union of regular curves. In particular, the goal of this section is to obtain the following result. 

\begin{prop}[${\rm C}^{1,\gamma}$ regularity]\label{step6}
There exists an increasing function $\xi:\R^+\to\R^+$ with $\lim_{r\to 0^+} \xi(r)=0$ with the following property. If $B(x,\bar r)\subseteq \textsf{D}$ is a ball with $\bar r<R_6$ and $\#\, \big(\partial \E\cap \partial B(x,\bar r)\big)=2$, then $\partial\E\cap B(x,\bar r)$ is a ${\rm C}^1$ curve of finite length having both endpoints in $\partial B(x,\bar r)$. Moreover, calling $\tau(y)\in\P^1$ the direction of the tangent vector at any $y\in\partial\E\cap B(x,\bar r)$, one has
\begin{equation}\label{uniformC1}
|\tau(y)-\tau(z)| \leq \xi(|y-z|)
\end{equation}
for every $y,\,z\in \partial\E\cap B(x,\bar r)$. Finally, if $\eta\beta>1$ and $h$ is locally $\alpha$-H\"older in the first variable, then it is possible to take $\xi(r)=K r^\gamma$ with some $K=K(\E,\, \textsf{D},\, g,\,h)>0$ and
\begin{equation}\label{heregamma}
\gamma=\frac 12\, \min\{\eta\beta-1,\, \alpha\}\,,
\end{equation}
so that in particular $\partial\E\cap B(x,\bar r)$ is ${\rm C}^{1,\gamma}$.
\end{prop}

The first step is to show that in a small ball the boundary is close to a line.

\begin{lem}[Almost alignment in a circle]\label{aa2bp}
There exists an increasing function $\xi_1:\R^+\to\R^+$ with $\lim_{r\to 0^+} \xi_1(r)=0$, which satisfies the Dini property, and which can be taken of the form $\xi(r)=Kr^\gamma$ with some $K=K(\E,\, \textsf{D},\, g,\,h)>0$ and $\gamma$ given by~(\ref{heregamma}) as soon as $\eta\beta>1$ and $h$ is locally $\alpha$-H\"older in the first variable, so that the following holds. Let $x\in \partial\E$ and $r<R_6$ be such that $B(x,r)\subseteq \textsf{D}$ and $\partial B(x,r) \cap \partial\E$ consists of two points, call them $a$ and $b$. Then, for every $y \in \partial\E\cap B(x,r/2)$ one has $|\angle yab| \leq \xi_1(r)$.
\end{lem}
\begin{proof}
We directly define $\E'$ as the cluster which coincides with $\E$ outside of $B(x,r)$ and such that $\partial\E'\cap B(x,r)$ is done by the segment $ab$. Keep in mind that, since $\partial\E\cap \partial B(x,r)$ is done by two points, Lemma~\ref{noisland} and Lemma~\ref{nolakesgen} imply that $B(x,r)$ is the union of two connected regions, each one contained in a set $E_i$ for some $0\leq i\leq m$. By Lemma~\ref{lem1reg} and Corollary~\ref{correg}, we know that the common boundary between these regions, which coincides with the whole $\partial\E\cap B(x,r)$, is a path $\gamma$ contained in $B(x,r)$ and connecting $a$ to $b$. Since $y\in\partial\E$, in particular $y\in\gamma$. Let us call $i_l$ and $i_r$ the two indices in $\{0,\, 1 ,\,\dots\,,\, m\}$ so that the set $E_{i_l}$ (resp., $E_{i_r}$) is on the left side of $\gamma$ (resp., on the right side).\par

Once again, for every $\nu\in\S^1$ we call $\hat\nu$ the angle obtained rotating $\nu$ of $90^\circ$ clockwise. This time, we set $\h(\nu)=h(x,\hat\nu)$ if $i_r=0$, $\h(\nu)=h(x,-\hat\nu)$ if $i_l=0$, and $\h(\nu)=(h(x,\hat\nu)+h(x,-\hat\nu))/2$ if both the indices $i_l$ and $i_r$ are different from $0$. We use then~(\ref{lengthcurves}) to define the length of curves with this choice of $\h$. Also by Lemma~\ref{13/2} and Lemma~\ref{lemmagiorgio}, we have
\begin{equation}\label{rsi}\begin{split}
P(\E')-P(\E) &\leq \len(ab) +\omega(r) |b-a| - \len(\gamma) +\omega(r) \haus^1(\gamma)\\
&\leq \len(ab) - \len(ay)-\len(yb) +\bigg(2 +7\,\frac{h_{\rm max}}{h_{\rm min}}\bigg) r \omega(r)\,.
\end{split}\end{equation}
Let us now write for brevity $\theta=\angle yab$. We claim that
\begin{equation}\label{thisest}
\len(ay)+\len(yb)-\len(ab) \geq c' r \sin^2 \theta\,,
\end{equation}
for a constant $c'$ which only depends on $h$ and $\textsf{D}$. To show this estimate, we call $y_\perp$ the projection of $y$ on $ab$. First of all, we can reduce ourselves to the ``symmetric'' case when $y_\perp$ is the middle point of $ab$. Indeed, assume that $y_\perp$ is not the middle point of $ab$ and let $a'b'$ the shortest segment containing $ab$ with middle point equal to $y_\perp$. If~(\ref{thisest}) holds true in the symmetric case, then in particular it holds true with $a',\,b'$ and $\theta'=\angle y{a'}{b'}$ in place of $a,\,b$ and $\theta$. Moreover, since $r/2 \leq |y-a|,\, |y-b| \leq 3r/2$, then $\sin\theta' \geq \sin\theta/3$. Since the triangular inequality implies
\[
\len(ay)+\len(yb)-\len(ab) \geq 
\len(a'y)+\len(yb')-\len(a'b')\,,
\]
we have then the validity of~(\ref{thisest}) in the general case, up to divide $c'$ by $9$. We are then left to show~(\ref{thisest}) in the symmetric case. Let us call
\begin{align*}
\nu=\frac{y_\perp-a}{|y_\perp-a|} \,, && w=\frac{y-y_\perp}{4|y_\perp-a|}\,,
\end{align*}
so that by construction $|\nu|=1$ and $|w|\leq 1$. Keep in mind that $\h$ is uniform round in the second variable (see Definition~\ref{strictconv}), and then~(\ref{numero}) and the convexity give
\[\begin{split}
\len(ay)+\len(yb)&-\len(ab)=|y_\perp-a| \big(\h(\nu+4w)+\h(\nu-4w) -2\h(\nu)\big) \\
&\geq \frac r2\, \big(\h(\nu+w)+\h(\nu-w) -2\h(\nu)\big) 
\geq r c |w|^2\geq \frac c{16}\, r \sin\theta^2\,,
\end{split}\]
and this proves~(\ref{thisest}). Inserting this inequality in~(\ref{rsi}), we obtain
\[
P(\E')-P(\E) \leq -c'r \sin^2 \theta +\bigg(2 +7\,\frac{h_{\rm max}}{h_{\rm min}}\bigg)r \omega(r)\,.
\]
Observe that $\big||\E'|-|\E|\big|\leq 2|B(x,r)| \leq 2\Cgr r^\eta$, and write for brevity $\widetilde\Ceeb = \Ceeb [ 2\Cgr r^\eta]$. Appling then once again Lemma~\ref{labello}, using the constant $\widetilde\Ceeb$ in place of $\Ceeb$ in~(\ref{proplab}), the optimality of $\E$ implies that
\[
c' r \sin^2 \theta \leq \bigg(2 +7\,\frac{h_{\rm max}}{h_{\rm min}}\bigg)r \omega(r) + \widetilde\Ceeb (2\Cgr r^\eta)^\beta \,,
\]
which implies
\[
\theta\leq\xi_1(r):= \frac \pi 2\, \bigg[\frac 1{c'}\bigg(\bigg(2 +7\,\frac{h_{\rm max}}{h_{\rm min}}\bigg) \omega(r) + \widetilde\Ceeb 2^\beta\Cgr^\beta r^{\eta\beta-1} \bigg)\bigg]^{1/2}\,.
\]
To conclude the proof, we have then to check that $\xi_1$ satisifies all the requirements. Keeping in mind that $\widetilde\Ceeb=\Ceeb[2\Cgr r^\eta]$ is an increasing function of $r$ which goes to $0$ when $r\searrow 0$, the fact that $\xi_1$ is an increasing function and that $\lim_{r\searrow 0} \xi_1(r)=0$ is true by construction. If $\eta\beta>1$ and $h$ is $\alpha$-H\"older in the first variable, then we obtain
\[
\xi_1(r) \lesssim \sqrt{r^\alpha + r^{\eta\beta-1}}\approx r^\gamma\,,
\]
with $\gamma$ given by~(\ref{heregamma}). Finally, up to multiplicative constants we have that
\[
\xi_1(r) \leq \sqrt{\omega(r) + r^{\eta\beta-1}\Ceeb[2\Cgr r^\eta]}
\leq \sqrt{\omega(r)} + \sqrt{r^{\eta\beta-1}\Ceeb[2\Cgr r]}\,,
\]
where the last inequality comes because $\eta\geq 1$. The Dini property of $\xi_1$ then follows, since $r\mapsto \omega(r)$ satisfies the $1/2$-Dini property, and the same is true for $r\mapsto\Ceeb[r]$ if $\eta\beta=1$.
\end{proof}

\begin{cor}\label{cor1pg}
For every $x\in \textsf{D}\cap \partial\E$ and every $r< \min\{{\rm dist}(x,\partial \textsf{D}),R_6\}/2C_2$ with the property that $\#\big\{ 0\leq i \leq m:\,|B(x,r)\cap E_i|>0\big\}=2$, one can choose a direction, that we call $\tau(x,r)\in\P^1$, so that each point $y\in\partial\E\cap\partial B(x,r)$ satisfies
\begin{equation}\label{firstone}
|\zeta(y-x)- \tau(x,r)| \leq 24 C_2 \xi_1(2C_2 r)
\end{equation}
where, for each $v\in\R^2\setminus\{0\}$, we denote by $\zeta(v)\in\P^1$ the direction of $v$. In particular, if $w,\,z\in\partial\E$ are two points such that, calling $d=|w-z|$, both $\tau(w,d)$ and $\tau(z,d)$ are defined, then
\begin{equation}\label{thirdone}
|\tau(w,d) - \tau(z,d) | \leq 48 C_2 \xi_1(2C_2 d)\,.
\end{equation}
In addition, for each $r' \in [r/2,r]$ one has
\begin{equation}\label{secondone}
|\tau(x,r)-\tau(x,r')| \leq 48 C_2 \xi_1(2C_2 r)\,.
\end{equation}

\end{cor}
\begin{proof}
Since by assumption the ball $B(x,2C_2 r)$ is contained in $\textsf{D}$ and its radius is less than $R_6\leq R_5$, we apply Lemma~\ref{atmost3} and find some $2r < \rho< 2C_2 r$ such that $\partial\E\cap \partial B(x,\rho)$ has at most three points. These points cannot be three since by assumption $B(x,\rho)\subseteq B(x,2C_2 r)$ intersects only two regions $E_i$, $0\leq i \leq m$, and they cannot be less than $2$ because $x\in\partial\E$ and by the no-islands Lemma~\ref{noisland} and the no-lakes Lemma~\ref{nolakesgen}. Therefore, $\partial\E\cap \partial B(x,\rho)$ has necessarily exactly two points, and we call them $a$ and $b$ for simplicity. The vector $\tau(x,r)\in\P^1$ can be then simply defined as the direction $\zeta(b-a)$ of the segment $ab$. Notice that this direction is not uniquely determined by $x$ and $r$, since it also depends on the particular choice of $\rho$.\par
To check the properties of $\tau$, let us take $y\in \partial B(x,r')\cap\partial\E$ for some $r/2\leq r'\leq r$. Just to fix the ideas, let us call $w_2$ the second coordinate of any point $w\in\R^2$, and let us assume that the segment $ab$ is horizontal, with $a_2=b_2=0$. Since both $y$ and $x$ belong to $\partial\E\cap B(x,\rho/2)$, Lemma~\ref{aa2bp} gives $|\angle xab|< \xi_1(\rho)$ and $|\angle yab|<\xi_1(\rho)$, which implies
\begin{align*}
|x_2| = |x-a| |\sin \angle xab| \leq \rho \sin \xi_1(\rho)\,, &&
|y_2| = |y-a| |\sin \angle yab| \leq 2 \rho \sin \xi_1(\rho)\,,
\end{align*}
and then
\[
|\sin \zeta(y-x)| =\frac{|y_2-x_2|}{|y-x|} \leq \frac{3 \rho\sin \xi_1(\rho)}{r'} \leq 12 C_2 \sin \xi_1(\rho)\leq 12 C_2 \sin \xi_1(2C_2 r)\,.
\]
Summarizing, since for every $0\leq \theta\leq \pi/2$ we have $\theta\geq\sin\theta\geq 2\theta/\pi\geq \theta/2$, and since $\tau(x,r)$ is the horizontal direction, we have proved that
\begin{equation}\label{funda}
|\zeta(y-x)-\tau(x,r)| \leq 24 C_2 \xi_1(2C_2 r) \qquad\qquad \forall \, y \in \partial\E\cap \Big(\overline{B(x,r)}\setminus B(x,r/2)\Big)\,.
\end{equation}
The particular case in which $y\in\partial B(x,r)$ is~(\ref{firstone}). Let now $r/2\leq r'\leq r$, and take a point $y\in \partial\E\cap\partial B(x,r')$. Since we can apply~(\ref{funda}) to $x$ and $y$ both with $r$ and with $r'$, we deduce
\[
|\tau(x,r)-\tau(x,r')| \leq |\zeta(y-x)-\tau(x,r)| + |\zeta(y-x)-\tau(x,r')|\leq 48 C_2 \xi_1(2C_2 r)\,,
\]
which is~(\ref{secondone}). Finally, let $z,\,w \in \partial \E \cap \partial B(x,r)$ be such that, calling $d=|w-z|$, one has $2C_2 r < \min\big\{{\rm dist} (z,\partial \textsf{D}),\, {\rm dist}(w,\partial \textsf{D}), R_6\big\}$, so that both $\tau(z,d)$ and $\tau(w,d)$ are defined. Then, we can apply~(\ref{funda}) with $r=d,\, x=z,\, y=w$, and also with $r=d,\, x=w,\, y=z$. This gives then~(\ref{thirdone}).
\end{proof}

We are now in a position to prove Proposition~\ref{step6}.

\begin{proof}[Proof (of Proposition~\ref{step6}).]
Let $x$ and $\bar r$ be as in the claim of the proposition. The fact that $\partial\E\cap B(x,\bar r)$ is an injective curve of finite length with both endpoints in $\partial B(x,\bar r)$ has already been observed in Lemma~\ref{aa2bp}. By Corollary~\ref{cor1pg}, a direction $\tau(y,r)$ is defined for every $y\in \partial\E\cap B(x,\bar r)$ and every $r<\min\{{\rm dist} (y,\partial \textsf{D}),R_6\}/2C_2$. Moreover, (\ref{secondone}) holds true as soon as $r/2\leq r'\leq r$. An obvious induction gives then, for every $n\in\N$ and every $r/2^n\leq r'\leq r$,
\begin{equation}\label{induc}
|\tau(y,r)-\tau(y,r')| \leq 48 C_2 \sum_{i=0}^{n-1} \xi_1(2C_2 r/2^i)\,.
\end{equation}
We define then
\[
\xi(r) = 144 C_2 \sum_{i=0}^{+\infty} \xi_1(2C_2 r/2^i)\,.
\]
Notice that the series converges because $\xi_1$ satisfies the Dini property. In particular, if $\eta\beta>1$ and $h$ is locally $\alpha$-H\"older in the first variable, then $\xi_1(r)=K_1 r^\gamma$, with $\gamma$ given by~(\ref{heregamma}) and $K_1$ being a constant depending on $\E,\, \textsf{D},\, g$ and $h$. Thus, also $\xi(r)=K r^\gamma$ by definition.\par

By~(\ref{induc}) we obtain that $\tau(y,r)$ converges to a direction when $r\searrow 0$, and we call $\tau(y)\in\P^1$ this limit direction. By construction, $|\tau(y,r)-\tau(y)|\leq \xi(r)/3$. Therefore, for every point $z\in \partial\E\cap \partial B(y,r)$, recalling~(\ref{firstone}) we have then $|\zeta(z-y)-\tau(y)|\leq \xi(r)/2$, so that $\tau(y)$ is the tangent vector at $y$ of the curve $\partial\E\cap B(x,\bar r)$. Finally, (\ref{uniformC1}) comes by~(\ref{thirdone}).
\end{proof}

\subsection{Conclusion}

In this short last section we can now give the proof of Theorem~\mref{main}, which basically consists in putting together the technical results of the preceding sections.

\begin{proof}[Proof of Theorem~\mref{main}]%\proofof{Theorem~\mref{main}}
Let $\E\subseteq\R^2$ be a minimal cluster, and let us fix two large, closed balls $\textsf{D}^-\comp \textsf{D}\subseteq\R^2$. Let $x$ be any point in $\textsf{D}^- \cap \partial\E$. If $x$ is not a triple point, by Lemma~\ref{atmost3} there is a small constant $r(x)<R_6$ such that $\partial\E\cap\partial B(x,r(x))$ consists of two points (the points cannot be three if $r(x)$ is small enough, as already noticed). By Proposition~\ref{step6} we have then that $\partial\E\cap\partial B(x,r(x))$ is a ${\rm C}^1$ curve, whose tangent vector satisfies the uniform estimate~(\ref{uniformC1}).\par

Suppose instead that $x$ is a triple point. Then, again by Lemma~\ref{atmost3}, there is a small constant $r(x)<R_6$ such that $\partial\E\cap\partial B(x,r(x))$ consist of three points, call them $a,\,b$ and $c$. Lemma~\ref{dist>R6} already gives that $\partial\E\cap B(x,r(x))$ is done by three paths of finite length, connecting $a,\, b$ and $c$ to $x$, and disjoint except for the common endpoint $x$. Let $z\neq x$ be any point of one of these paths. Since $z$ is not a triple point, by the above argument we know that $\partial\E$ is a ${\rm C}^1$ curve in a neighborhood of $z$, and the tangent vector satisfies the uniform estimate~(\ref{uniformC1}). As a consequence, the three paths are three ${\rm C}^1$ paths, and they meet at $x$ with three well-defined tangent vectors.

Summarizing, each point $x\in \partial \E \cap \textsf{D}^-$ is center of a ball, in the interior of which $\partial\E$ is given either by a single ${\rm C}^1$ curve, or by three ${\rm C}^1$ curves meeting with three tangent vectors in the center. By compactness, we can cover $\partial\E\cap \textsf{D}^-$ with finitely many such balls, and then the Steiner property of $\E$ follows. The ${\rm C}^{1,\gamma}$ regularity in the case that $h$ is locally $\alpha$-H\"older in the first variable and $\eta\beta>1$ is given by Proposition~\ref{step6}.
\end{proof}

\section{Final comments}\label{s:final}

This last section is devoted to present a couple of final comments about our result. The first observation is about the role of the ${\rm C}^1$ property to obtain that multiple points are necessarily triple points, and the second one is about the directions of the arcs at triple points.

\subsection{The importance of the ${\rm C}^1$ property to obtain triple points\label{caseL1}}

Our main result, Theorem~\mref{main}, concerns the Steiner property for minimal clusters, that is, the boundary of a minimal cluster is made by finitely many ${\rm C}^1$ arcs which meet each other in triple points. We have shown that this property is true as soon as, together with the ``correct'' growth conditions and $\eps-\eps^\beta$ property, $h$ is strictly convex, uniformly round and ${\rm C}^1$ in the second variable. As already discussed in the Introduction, the importance of the strict convexity and uniform roundedness to obtain a Steiner propery is very simple to understand. Concerning the ${\rm C}^1$ regularity of $h$, it is also clear that this is necessary to get local ${\rm C}^1$ regularity of minimal clusters. This has nothing particular to do with the fact that we deal with clusters, the very same happens even in the much simpler case of isoperimetric sets. For instance, if $h$ does not depend on the first variable, then isoperimetric sets are translations and homotheties of the unit ball of $h$, so they are exactly as regular as $h$ is. Less obvious is the role played by the ${\rm C}^1$ property of $h$ in order to get triple points. This section is devoted to show by means of an example that quadruple points may occur for a density which is strictly convex and uniformly round but not ${\rm C}^1$. In fact, as shown in~\cite{MFG}, quadruple points are the worst than can happen, that is, a minimizing cluster for a generic norm in $\R^2$ may not have multiple points where more than four arcs meet. See also~\cite{ACH,LM} and~\cite[Chapter~10.10]{M1}, where minimizing networks in $\R^N$ for uniformly convex ${\rm C}^1$ norms are considered.
%studying how many segments are allowed to meet at a point of a minimizing network. in particular in \cite{LM} it is shown that in $\R^N$,  $N+1$ segments meeting at a point, but not $N+2$, and in \cite{ACH} it is shown that in $\mathbb R^2$, norm-minimizing networks are allowed to meet in fours, although never in fives.

We start with the following weaker example, depicted in Figure~\ref{FigsecL1}, left.

\begin{example}\label{exL1}
We want to show an example of an isoperimetric cluster with a quadruple point. It is possible to find a minimal cluster $\E=(E_1,E_2,E_3,E_4)$ corresponding to two continuous densities $g$ and $h$ so that, calling $\Q=[-1,1]\times [-1,1]$:
\begin{itemize}
\item $g(x)=1$ and $h(x,\nu) = \|\nu\|_\infty=\max\{|\nu_1|,\,|\nu_2|\}$ for every $x\in \Q$ and $\nu\in\S^1$;
\item $|E_1\cap \Q|= |E_2\cap \Q|=|E_3\cap \Q|=|E_4\cap \Q|=1$;
\item each of the four intersections $E_i \cap\partial \Q$ coincides with a side of the square $\Q$.
\end{itemize}
We omit the proof of this statement, which is lengthy and technical. However, the strategy of the proof is simple and goes as follows. Given any cluster $\E_0$ (in particular, a cluster whose boundary in $\Q$ coincides with the two diagonals), it is the minimal cluster for $g_0\equiv 1$ and a discontinuous density $h_0$ with values in $[0,+\infty)$, by setting $h_0\equiv 0$ on the boundary of the cluster and $h_0\equiv 1$ outside. The searched continuous densities $g$ and $h$---which have values in $(0,+\infty)$---can then be found as carefully chosen smoothings of $g_0$ and $h_0$, in such a way that there exists a minimal cluster $\E$, not necessarily coinciding with $\E_0$ but satisfying the properties above.

In such a situation, the perimeter of the cluster $\E$ inside the square is at least $4$, and it equals $4$ if and only if the boundary of the cluster inside the square is done by the two diagonals. This must then be the case, and so the origin is a quadruple point.
\end{example}

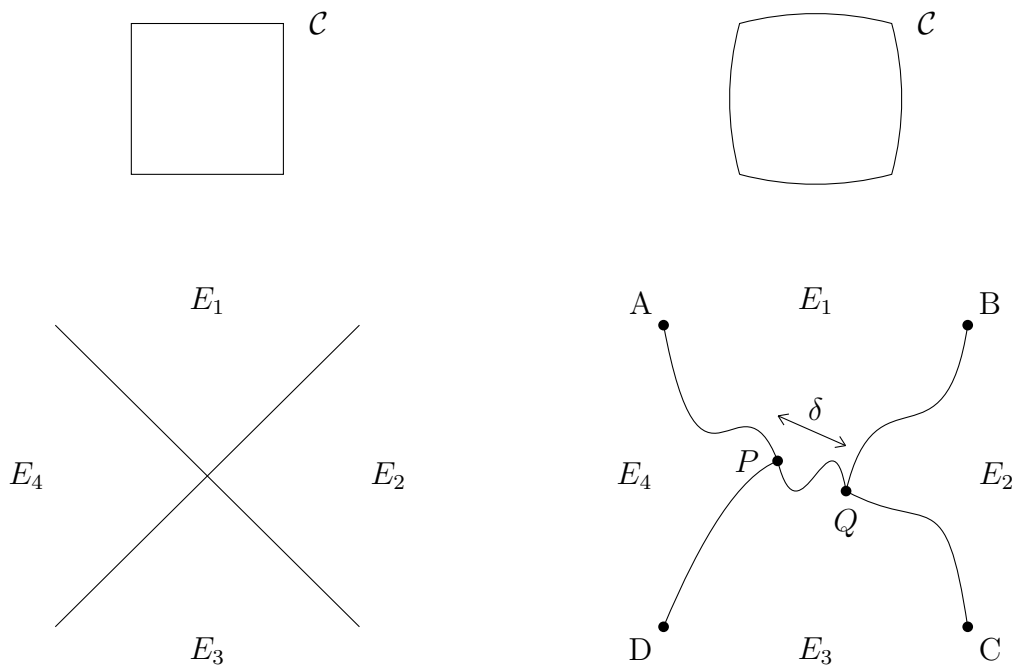
\begin{figure}[htbp]
\begin{tikzpicture}[>=>>>]
\draw (0,0) -- (2,0) -- (2,2) -- (0,2) --cycle;
\draw (2.2,2) node[anchor=west] {$\C$};
\draw (-1,-2) -- (3,-6);
\draw (3,-2) -- (-1,-6);
\draw (1,-2) node[anchor=south] {$E_1$};
\draw (3,-4) node[anchor=west] {$E_2$};
\draw (1,-6) node[anchor=north] {$E_3$};
\draw (-1,-4) node[anchor=east] {$E_4$};
%\draw (7,0) arc(-100:-80:5.76) arc(-10:10:5.76) arc(80:100:5.76) arc(170:190:5.76);
\draw (7,0) arc(-105:-75:3.864) arc(-15:15:3.864) arc(75:105:3.864) arc(165:195:3.864);
\draw (9.2,2) node[anchor=west] {$\C$};
%\draw (10,-2) -- (7,-6);
\draw (6,-2) .. controls (6.5,-4.7) and (7,-2.5) .. (7.5,-3.8);
\draw (8.4,-4.2) .. controls (9.5,-4.8) and (9.7,-4) .. (10,-6);
\fill (7.5,-3.8) circle (2pt);
\fill (8.4,-4.2) circle (2pt);
\draw[<->] (7.5,-3.2) -- (8.4,-3.6);
\draw (8,-3.4) node[anchor=south] {$\delta$};
\draw (7.4,-3.8) node[anchor=east] {$P$};
\draw (8.4,-4.3) node[anchor=north] {$Q$};
\draw (6,-6) .. controls (6.5,-4.8) and (7,-4) .. (7.5,-3.8);
\draw (8.4,-4.2) .. controls (8.8,-2.5) and (9.7,-4) .. (10,-2);
\draw (7.5,-3.8) .. controls (7.8,-5) and (8.2,-3) .. (8.4,-4.2);
\draw (8,-2) node[anchor=south] {$E_1$};
\draw (10,-4) node[anchor=west] {$E_2$};
\draw (8,-6) node[anchor=north] {$E_3$};
\draw (6,-4) node[anchor=east] {$E_4$};
\fill (6,-2) circle (2pt);
\fill (10,-2) circle (2pt);
\fill (6,-6) circle (2pt);
\fill (10,-6) circle (2pt);
\draw (6,-2) node[anchor=south east] {A};
\draw (10,-2) node[anchor=south west] {B};
\draw (10,-6) node[anchor=north west] {C};
\draw (6,-6) node[anchor=north east] {D};
\end{tikzpicture}
\caption{Left: the unit ball $\C$ and an isoperimetric cluster for Example~\ref{exL1}. Right: the unit ball $\C$ and an (impossible) isoperimetric cluster without quadruple points for Example~\ref{exL2}.}\label{FigsecL1}
\end{figure}

We can now slightly modify the above example, so that the density $h$ becomes strictly convex and uniformly round, and still a quadruple point occurs.

\begin{example}\label{exL2}
We define this time a density $h$ inside $\Q$ as in Figure~\ref{FigsecL1}, right. It is very close to the $L^\infty$ density of Example~\ref{exL1}, but the four sides of the unit ball are now substituted by four arcs, with strictly positive but very small curvature, and with the same four endpoints, i.e.\ $(\pm 1,\,\pm 1)$. Notice that $h$ is not ${\rm C}^1$, but it is strictly convex and uniformly round.\par

As in the example above, with a suitable choice of $g$ and $h$ outside of the square (and again with $g\equiv 1$ inside) we can obtain a minimal cluster $\E$ so that the four sides of the square $\Q$ belong to the sets $E_1,\, E_2,\, E_3$ and $E_4$, and that these four sets have volume $1$ each inside the square. We claim that then the minimal cluster has again a quadruple point at the origin. If this is false, then there is either a quadruple point at a point different from the origin, or (at least) two triple points.\par

We can then call $P$ and $Q$ two multiple points, as in the figure, and assume that they do not coincide both with the origin, in particular they are either both triple points and distinct, or they could coincide and be a single quadruple point, but not in the origin. Notice that, since the density is very close to the one of Example~\ref{exL1}, and in that case the boundary of the minimal cluster was done by the two digonals $AC$ and $BD$, then both the points $P$ and $Q$ must be very close to the origin, so in particular $d,\,\delta \ll 1$, where we call $\delta=|P-Q|$ the Euclidean distance between the two points and $d=\max\{|P-O|,\, |Q-O|\}$ the Euclidean distance between the origin and the furthest of the two points, that we assume to be $P$ to fix the ideas. By Lemma~\ref{lemmagiorgio}, the perimeter of the cluster inside the square is
\begin{equation}\label{sqL1}
P(\E; \Q)\geq \len(AP)+\len(DP)+\len(PQ)+\len(BQ)+\len(CQ)\,,
\end{equation}
where by $\len$ we denote the length of a curve, or a segment, with repect to $h$, that is, the definition~(\ref{lengthcurves}) with $h$ in place of $\h$. Notice that in this case there is no need to consider oriented segments, since the density is symmetric, and there is also no need to consider a clockwise rotation of $90^\circ$ as through the rest of the paper, because the density remains the same after a rotation of $90^\circ$. There is a constant $C>0$, depending on $h$ such that
\[
\len(AP)+\len(BP)+\len(CP)+\len(DP) \geq 4 + C d\,.
\]
In fact, the best constant $C$ for which the above inequality is true depends continuously on the curvature of the four arcs of $\C$. Since we have $C=1$ for the case of Example~\ref{exL1}, which corresponds to zero curvature, we can assume $C>1/2$ up to have chosen a sufficiently low curvature. Similarly, we have
\[
\len(BQ)+\len(CQ) \geq \len (BP)+\len(CP) - C' \delta\,.
\]
This estimate is again easily seen to hold with $C'=2$ for the density of Example~\ref{exL1}, so with some $C'<3$ in the present case. Inserting the last two estimates in~(\ref{sqL1}), and keeping in mind that the optimal cluster has at most perimeter $4$ in $\Q$ (because we can use the cluster with boundary in $\Q$ given by the two diagonals as competitor), we get
\begin{equation}\label{eL2}
d< 6\delta\,,
\end{equation}
that is, the points $P$ and $Q$ cannot be much closer to each other than to the origin. There exists a third constant $c>0$ such that
\[
\len(AP)+\len(DP) \geq \len (AO)+\len(DO) - c|P-O|=2 - cd\,.
\]
In fact, for the density of Example~\ref{exL1} this is true with $c=0$, because the unit ball for $h$ in that case is a square. In our case, provided that the curvature of the arcs of $\C$ is suficiently small, we can assume $c$ as small as desired. Similarly,
\[
\len(BQ)+\len(CQ) \geq 2 - c|Q-O|\geq 2 - cd\,.
\]
Finally,
\[
\len(PQ) \geq \frac{\sqrt 2}2\, |P-Q| = \frac{\sqrt 2}2\, \delta\,.
\]
Inserting the last three estimates in~(\ref{sqL1}), we get
\[
cd \geq \frac{\sqrt 2}4\, \delta\,,
\]
and since as observed $c$ can be taken arbitrarily small this gives a contradiction to~(\ref{eL2}). We have then proved that also for this modified density $h$, which is not ${\rm C}^1$ but which is strictly convex and uniformly round, a minimal cluster can have a quadruple point.
\end{example}

\subsection{The directions of the arcs at triple points\label{whichdir}}

This section is devoted to discuss which can be admissible directions for the tangents of $\partial^*\E$ at some triple point. First of all, we can immediately write down the first order minimality property.
\begin{lem}\label{qs}
Let $\h:\R^2\to\R^+$ be a positively $1$-homogeneous fuction, strictly positive and ${\rm C}^1$ except at $0$ and with strictly convex unit ball. Let moreover $A,\,B,\,C,\,O$ be four distinct points such that $A,\,B$ and $C$ are nonaligned. Then, $O$ uniquely minimizes the function $L:\R^2\to\R^+$ given by
\begin{equation}\label{Lgrosso}
L(P):= \h(PA)+\h(PB)+\h(PC)
\end{equation}
if and only if
\begin{equation}\label{1storder}
\nabla \h(OA)+\nabla\h(OB)+\nabla\h(OC)=0\,.
\end{equation}
\end{lem}
\begin{proof}
The function $P\mapsto L(P)$ is strictly convex because so is the unit ball of $\h$ and because $A,\,B,\,C$ are nonaligned. Moreover, this function is ${\rm C}^1$ except at $A,\,B$ and $C$. Therefore, $O$ uniquely minimizes $L$ if and only if $\nabla L(O)=0$. In particular,
\[
\nabla L(O) = -\big(\nabla \h(OA) +\nabla \h(OB) + \nabla \h(OC)\big)\,,
\]
hence we have concluded.
\end{proof}

Notice that the above geometrical property characterizes the possible directions corresponding to triple points. Let us be more precise. Suppose for a moment, just for simplicity, that $\h$ is symmetric, so that the length of curves is defined (otherwise one has to speak about \emph{oriented} curves, as already done in Section~\ref{sec:proof}). Let then $A,\,B$ and $C$ be three points in $\R^2$. By means of Lemma~\ref{lemmagiorgio}, it is very simple to notice that the shortest connected set containing $A,\,B$ and $C$ is always given by the three segments joining $A,\,B$ and $C$ with some point $O$, which might coincide with one between $A,\,B$ and $C$. Of course, this point $O$ minimizes the function $L(P)$ defined in~(\ref{Lgrosso}).

Let us consider this function in the general case when $\h$ does not need to be symmetric. The existence of a point $O$ minimizing $L$ is obvious, and by the strict convexity of the unit ball we can observe that such a point is uniquely determined if the points are not aligned. In addition, if the points are not aligned and $O$ is not one of them, the three directions $OA,\,OB$ and $OC$ necessarily satisfy the relation~(\ref{1storder}). It is interesting to observe ``how many'' admissible triples there are, and this is explained by the lemma below.
\begin{lem}
Let $\h:\R^2\to\R^+$ be as in Lemma~\ref{qs}. Then, there exists a triple of non-aligned points $\{A,\,B,\,C\}$ in $\partial\C$, with the property that~(\ref{1storder}) is satisfied with some $O\notin \{A,\,B,\,C\}$. In addition, if $\h$ is symmetric then for every $A\in\partial \C$ there exists a unique choice of $\{B,\,C\}$ in $\partial\C$ so that the triple $\{A,\,B,\,C\}$ has the previous property. Instead, if $\h$ is not symmetric, it is possible that for some $A\in\partial\C$ there is either no such pair, or more than one.
\end{lem}
\begin{proof}
We have already noticed that for every three non-aligned points $A,\,B,\,C$ there exists a unique point $O$ minimizing the function $L$ defined in~(\ref{Lgrosso}), and the three directions $OA,\,OB$ and $OC$ satisfy the property~(\ref{1storder}) if $O\notin \{A,\,B,\,C\}$. To prove the first part of the statement we want then to find three such points. Notice that the requirement that the points $A,\,B,\,C$ belong to $\partial\C$ is just to fix their length, but since $\nabla\h$ is $0$-homogeneous this can be achieved for free just dividing the length of the segments $OA,\,OB$ and $OC$ by $\h(OA),\,\h(OB)$ and $\h(OC)$ recpectively.\par

We start by taking three non-aligned points $A_0,\,B,\,C\in\R^2$. If the corresponding minimizing point $O$ is not one of them, we are already done. Otherwise, we can assume that $O=A_0$. We will obtain the first part of the statement by finding a point $A_t$ such that $A_t,\,B$ and $C$ are still non-aligned, and the point $O_t$ minimizing the function $L_t(P)=\h(PA_t)+ \h(PB)+\h(PC)$ is not in $\{A_t,\,B,\,C\}$.\par

Up to change the names of the points, we assume that $\h(BC)\leq \h(CB)$. Then, we let $A_1$ be the point such that $C$ is the middle point of the segment $A_1 B$. Moreover, for any $0<t<1$ we write $A_t=tA_1 + (1-t)A_0$. For every $0\leq t<1$ the points $A_t,\, B$ and $C$ are not aligned, hence $L_t$ is minimized by a unique point, that we call $O_t$. If, for some $0<t<1$, the point $O_t$ is not one between $A_t,\, B$ and $C$ then we are done. If this does not happen, then by uniqueness and continuity we derive $O_t=A_t$ for every $0<t<1$. Again by continuity, the point $A_1$ is then a minimizer (not necessarily the unique one) of the function $L_1$. And in turn, this gives a contradiction because
\[
L_1(A_1) = 3 \h(CB) > \h(CB) +\h(BC) = L_1(C)\,.
\]
The first part of the claim is then proved.\par

Let us now pass to the second part. As already observed in Section~\ref{sect90}, see in particular~(\ref{1order}), for every $P\in\partial\C$ we have
\begin{equation}\label{weknowthat}
\nabla \h(OP)=\frac{\nu_P}{OP\cdot \nu_P}\,,
\end{equation}
where $\nu_P$ denotes the outer unit normal vector to $P$ at $\partial\C$. Fix now a direction $\eta\in\S^1$, and consider the line passing through $O$ in direction $\eta$. As in Figure~\ref{Figsymm}, left, we call $R^-$ and $R^+$ the two intersections of this line with $\partial\C$, being $OR^-\cdot \eta < 0 < OR^+ \cdot \eta$, and $Q^+,\, Q^-$ the two points of $\partial\C$ which respectively maximize and minimize the signed distance with the line. A simple geometric observation, coming from the strict convexity and regularity of $\partial\C$, shows the following. The function $\partial\C\ni P\mapsto \nabla\h(OP)\cdot \eta=\partial\h/\partial\eta(OP)$ is $0$ in $P=Q^+$, then it continuously strictly increases when $P$ is moved between $Q^+$ and $R^+$, reaching its maximum at $P=R^+$, then it continuously strictly decreases when $P$ goes from $R^+$ to $Q^-$, and it is $0$ again in $P=Q^-$. Similarly, the function decreases for $P$ between $Q^-$ and $R^-$, where the minimum is reached, and then it increases up to $P=Q^+$.\par
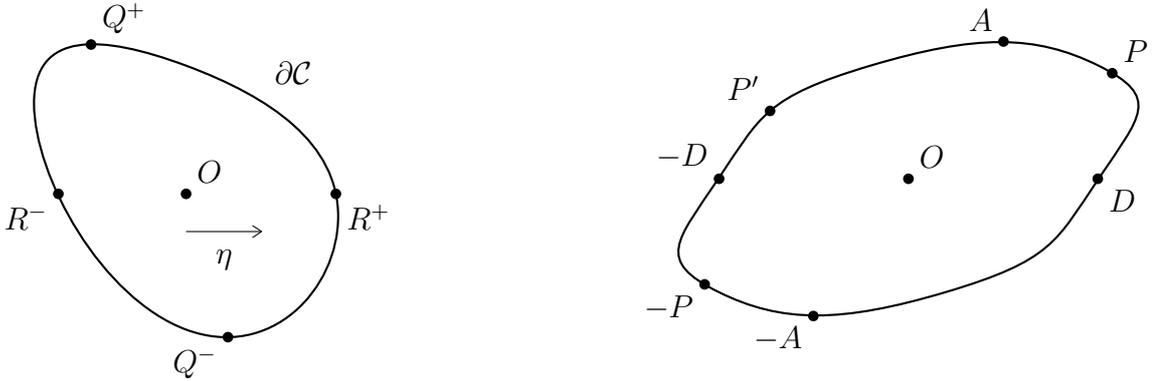
\begin{figure}[htbp]
\begin{tikzpicture}[>=>>>,smooth cycle,scale=.9]
\draw[line width=.8pt] plot [tension=1] coordinates {(-2,1.5) (0,2) (2,0) (0,-1.5)};
\fill (0,0.3) circle(2pt);
\draw (0,0.3) node[anchor=south west] {$O$};
\draw[->] (0,-0.2) -- (1,-0.2);
\draw (0.5,-0.3) node[anchor=north] {$\eta$};
\fill (1.97,0.3) circle(2pt);
\fill (-1.68,0.3) circle (2pt);
\draw (-1.68,0.3) node[anchor=north east] {$R^-$};
\draw (1.97,0.3) node[anchor=north west] {$R^+$};
\fill (-1.25,2.28) circle(2pt);
\fill (.55,-1.6) circle(2pt);
\draw (-1.25,2.28) node[anchor=south west] {$Q^+$};
\draw (.55,-1.6) node[anchor=north east] {$Q^-$};
\draw (1.8,2.2) node[anchor=north east] {$\partial\C$};
\draw[line width=.8pt] plot [tension=1] coordinates {(8,2) (11,2) (11,0.5) (9,-1) (6,-1) (6,0.5)};
\fill (8.5,.5) circle(2pt);
\draw (8.5,.5) node[anchor=south west] {$O$};
\fill (9.75,2.32) circle(2pt);
\draw (9.75,2.32) node[anchor=south east] {$A$};
\fill (7.25,-1.32) circle(2pt);
\draw (7.25,-1.32) node[anchor=north east] {$-A$};
\fill (6.01,.5) circle(2pt);
\draw (6.01,.5) node[anchor=south east] {$-D$};
\fill (10.99,.5) circle(2pt);
\draw (10.99,.5) node[anchor=north west] {$D$};
\fill (11.18,1.9) circle(2pt);
\draw (11.18,1.9) node[anchor=south west] {$P$};
\fill (5.82,-.9) circle(2pt);
\draw (5.82,-.9) node[anchor=north east] {$-P$};
\fill (6.68,1.4) circle(2pt);
\draw (6.68,1.4) node[anchor=south east] {$P'$};
%\draw (6,.5) -- (10,.5);
%\draw (6.02,0) -- (6.02,1);
\end{tikzpicture}
\caption{Left: the function $P\mapsto \nabla(OP)\cdot \eta$. Right: situation when $\h$ is symmetric.}\label{Figsymm}
\end{figure}

Let us now fix a point $A\in \partial\C$, and let us start by considering the symmetric case. We have to show that there exists a unique pair $\{B,C\}$ such that the triple $\{A,B,C\}$ satisfies~(\ref{1storder}). Up to a rotation, we assume that $A$ is the point of $\partial\C$ with biggest second coordinate, so that $\nu_A=(0,1)$. The situation is depicted in Figure~\ref{Figsymm}, right. Keep in mind that, as observed above, $\partial\h/\partial x(OP)$ is positive for points $P$ in the ``right part'' of $\partial\C$, i.e., in the clockwise arc from $A$ to $-A$, and it is negative for points in the ``left part'' of $\partial \C$. Since $\partial \h/\partial x(OA)=0$, this implies that a triple $\{A,\,B,\,C\}$ satisfying~(\ref{1storder}) must necessarily have one between $B$ and $C$ in the ``right part'' of $\partial\C$, and the other one in the left part. Let us now take a point $P$ in the right part of $\partial\C$, and let us ask ourselves whether or not a suitable triple may exist with $B=P$. This happens if and only if there is some $C\in\partial\C$ such that, calling $\eta$ the direction of the vector $OA$,
\begin{align}\label{devsod}
\frac{\partial\h}{\partial x}\, (OC) = -\frac{\partial\h}{\partial x}\, (OP)\,, && \frac{\partial\h}{\partial \eta}\, (OC) +\frac{\partial\h}{\partial \eta}\, (OP) = -\frac{\partial\h}{\partial \eta}\, (OA)\,,
\end{align}
The monotonicity of $\partial\h/\partial x$ observed above, together with the symmetry of $\C$, ensures that there are exactly two points satisfying the first equality. One of them is $-P$, for which the second equality is surely false because $\partial \h/\partial\eta(OC) +\partial \h/\partial\eta(OP)=0$, and the other one is some point $P'$. As shown in the figure, $P$ is above $D=\partial\C \cap \{(x,0):\, x>0\}$ if and only if $P'$ is above $-D$. Summarizing, a suitable pair can only exist with $B=P$ and $C=P'$ for some $P$ in the right part of $\partial\C$, and in particular we have only to take care of the second equality in~(\ref{devsod}) since the first one is true by construction. Notice that, if $P$ continuously ranges from $A$ to $A'$ in the right part of $\partial\C$, then $P'$ continuously ranges from $A$ to $A'$ in the left part, and there is a one-to-one correspondence between $P$ and $P'$. Again recalling the observation above about the monotonicity of $\partial\h/\partial\eta$, we have that the quantity $\partial\h/\partial\eta(OP)$ is strictly decreasing when $P$ moves from $A$ to $A'$, and the same happens for $\partial\h/\partial\eta(OP')$. As a consequence, there can be at most a single point $P$ such that the right equality in~(\ref{devsod}) holds with $C=P'$. And finally, the existence of such a point $P$ is ensured by the continuity, since for $P=A$
\[
\frac{\partial\h}{\partial \eta}\, (OP') +\frac{\partial\h}{\partial \eta}\, (OP) =  2\,\frac{\partial\h}{\partial \eta}\, (OA) > - \,\frac{\partial\h}{\partial \eta}\, (OA)
\]
and for $P=-A$
\[
\frac{\partial\h}{\partial \eta}\, (OP') +\frac{\partial\h}{\partial \eta}\, (OP) =  -2\,\frac{\partial\h}{\partial \eta}\, (OA) < - \,\frac{\partial\h}{\partial \eta}\, (OA)\,.
\]
Now, let us remove the assumption that $\h$ is symmetric, and let us present an example in which no pair $\{B,\,C\}$ exists such that $\{A,\,B,\,C\}$ satisfies~(\ref{1storder}), and another example in which more than a single pair exists.\par
The first example, depicted in Figure~\ref{Fignonsymm}, left, is very simple, it is enough to take as $\C$ a disk with radius $1$ centered at the point $(0,-1/2)$, and call $A=(0,1/2)$ and $A'=(0,-3/2)$. By~(\ref{weknowthat}), for every possible choice of $B,\,C\in \partial\C$ we have
\[
\frac{\partial \h}{\partial y}\, (OA)+\frac{\partial \h}{\partial y}\, (OB)+\frac{\partial \h}{\partial y}\, (OC) \geq \frac{\partial \h}{\partial y}\, (OA) + 2\,\frac{\partial \h}{\partial y}\, (OA') = 2 - \frac 43 > 0\,,
\]
and then $\{A,\,B,\,C\}$ cannot be an admissible triple.\par
\begin{figure}[htbp]
\begin{tikzpicture}[>=>>>,smooth cycle]
\fill (0,0.8) circle(2pt);
\draw (0,0.8) node[anchor=north east] {$O$};
\draw (0,0) circle (1.6);
\fill (0,1.6) circle(2pt);
\draw (0,1.6) node[anchor=south east] {$A$};
\fill (0,-1.6) circle(2pt);
\draw (0,-1.6) node[anchor=north east] {$A'$};
\draw (1.1,1.1) node[anchor=south west] {$\partial\C$};
\fill (8,-0.5) circle(2pt);
\draw (8,-0.8) node[anchor=south west] {$O$};
\fill (8,1.5) circle(2pt);
\draw (8,1.5) node[anchor=south] {$A$};
\draw[line width=1pt] (8,1.5) arc(90:60:2);
\draw[line width=1pt] (8,1.5) arc(90:120:2);
\fill (8.68,1.38) circle(2pt);
\draw (8.68,1.38) node[anchor=south west] {$B$};
\fill (7.32,1.38) circle(2pt);
\draw (7.32,1.38) node[anchor=south east] {$B'$};
\draw (8.68,1.38) -- (8,-0.5) -- (8,1.5);
\draw (8,0.2) arc(90:70:0.7);
\draw (8,.2) node[anchor=east] {$\theta$};
\draw (6.5,-1.24) -- (9.5,-1.76);
\draw (6.5,-1.15) node[anchor=east] {$\tau$};
\draw (9.5,-1.24) -- (6.5,-1.76);
\draw (10.1,-1.15) node[anchor=east] {$\tau'$};
\fill (7,-1.33) circle(2pt);
\draw (7,-1.33) node[anchor=south west] {$C$};
\fill (9,-1.33) circle(2pt);
\draw (9,-1.33) node[anchor=south east] {$C'$};
\end{tikzpicture}
\caption{Left: an example with no admissible triple containing $A$. Right: an example with multiple admissible triples containing $A$.}\label{Fignonsymm}
\end{figure}
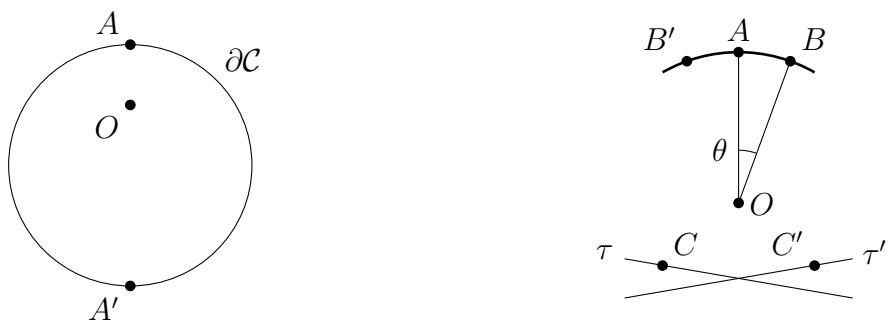
Let us now present an example with more than a single admissible triple containing a given point $A$. As shown in Figure~\ref{Fignonsymm}, right, we define $A=(0,1)$ and we let $\partial\C$ coincide with the circle $\partial B(0,1)$ for a short while around $A$. We let also $B=(\sin\theta,\cos\theta)$ for a small $\theta>0$. We have then, by construction and recalling~(\ref{weknowthat}),
\[
\nabla \h(OA)+\nabla\h(OB)=(0,1) + (\sin\theta,\cos\theta)\,,
\]
and then a point $C\in\partial\C$ completes a suitable triple together with $A$ and $B$ if and only $\nabla\h(OC)=(-\sin\theta,-1-\cos\theta)$. Again by~(\ref{weknowthat}), this is equivalent to say that the tangent line to $\partial\C$ at $C$ is the line $\tau$ whose direction is orthogonal to the vector $(\sin\theta,1+\cos\theta)$, and whose signed distance from the origin is $-\big|(\sin\theta,1+\cos\theta)\big|^{-1}$.\par
We can then fix a point $C\in \tau$, with first coordinate slightly negative. If $C\in\partial\C$ and $\tau$ is the tangent line to $\partial\C$ at $C$ then $\{A,\,B,\,C\}$ is an admissible triple. Let us then call $\tau'$ (resp., $B',\,C'$) the line obtained from $\tau$ (resp., the point obtained from $B,\,C$) with a symmetry with respect to the vertical axis $\{x=0\}$. Therefore, as soon as $\partial\C$ contains the small arc of circle around $A$, and the points $C$ and $C'$ with tangent lines $\tau$ and $\tau'$ (and this is obviously possible with a non-symmetric unit ball $\C$), then both $\{A,\,B,\,C\}$ and $\{A,\,B',\,C'\}$ are admissible triples, so the uniqueness does not hold.
\end{proof}

Let us briefly describe an explicit example of a unit ball and of a corresponding admissible triple.
\begin{example}
For $p>1$, let us consider the norm $\h$ corresponding to the unit ball $\C=\{(x,y)\in\R^2:\, |x|^p +|y|^p\leq 1\}$. Observe that this norm is symmetric. Let us now consider $A=(0,1)\in\partial\C$. Then, a boring but elementary calculation ensures that the unique pair $\{B,\,C\}\in\partial\C$ such that $\{A,\,B,\,C\}$ satisfies~(\ref{1storder}) is given by the two points for which $\angle AOB= 2\pi-\angle AOC= \alpha$, being $\tan\alpha= -(2^p-1)^{1/p}$. Notice that of course, for $p=2$, this reduces to the well-known $120^\circ$ rule.
\end{example}

To conclude, we can ``translate'' the property~(\ref{1storder}) to triple points of optimal clusters. As already noticed several times, the study of the perimeter coincides with the study of the minimal length of curves, except that we have to rotate the normal vectors so to obtain the normal ones. And moreover, depending on the colors of the regions, the rotated function might have to be simmetrized. Precisely, we can prove the following result.
\begin{prop}
Let $h$ satisfy the assumptions of Theorem~\mref{main}, and let $O$ be a triple point of an optimal cluster $\E$. Call $\theta_1,\,\theta_2,\,\theta_3\in\S^1$, ordered in clockwise sense, the directions of the three arcs of $\partial\E$ meeting at $O$. For every $\nu\in\S^1$, let us call $\hat\nu$ the direction obtained by rotating $\nu$ of $90^\circ$ in the clockwise sense, and let $\h:\R^2\to\R$ be given by $\h(\nu)=h(O,\hat\nu)$. Then, the minimality property $\Theta=0$ holds, where $\Theta\in\R^2$ is defined as follows. If the three regions meeting at $O$ are all colored, then
\begin{equation}\label{321}
\Theta = \frac{\nabla\h(\theta_1)-\nabla\h(-\theta_1)}2+\frac{\nabla\h(\theta_2)-\nabla\h(-\theta_2)}2+\frac{\nabla\h(\theta_3)-\nabla\h(-\theta_3)}2\,.
\end{equation}
If the region between $\theta_1$ and $\theta_2$ is white, then
\begin{equation}\label{322}
\Theta = \nabla\h(\theta_1)-\nabla\h(-\theta_2)+\frac{\nabla\h(\theta_3)-\nabla\h(-\theta_3)}2\,.
\end{equation}
Notice that, if $h$ is symmetric, then both the above definitions simply reduce to
\[
\Theta = \nabla\h(\theta_1)+\nabla\h(\theta_2)+\nabla\h(\theta_3)\,.
\]
\end{prop}
\begin{proof}
Let us call $A,\,B,\,C$ the points in $\partial B(O,1)$ in the directions $\theta_1,\,\theta_2$ and $\theta_3$. For every point $D\in B(O,1)$, let us define $L(D)$ as
\[
L(D)= \frac{\h(DA) + \h(AD)}2 +\frac{\h(DB) + \h(BD)}2+\frac{\h(DC) + \h(CD)}2
\]
if the three regions meeting at $O$ are all colored, while
\[
L(D)=\h(DA) + \h(BD)+\frac{\h(DC) + \h(CD)}2
\]
if the regions between $\theta_1$ and $\theta_2$ is white. Notice that, thanks to the definitions~(\ref{321}) and~(\ref{322}), if $|D|=\eps\ll 1$ then
\[
L(D) = L(O) - \Theta\cdot D + o(\eps)\,.
\]
As a consequence, if $\Theta\neq 0$ there are a constant $c>0$ and a point $D\in B(0,1)$ such that
\begin{equation}\label{tbr}
L(D) = L(O) - c\,.
\end{equation}
Let now $r\ll 1$ be a small constant. Keep in mind that $O$ is a triple point, and the arcs meeting at $O$ correspond to the directions $\theta_1,\,\theta_2,\,\theta_3$. Therefore, as soon as $r$ is small enough, $\partial\E\cap \partial B(O,r)$ consists of three points $A',\,B'$ and $C'$, and the directions of $OA',\, OB'$ and $OC'$ are arbitrarily close to $\theta_1,\,\theta_2$ and $\theta_3$. Let us also define $P'$ as the perimeter obtained by using~(\ref{weightedvolper}) with $h'$ in place of $h$, where $h'$ is defined as $h'(x,\nu)=h(x,\nu)$ if $x\notin B(O,r)$, and $h'(x,\nu)=h(O,\nu)$ if $x\in B(O,r)$. In addition, for every point $Q\in B(O,r)$ we call $\E'_Q$ the cluster which equals $\E$ outside the ball $B(O,r)$, and such that $\partial\E'\cap B(O,r)$ is given by the three segments $QA,\,QB$ and $QC$. In particular, we call $\E'=\E'_{rD}$. By~(\ref{tbr}) and rescaling, and also using Lemma~\ref{lemmagiorgio}, we can estimate
\begin{equation}\label{tbs}
P'(\E') = P'(\E'_O) - cr \leq P'(\E) - cr\,.
\end{equation}
We can then conclude by finding a contradiction with the same argument used several times in Section~\ref{sec:proof}. Namely, if $r\ll 1$ we have that
\begin{align*}
\big| P\big(\E;B(O,r)\big) - P'\big(\E;B(O,r)\big) \big| \ll r\,, &&
\big| P\big(\E';B(O,r)\big) - P'\big(\E';B(O,r)\big) \big| \ll r\,,
\end{align*}
hence for $r$ small enough~(\ref{tbs}) gives
\[
P(\E') \leq P(\E) - \frac c2\, r\,.
\]
And finally, this estimate together with Lemma~\ref{labello} allows to find a competitor $\E''$ with $|\E''|=|\E|$ and $P(\E'')< P(\E)$, which is the searched contradiction.
\end{proof}

\appendix

\section{Some properties about quasi-minimal sets and porous sets}

In this short appendix, we present some known results concerning quasi-minimal and porous sets, and their boundaries. We will not need to deal with densities, so we will only use the standard Euclidean volume $\Ve{\cdot}$ and perimeter $\Pe$. First of all, we recall a couple of important standard definitions, see for instance~\cite{Tam,DavSem,KKLS} (we actually deal only with the case of subsets of $\R^N$, while~\cite{KKLS} considers more general metric spaces with doubling measures).

\begin{defn}[Quasi-minimal sets, porous sets]
Let $F\subseteq\R^N$ be a Borel set with locally finite perimeter. We say that \emph{$F$ is quasi-minimal} if there exists a constant $C_{qm}$ such that, for every ball $B(x,r)\subseteq\R^N$ and every set $H\subseteq\R^N$ with $H\Delta F\comp B(x,r)$, one has
\[
\Pe\big(F;B(x,r)\big) \leq C_{qm} \Pe\big(H;B(x,r)\big)\,.
\]
We say that \emph{$F$ is porous} if there exists $\delta>0$ such that, for every $x\in \partial F$ (the topological boundary) and every small ball $B(x,r)$, there exist a ball $B(y,\delta r)\subseteq B(x,r)\cap F$ and a ball $B(z,\delta r)\subseteq B(x,r)\setminus F$.
\end{defn}

The following result is well-known, see for instance~\cite[Theorem~1.8]{DavSem} or~\cite[Theorem~5.2]{KKLS}.
\begin{thm}\label{qmpor}
Every quasi-minimal set $F\subseteq\R^N$ admits a porous representative.
\end{thm}

The convenience of the notion of porosity is mainly given by the following standard fact, that we prove just for completeness.

\begin{lem}\label{porope}
Let $F\subseteq\R^N$ be a porous set. Then the set $F^{(1)}$ of the points of density $1$ of $F$ is open. Moreover, the reduced boundary $\partial^* F$ and the topological boundary $\partial F$ coincide up to $\haus^{N-1}$-negligible subsets.
\end{lem}
\begin{proof}
The inclusion $\partial^* F \subseteq \partial F$ is always satisfied. Let now $x$ be any point in $\partial F$. By the definition of porosity, the density of $F$ at $x$ is between $\delta^N$ and $1-\delta^N$, so that $x\notin F^{(0)}\cup F^{(1)}$. Since $F^{(0)}\cup F^{(1)}$ fill the whole $\R^N\setminus\partial^* F$ up to zero $\haus^{N-1}$-measure, we deduce that $\haus^{N-1}(\partial F\setminus \partial^* F)=0$. In particular, we have observed that a point of $F^{(1)}$ cannot belong to $\partial F$, hence it must be either in the interior of $F$, or in the interior of $\R^N\setminus F$. Since the latter possibility is excluded by the positive density, we deduce that $F^{(1)}$ is open.
\end{proof}

We conclude with a $2$-dimensional property of porous sets without holes, that we formally define below. Also this property is not new, but we give a simple proof for completeness.

\begin{defn}[Holes]\label{defholes}
Let $F\subseteq\R^2$ be a Borel set with locally finite perimeter. We say that \emph{$F$ has a hole $U$} if there exists a bounded set $U\subseteq \R^2\setminus F$ with $\haus^2(U)>0$ and such that, up to $\haus^1$-negligible sets,
\[
\partial^* F = \partial^* U \cup \partial^* (F\cup U)\,.
\]
\end{defn}

In order to state the next result, we recall that a set $F$ is said \emph{connected in the measure theoretical sense} if, whenever one writes $F=F'\cup F''$ with two essentially disjoint sets $F',\, F''$ so that, up to $\haus^1$-negligible subsets, $\partial^* F = \partial^* F' \cup \partial^* F''$, it must be $\min \{ \haus^2(F'),\,\haus^2(F'')\}=0$.

\begin{lem}\label{noislPCC}
Let $F\subseteq\R^2$ be an open, porous set of finite (Euclidean) perimeter, connected in the measure theoretical sense and without holes. Then, $\partial F$ is a closed curve. More precisely, there exists an injective curve $\gamma:\S^1\to\R^2$ of finite length such that $\partial F=\gamma(\S^1)$. As a consequence, $F$ is connected also in the topological sense.
\end{lem}
\begin{proof}
Since $F$ has finite perimeter, by the compactness results in $BV$ we have a sequence of smooth sets $F_j$ such that
\begin{align}\label{BVstrict}
\Ve{F_j\Delta F} \to 0\,, && \Pe(F_j) \to \Pe(F)\,.
\end{align}

\step{I}{Reduction to the case of connected sets $F_j$.}
First of all, we want to reduce ourselves to the case when the sets $F_j$ are connected. Since $F_j$ is regular, we can write it as $F_j^1\cup F_j^2$, where $F_j^1$ is the connected component with biggest area, and $F_j^2$ is the union of all the other connected components. Observe that, by the isoperimetric inequality,
\[
\Pe(F_j) \geq \frac{2 \sqrt{\pi}}{\sqrt{\Ve{F_j^1}}}\, \Ve{F_j}\,,
\]
and since $\Ve{F_j}\to \Ve{F}$ and $\Pe(F_j)\to \Pe(F)$ we deduce that $\Ve{F_j^1}$ is bounded away from $0$. Up to a subsequence, we can assume that the characteristic functions of $F_j^1$ and $F_j^2$ weakly converge in $BV$ to the characteristic functions of two sets, that we call $F^1$ and $F^2$. Notice that $F^1\cap F^2=\emptyset$ and $F^1\cup F^2=F$. By the lower semicontinuity of the perimeter under weak $BV$ convergence and~(\ref{BVstrict}), we have
\[\begin{split}
\Pe(F) &\leq \Pe(F^1)+\Pe(F^2) \leq \liminf \Pe(F_j^1) + \liminf \Pe(F_j^2)\\
&\leq \liminf \big(\Pe(F_j^1)+\Pe(F_j^2)\big) = \liminf \Pe(F_j) = \Pe(F)\,,
\end{split}\]
and since $F$ is connected in the measure theoretical sense the first inequality is strict unless one of the two sets is negligible. Since the strict inequality is impossible and $F^1$ is not negligible by construction, we deduce that $\Ve{F^2}=\lim \Ve{F^2_j}=0$, and as a byproduct the above chain of inequalities implies also that $\Pe(F^2_j)\to 0$. Therefore, (\ref{BVstrict}) still holds true replacing the sets $F_j$ with the connected sets $F_j^1$, and this concludes the step.

\step{II}{Reduction to the case of sets $F_j$ with $\partial F_j$ connected.}
We want now to reduce ourselves to the case when the sets $F_j$ have connected boundaries (hence, they have no holes). Since $F_j$ is a smooth, connected set, it is possible to write it as $F_j=G_j\setminus U_j$, where $G_j$ has smooth, connected boundary, and $U_j\comp G_j$. In particular, $\Pe(F_j)=\Pe(G_j)+\Pe(U_j)$. Up to a subsequence, we can assume that the characteristic functions of $G_j$ and of $U_j$ weakly converge in $BV$ to the characteristic functions of two sets, that we call $G$ and $U$. Notice that $U\subseteq G$ and that $F=G\setminus U$, hence by lower semicontinuity of the perimeter we have
\[\begin{split}
\Pe(F)&=\lim \Pe(F_j) = \lim \big(\Pe(G_j) + \Pe(U_j)\big)\\
&\geq \liminf \Pe(G_j) + \liminf \Pe(U_j)
\geq \Pe(G) + \Pe(U)\\
&= \Pe(F\cup U) + \Pe(U)\,.
\end{split}\]
Since $F$ has no holes, we deduce that $U$ is negligible, and that $\Pe(U_j)\to 0$. As a consequence, we can replace the sets $F_j$ with the sets $G_j$ and~(\ref{BVstrict}) still holds true, which concludes also this step.

\step{III}{Conclusion.}
By steps~I and~II, we have a sequence of sets $F_j$ satisfying~(\ref{BVstrict}) and having a closed, regular curve as boundary. There are then smooth functions $\gamma_j:\S^1\to\R^2$, injective and with $|\gamma_j'|$ constant (thus constantly equal to $\Pe(F_j)/2\pi$). Up to subsequences, the functions $\gamma_j$ uniformly converge to a Lipschitz function $\gamma:\S^1\to\R^2$. By construction $\partial F\subseteq \gamma(\S^1)$, thus
\[
\Pe(F) = \haus^1(\partial F)\leq \haus^1(\gamma) \leq \liminf\haus^1(\gamma_j) =\liminf \Pe(F_j) = \Pe(F)\,.
\]
We deduce that $\partial F=\gamma(\S^1)$, and the curve $\gamma$ is injective since $F$ is connected.
\end{proof}

\section*{Acknowledgments}
The authors acknowledge the support of the GNAMPA--INdAM projects \emph{Problemi isoperimetrici in spazi Euclidei e non} (n.\ UUFMBAZ-2019-000473 11-03-2019), \textit{Problemi isoperimetrici con anisotropie} (n.\ U-UFMBAZ-2020-000798 15-04-2020) and \textit{Analisi geometrica in strutture subriemanniane} (n. CUP\_E55F2200\-0270001).

The first author also acknowledges the support received from the European Union's Horizon 2020 research and innovation programme under the \textit{Marie Sk\l o\-dowska Curie grant n.\ 794592} and from the ANR-15-CE40-0018 project \textit{SRGI -- Sub-Riemannian Geometry and Interactions}.
The third author also acknowledges the support of the \textit{ERC Starting
Grant 676675 FLIRT – Fluid Flows and Irregular Transport} and has received funding from the European Research Council (ERC) under the European Union’s Horizon 2020 research and innovation program (grant agreement No.~945655).

\end{document}